\documentclass{amsart}

\usepackage{graphics,graphicx}
\usepackage{tikz-cd}
\usetikzlibrary{calc}
\usetikzlibrary{arrows} 

\usepackage[utf8]{inputenc}
\usepackage[T1]{fontenc}

\usepackage{amsmath}
\usepackage{amsthm}
\usepackage{amssymb}

\usepackage{xfrac}

\newtheorem{thm}{Theorem}[section]
\newtheorem{cor}[thm]{Corollary}
\newtheorem{lem}[thm]{Lemma}
\newtheorem{prop}[thm]{Proposition}

\newtheorem*{thma}{Theorem A}
\newtheorem*{thmb}{Theorem B}
\newtheorem*{thmc}{Theorem C}
\newtheorem*{thmd}{Theorem D}
\newtheorem*{thme}{Theorem E}

\theoremstyle{definition}
\newtheorem{defn}[thm]{Definition}
\newtheorem{rem}[thm]{Remark}
\newtheorem{exa}[thm]{Example}

\begin{document}

\title[Horosymmetric varieties]{Kähler geometry of horosymmetric varieties, 
and application to Mabuchi's K-energy functional}

\author{Thibaut Delcroix}

\date{}

\begin{abstract}
We introduce a class of almost homogeneous varieties contained in the class of spherical 
varieties and containing horospherical varieties as well as complete symmetric 
varieties. We develop Kähler geometry on these varieties, with applications to 
canonical metrics in mind, as a generalization of the Guillemin-Abreu-Donaldson 
geometry of toric varieties. Namely we associate convex functions with hermitian 
metrics on line bundles, and express the curvature form in terms of this function,
as well as the corresponding Monge-Ampère volume form and scalar curvature.  
We then provide an expression for the Mabuchi functional and derive as an 
application a combinatorial sufficient condition of properness similar to 
one obtained by Li, Zhou and Zhu on 
group compactifications.
\end{abstract}

\maketitle

\section*{Introduction}

Toric manifolds are complex manifolds equipped with an action of $(\mathbb{C}^*)^n$ 
such that there is a point with dense orbit and trivial stabilizer. 
The Kähler geometry of toric manifolds plays a fundamental role in Kähler geometry 
as a major source of examples as well as a testing ground for conjectures. It 
involves strong interactions with domains as various as convex analysis, 
real Monge-Ampère equations, combinatorics of polytopes, 
algebraic geometry of toric varieties, etc.
The study of Kähler metrics on toric manifolds relies strongly on works of 
Guillemin, Abreu, then Donaldson. They have developped, using Legendre transform 
as a main tool, a very precise setting including: 
\begin{itemize}
\item a model behavior for smooth Kähler metrics,
\item a powerful expression of the scalar curvature, 
\item applications to the study of canonical Kähler metrics \emph{via} the 
ubiquitous Mabuchi functional.
\end{itemize}
This setting allowed Donaldson to prove the Yau-Tian-Donaldson conjecture for 
constant scalar curvature Kähler (cscK) metrics on toric surfaces. That is, he showed that 
existence of cscK metrics in a given Kähler class corresponding to a polarization 
of a toric surface $X$ by an ample line bundle $L$ is equivalent to K-stability 
of $(X,L)$ with respect to test configurations that are equivariant with respect 
to the torus action and further translated this condition into a condition on the 
associated polytope.

Our goal in this article is to generalize this setting to a much larger class of 
varieties, that we introduce: the class of \emph{horosymmetric varieties}. 

Let $G$ be a complex connected linear reductive group. A normal algebraic $G$-variety 
$X$ is called spherical if any Borel subgroup $B$ of $G$ acts with an open 
dense orbit on $X$. Major subclasses of spherical varieties are given by 
biequivariant group compactifications, and horospherical varieties. A horospherical 
variety is a $G$-variety with an open dense orbit which is a $G$-homogeneous fibration 
over a generalized flag manifold with fiber a torus $(\mathbb{C}^*)^r$. 
The author's previous work on spherical varieties has highlighted how 
they provide a richer source of examples than toric varieties, with several 
examples of behavior which cannot appear for toric varieties. 
While it was possible to work on the full class of spherical manifolds, 
from the point of view of algebraic geometry, for 
the question of the existence of Kähler-Einstein metrics thanks to the 
proof of the Yau-Tian-Donaldson conjecture for Fano manifolds, it is necessary 
to develop Guillemin-Abreu-Donaldson theory to treat more general questions.
It seems a very challenging problem to do this uniformly for all spherical 
varieties. 
On the other hand, the author did develop part of this setting for group 
compactifications and horospherical varieties. 

Group compactifications do not share a nice property that toric, horospherical 
and spherical varieties possess and frequently used in Kähler geometry: 
a codimension one invariant irreducible 
subvariety in a group compactification leaves the class of group compactifications.
We introduce the class of horosymmetric varieties as a natural subclass of spherical 
varieties containing horospherical varieties, group compactifications and more 
generally equivariant compactifications of symmetric spaces, which possesses the 
property of being closed under taking a codimension one invariant irreducible 
subvariety. 
The definition is modelled on the description of orbits of wonderful compactifications 
of adjoint (complex) symmetric spaces by De Concini and Procesi: they all have a dense orbit 
which is a homogeneous fibration over a generalized flag manifold, whose fibers 
are symmetric spaces. 
We say a normal $G$-variety is \emph{horosymmetric} if it admits an open dense orbit 
which is a homogeneous fibration over a generalized flag manifold, whose fibers 
are symmetric spaces. 
Such a homogeneous space is sometimes called a parabolic induction 
from a symmetric space.
Here we allow symmetric spaces under reductive groups, thus 
recovering horospherical varieties by considering $(\mathbb{C}^*)^r$ as 
a symmetric space for the group $(\mathbb{C}^*)^r$ and the involution 
$\sigma(t)=t^{-1}$. 
For the sake of giving precise statement in this introduction, let us introduce some 
notations. A horosymmetric homogeneous space is a homogeneous space $G/H$ such that 
there exists 
\begin{itemize}
\item a parabolic subgroup $P$ of $G$, with unipotent radical $P^u$, 
\item a Levi subgroup $L$ of $P$,
\item and an involution of complex groups $\sigma:L\rightarrow L$, 
\end{itemize}
such that $P^u\subset H$ and $(L^{\sigma})^0\subset L\cap H$ as a finite index subgroup, 
where $L^{\sigma}$ denotes the subgroup of elements fixed by $\sigma$ 
and $(L^{\sigma})^0$ its neutral connected component. 

Spherical varieties in general admit a combinatorial description: there is on 
one hand a complete combinatorial caracterization of spherical homogeneous spaces by Losev \cite{Los09}
and on the other hand a combinatorial classification of embeddings of 
a given spherical homogeneous space by Luna and Vust \cite{LV83,Kno91}. 
General results about parabolic induction allow to derive easily the 
information about a horosymmetric homogeneous spaces from the information 
about the symmetric space fiber. For symmetric spaces, most of the information 
is contained in the restricted root system. Choose a torus $T_s\subset L$ on 
which the involution acts as the inverse, and maximal for this property. It 
is contained in a maximal $\sigma$-stable torus $T$ in $L$.  
Then consider the restriction of 
roots of $G$ (with respect to $T$) to $T_s$. Let $\Phi_s$ denote the subset 
of roots whose restriction are not identically zero. 
The restrictions of the roots in $\Phi_s$ form a (possibly non reduced) root system called 
the restricted root system, 
with corresponding notions of restricted Weyl group $\bar{W}$, restricted Weyl chambers, etc.
Let $\mathfrak{Y}(T_s)$ denote the group of one-parameter subgroups of $T_s$, and 
identify $\mathfrak{a}_s=\mathfrak{Y}(T_s)\otimes \mathbb{R}$ with the Lie algebra of the 
non-compact part of the torus $T_s$.
The image of $\mathfrak{a}_s$ by the exponential, then by the action 
on the base point $x\in X$, intersects every orbit of a maximal compact subgroup 
$K$ of $G$ along restricted Weyl group orbits (see Section~\ref{sec_homogeneous} for details). 

Let $\mathcal{L}$ be a $G$-linearized line bundle on a horosymmetric homogeneous space. 
It is determined by its isotropy character $\chi$.  
To a $K$-invariant metric $h$ on $\mathcal{L}$ we associate a function 
$u:\mathfrak{a}_s\rightarrow \mathbb{R}$, called the 
\emph{toric potential}, which together with $\chi$ totally encodes the metric. 
One of our main result is the derivation of an expression of the curvature form 
$\omega$ of $h$ in terms of its toric potential. 
We achieve this for the general case, but let us only give the statement in the 
nicer situation where the restriction of $\mathcal{L}$ to the symmetric space 
fiber is trivial. 
By fixing a choice of basis of a complement of $\mathfrak{h}$ in $\mathfrak{g}$, 
we may define reference real (1,1)-forms $\omega_{\heartsuit,\diamondsuit}$ indexed by 
couples of indices in $\{1,\ldots,r\}\cup\Phi_s^+\cup\Phi_{Q^u}$ (where $r=\mathrm{dim}(T_s)$, $\Phi_s^+$ is 
the intersection of $\Phi_s$ with some system of positive roots, $\Phi_{Q^u}=-\Phi_{P^u}$
is the set of opposite of roots of $P^u$) that form a pointwise basis, 
and we express the curvature form in these terms. Given a root $\alpha$ of $G$, 
we denote by $\alpha^{\vee}$ its associated coroot. 

\begin{thma}
Assume that the restriction of $\mathcal{L}$ to the symmetric 
fiber is trivial. Let $a\in \mathfrak{a}_s$ be such that $\beta(a)\neq 0$ for all 
$\beta \in \Phi_s$. Then 
\begin{align*} 
\omega_{\exp(a)H} = & \sum_{1\leq j_1,j_2\leq r} 
\frac{1}{4}d^2_au(l_{j_1},l_{j_2}) 
\omega_{j_1,j_2}  
+ \sum_{\alpha\in \Phi_{Q^u}} 
\frac{-e^{2\alpha}}{2}(d_au-2\chi)(\alpha^{\vee})
\omega_{\alpha,\alpha} \\
& + \sum_{\beta \in \Phi_s^+} 
\frac{d_au(\beta^{\vee})}{\sinh(2\beta(a))} 
\omega_{\beta,\beta}.
\end{align*}
\end{thma}

The previous theorem concerns only horosymmetric homogeneous spaces. 
To move on to horosymmetric varieties, we need more input from the 
general theory of spherical varieties. To a $G$-linearized line bundle 
on a horosymmetric variety are associated several polytopes. 
Of major importance is the (algebraic) moment polytope $\Delta^+$, obtained as the closure of the 
set of suitably normalized highest weights of the spaces of multisections 
of $\mathcal{L}$, seen as $G$-representations. It lies in the real 
vector space $\mathfrak{X}(T)\otimes \mathbb{R}$, where $\mathfrak{X}(T)$
denotes the group of characters of $T$.
The main application of the moment polytope to Kähler geometry is that it 
controls the asymptotic behavior of toric potentials, which in the case 
of positively curved metrics further allows to give a formula for 
integration with respect to the Monge-Ampère of the curvature from, 
in conjunction with the previous theorem. 
Again we do not state 
our results in their most general form in this introduction but 
prefer to give the general 
philosophy in a situation which is simpler than the general one. 

\begin{thmb}
Assume that $\mathcal{L}$ is an ample $G$-linearized line bundle on a 
non-singular horosymmetric variety $X$. Assume furthermore that it 
admits a global $Q$-semi-invariant holomorphic section, where $Q$ is the parabolic 
opposite to $P$ with respect to $T$. Let $h$ be a smooth $K$-invariant 
metric on $\mathcal{L}$ with positive curvature $\omega$ and toric potential $u$. 
Then 
\begin{enumerate}
\item $u$ is smooth, $\bar{W}$-invariant and strictly convex, 
\item $a\mapsto d_au$ defines a diffeomorphism from $\mathrm{Int}(-\mathfrak{a}_s^+)$ onto 
$\mathrm{Int}(2\chi-2\Delta^+)$. 
\end{enumerate}
Let $\psi$ denote a $K$-invariant function on $X$, integrable with 
respect to $\omega^n$. 
Let $dq$ denote the Lebesgue measure on the affine span of $\Delta^+$,
normalized by the lattice $\chi+\mathfrak{X}(T/T\cap H)$. 
Then there exist a constant $C$, independent of $h$ and $\psi$, such that 
\[
\int_X\psi\omega^n  = 
C\int_{\Delta^+} \psi(d_{2\chi-2q}u^*)
P_{DH}(q) dq. 
\]
where 
$P_{DH}(q)=\prod_{\alpha\in\Phi_{Q^u}\cup \Phi_s^+} \kappa(\alpha,q)/\kappa(\alpha,\rho)$ ,
$\rho$ is the half sum of positive roots of $G$ and $u^*$ is the convex 
conjugate of $u$.
\end{thmb} 

The two theorems above form a strong basis to attack Kähler geometry 
questions on horosymmetric varieties. They are for example all that is needed 
to study the existence of Fano Kähler-Einstein metrics with the strategy 
following the lines of Wang and Zhu's work on toric Fano manifolds. 
They also allow to push further, namely to compute an expression of the 
scalar curvature, to compute an expression of the (log)-Mabuchi functional, 
then to obtain a coercivity criterion for this functional, in the line 
of work of Li-Zhou-Zhu for group compactifications. 

The Mabuchi functional 
is a functional on the space of Kähler metrics in a given class, 
whose smooth minimizers if they exist should be cscK metrics. 
There are several extensions of this notion, in particular 
the log-Mabuchi functional, related to the existence 
of log-Kähler-Einstein metrics. 
A natural way to search for minimizers of this functional is to 
try to prove its properness, or coercivity, with respect to the 
$J$-functional. The $J$-functional is another standard functional in Kähler 
geometry, which may be considered as a measure of distance from a 
fixed reference metric in the space of Kähler metrics.

We provide in this paper an application of our setting to this problem 
of coercivity of the Mabuchi functional, obtaining a very general, 
but at the same time far from optimal, coercivity criterion for the 
Mabuchi functional on horosymmetric varieties. 
We work under several simplifying assumptions to carry out the proof 
while keeping a reasonnable length for the paper, but expect that 
several of these assumptions can be removed with a little work. 

Instead of stating these assumptions in this introduction, let us 
state the result in three examples of situations where they are satisfied.  
They are as follows, in all cases, $G$ is a complex connected linear 
reductive group and $X$ is a smooth projective $G$-variety.
\begin{enumerate}
\item The manifold $X$ is a group compactification, that is, $G=G_0\times G_0$ 
and there exists a point $x\in X$ with stabilizer $\mathrm{diag}(G_0)\subset G$ 
and dense orbit. We may consider any ample $G$-linearized line bundle on $X$. 
\item The manifold $X$ is a homogeneous toric bundle under the action of $G$, 
that is there exists 
a projective homogeneous $G/P$ and a $G$-equivariant surjective morphism 
$X\rightarrow G/P$ with fiber isomorphic to a toric variety, under the action 
of $P$ which factorizes through a torus $(\mathbb{C}^*)^r$. 
We consider any ample $G$-linearized 
line bundles on $X$.
\item The manifold $X$ is a toroidal symmetric variety of type AIII($r$, $m>2r$), that 
is, $m$ and $r$ are two positive integers with $m>2r$, $G=\mathrm{SL}_m$, 
there exists a point $x\in X$ with dense orbit, whose orbit is isomorphic 
to $\mathrm{SL}_m/\mathrm{S}(\mathrm{GL}_r\times \mathrm{GL}_{m-r})$, and 
there exists a dominant $G$-equivariant morphism from $X$ to the wonderful 
compactification of this symmetric space. We may consider any ample 
$G$-linearized line bundle 
which restricts to a trivial line bundle on the dense orbit. 
\end{enumerate}

Let $\Theta$ be a $G$-equivariant boundary, that is, an \emph{effective} 
$\mathbb{Q}$-divisor
$\Theta=\sum_Y c_yY$ where $Y$ runs over all $G$-stable irreducible 
codimension one submanifolds of $X$. We assume furthermore 
that the support of $\Theta$ 
is simple normal crossing and $c_Y<1$ for all $Y$. In particular, 
the pair $(X,\Theta)$ is klt. 
It follows from the combinatorial description of horosymmetric varieties 
that to each $Y$ as above is associated an element $\mu_Y$ of 
$\mathfrak{Y}(T_s)$. 
Let $\Delta^+$ be the moment polytope of $\mathcal{L}$, and let 
$\lambda_0$ be a well chosen point in $\Delta^+$ (see Section~\ref{sec_mabuchi}). 
Let $\tilde{\Delta}_Y^+$ denote the bounded cone
with vertex $\lambda_0$ and base the face of $\Delta^+$ whose outer normal 
is $-\mu_Y$ in the affine space $\chi+\mathfrak{X}(T/T\cap H)\otimes \mathbb{R}$.
Let $\chi^{ac}$ denote the restriction of the character 
$\sum_{\alpha\in \Phi_{Q^u}}\alpha$ of $P$ to $H$, 
and set 
\[
\Lambda_Y=\frac{-c_Y+1-\sum_{\alpha\in\Phi_{Q^u}\cup \Phi_s^+}\alpha(\mu_Y)}{
\sup\{p(\mu_Y),p\in\chi-\Delta^+\}},
\] 
\[
I_H(a)=\sum_{\beta\in \Phi_s^+}\ln \sinh(-2\beta(a))
-\sum_{\alpha\in\Phi_{Q^u}}2\alpha(a),
\] 
and 
\[
\bar{S}_{\Theta}=\bar{S}-n\sum_Yc_Y\mathcal{L}|_Y^{n-1}/\mathcal{L}^n.
\]
Let $\mathrm{Mab}_{\Theta}$ denote the log-Mabuchi functional in 
this setting, and write $\mathrm{Mab}_{\Theta}(u)$ for its value on 
the hermitian metric with toric potential $u$. For the question 
of coercivity, the log-Mabuchi functional matters only up to 
normalizing additive and multiplicative constants, so we ignore 
these in the statement.

\begin{thmc}
Let $p=p(q):=2(\chi-q)$, then we have  
\begin{align*}
\mathrm{Mab}_{\Theta}(u) = &
\sum_Y \Lambda_Y 
\int_{\tilde{\Delta}^+_Y} (nu^*(p)-u^*(p)\sum \frac{\chi(\alpha^{\vee})}{q(\alpha^{\vee})} +d_pu^*(p))P_{DH}(q)dq \\
& 
+\int_{\Delta^+} u^*(p)(\sum \frac{\chi^{ac}(\alpha^{\vee})}{q(\alpha^{\vee})}-\bar{S}_{\Theta})P_{DH}(q)dq 
-\int_{\Delta^+}I_H(d_pu^*)P_{DH}(q)dq \\
&
-\int_{\Delta^+}\ln\det (d^2_pu^*)P_{DH}(q)dq 
\end{align*}
where $u^*$ denotes the Legendre transform, or convex conjugate, of $u$.
\end{thmc}

As an application of this formula, we obtain the following sufficient condition for 
coercivity.
Consider the function $F_{\mathcal{L}}$ defined piecewise by 
\[
F_{\mathcal{L}}(q)=(n+1)\Lambda_Y-\bar{S}_{\Theta}
+\sum\frac{(\chi^{ac}-\Lambda_Y\chi)(\alpha^{\vee})}{q(\alpha^{\vee})}
\]
for $q$ in $\tilde\Delta_Y^+$. 
Define an element $\mathrm{bar}$ of the affine space generated by 
$\Delta^+$ by setting 
\[
\mathrm{bar}=
\int_{\Delta^+} qF_{\mathcal{L}}(q)\frac{P_{DH}(q)dq}{\int_{\Delta^+}P_{DH}dq}.
\]
Despite the notation, it is not in general the barycenter of $\Delta^+$ 
with respect to a positive measure. We will however see how to consider 
it as a barycenter in the article. 
The coercivity criterion is stated in terms of $F_{\mathcal{L}}$ and 
$\mathrm{bar}$. 
Let $2\rho_H$ denote the element of $\mathfrak{a}_s^*$ defined by the 
restriction of $\sum_{\alpha\in \Phi_{Q^u}\cup\Phi^+}\alpha$ to $\mathfrak{a}_s$.

\begin{thmd}
Assume that $F_{\mathcal{L}} >0$ and that the point 
\[
(\min_Y \Lambda_Y)\frac{\int_{\Delta^+}P_{DH}dq}{\int_{\Delta^+}F_{\mathcal{L}}P_{DH}dq}(\mathrm{bar}-\chi)-2\rho_H
\] 
is in the relative interior of the dual cone of $\mathfrak{a}_s^+$.
Then the Mabuchi functional is proper modulo the action of 
$Z(L)^0$.
\end{thmd}

This allows to obtain infinite families of new examples of classes 
where the Mabuchi functional is proper. While it is not known in 
general that this implies the existence of a cscK metric, it is 
proved that there exists a weak minimizer as candidate \cite{DR17}. 
Furthermore, in the case of homogeneous toric fibrations with 
fiber a toric surface, the existence of cscK metrics under 
the coercivity condition is proved in \cite{CHLLS}, hence 
in this case Theorem~D provides a combinatorial sufficient 
condition of existence of cscK metrics. 

In the case when $\mathcal{L}=K_X^{-1}\otimes \mathcal{O}(-\Theta)$ is ample, 
the pair $(X,\Theta)$ is log-Fano and the minimizer of the log-Mabuchi 
functional are log-Kähler-Einstein metrics. In this situation 
it is known that properness of the log-Mabuchi functional implies 
the existence of log-Kähler-Einstein metrics (see \cite{Dar17} for a 
statement allowing automorphisms), 
so we obtain the following corollary.

\begin{thme}
Assume $\mathcal{L}=K_X^{-1}\otimes \mathcal{O}(-\Theta)$, 
then $(X,\Theta)$ admits a log-Kähler-Einstein metric 
provided $\mathrm{bar}-\sum_{\alpha\in \Phi_{Q^u}\cup\Phi^+}\alpha$ 
is in the relative interior of the dual cone of $\mathfrak{a}_s^+$.
\end{thme}

If one works on, say, a biequivariant compactification of a 
semisimple group, then it is not hard to check that the condition 
above is open as $\mathcal{L}$ and $\Theta$ vary. 
Starting from an example of 
Kähler-Einstein Fano manifold obtained in \cite{DelKE}, 
we can extract from this corollary an explicit subset  
of $K_X^{-1}$ in the ample cone, with non-empty interior, such that each 
corresponding $\mathcal{L}$ writes as $\mathcal{L}=K_{(X,\Theta)}^{-1}$ 
and the pair $(X,\Theta)$ admits a log-Kähler-Einstein metric. 

 While the point of view adopted in this article is 
definitely in line with the author's earlier work on group 
compactifications and horospherical varieties, it should be 
mentionned that there were previous works, and different 
perspectives on both these classes. Group compactifications 
have been studied in detail from the algebraic point of view 
and the first article about the existence of 
canonical metrics on these was \cite{AK05} to the author's 
knowledge, and it builded on the extensive study of reductive 
varieties in \cite{AB04SRVI,AB04SRVII}. Homogeneous toric 
bundles have been studied for the Kähler-Einstein metric existence 
problem by Podesta and Spiro \cite{PS10}, their point of view on 
the Kähler geometry of this subclass of horospherical varieties 
being somewhat different from the author's. Donaldson highlighted 
in \cite{Don08} the importance and studying these varieties, and there 
were partly unpublished work of Raza and Nyberg on these subjects in 
their PhD theses \cite{Raz,Raz07,Nyb}. 
Finally, concerning the application to the Mabuchi functional, 
we were strongly influenced by Li, Zhou and Zhu's article \cite{LZZ}.
The latter in turn used as foundations on one side our work on group 
compactifications and on the other side a strategy for obtaining 
coercivity of the Mabuchi functional developped initially by 
Zhou and Zhu \cite{ZZ08}.
It should be noted that the criterion we obtain for (non-semi-simple) group 
compactifications is \emph{a priori} not equivalent to the one given in 
\cite{LZZ}. We do not claim 
that ours is better but only that theirs did not generalize naturally 
to our broader setting. 

The paper is organized as follows.
Section~\ref{sec_homogeneous} 
is devoted to the introduction of horosymmetric homogeneous spaces, and  of the 
combinatorial data associated to them. 
In Section~\ref{sec_curvature}, 
we introduce the toric potential of a $K$-invariant metric on 
a $G$-linearized line bundle on a horosymmetric homogeneous space, and compute 
the curvature form of such a metric in terms of this function. Even though 
the proof is rather technical, involving a lot of Lie bracket computations, 
it is a central part of the theory to have this precise expression. Theorem~A 
 is Corollary~\ref{cor_curv}, a special case of Theorem~\ref{thm_curv}.
In Section~\ref{sec_varieties}, 
we switch to horosymmetric varieties, we recall their combinatorial classification 
inherited from the theory of spherical varieties, and we check that a $G$-invariant 
irreducible codimension one subvariety remains horosymmetric. 
Section~\ref{sec_bundles} 
presents the combinatorial data associated with line bundles on horosymmetric 
varieties, and in particular the link between several convex polytopes associated 
to such a line bundle. 
Section~\ref{sec_metrics} 
applies the previous sections to hermitian metrics on polarized horosymmetric 
varieties, to obtain the behavior of toric potentials and an integration 
formula. In particular, Theorem~B is proved here 
(Proposition~\ref{prop_positive_metrics} and Proposition~\ref{prop_inthoro}). 
Finally, we give in Section~\ref{sec_mabuchi} 
the application to the Mabuchi functional, starting with a computation of 
the scalar curvature, then of the Mabuchi functional, to arrive to a coercivity 
criterion. Theorem~C, Theorem~D and Theorem~E are proved in this final section
(respectively in Theorem~\ref{thm_mabuchi}, Theorem~\ref{thm_coercivity} 
and Corollary~\ref{cor_log}). 

We tried to illustrate all notions by simple examples 
(even if they sometimes appear trivial, we believe they are essential to make the 
link between the theory of spherical varieties and standard examples of 
complex geometry) 
and to follow for the whole paper the example of symmetric varieties 
of type AIII.

\section{Horosymmetric homogeneous spaces}

\label{sec_homogeneous}

In this section we introduce horosymmetric homogeneous spaces and their 
associated combinatorial data, extracting from the literature the results 
needed for the next sections. 

\subsection{Definition and examples}

We always work over the field $\mathbb{C}$ of complex numbers.
Given an algebraic group $G$, we denote by $G^u$ its \emph{unipotent radical}. 
A subgroup $L$ of $G$ is called a \emph{Levi subgroup} of $G$ if it is a reductive 
subgroup of $G$ such that $G$ is isomorphic to the semidirect product 
of $G^u$ and $L$. 
There always exists a Levi subgroup, 
and any two Levi subgroups are conjugate by an element of $G^u$. 
A complex algebraic group $G$ is called \emph{reductive} if $G^u$ is trivial.

From now on and for the whole paper, 
$G$ will denote a connected, reductive, complex, linear algebraic group. 
Recall that a \emph{parabolic subgroup} of $G$ is a closed subgroup $P$ such that 
the corresponding homogeneous space $G/P$ is a projective manifold, called 
a \emph{generalized flag manifold}. Recall also for later use that 
a \emph{Borel subgroup} of $G$ is a parabolic subgroup which is minimal 
with respect to inclusion. Note that any parabolic subgroup of $G$ contains 
at least one Borel subgroup of $G$.

\begin{defn}
A closed subgroup $H$ of $G$ is a \emph{horosymmetric subgroup} if there 
exists a parabolic subgroup $P$ of $G$, a Levi subgroup $L$ of $P$ and 
a complex algebraic group involution $\sigma$ of $L$ such that 
\begin{itemize}
\item $P^u\subset  H\subset P$ and
\item $(L^{\sigma})^0\subset L\cap H$ as a finite index subgroup, 
\end{itemize}
where $L^{\sigma}$ denotes the subgroup of elements fixed by $\sigma$ 
and $(L^{\sigma})^0$ its neutral connected component. 
\end{defn}

\begin{rem}
The condition of being horosymmetric may be read off directly from the Lie algebra 
of $H$. As a convention, we denote the Lie algebra of a group by 
the same letter, in fraktur gothic lower case letter.
Then $H$ is symmetric if and only if there exists a parabolic subgroup $P$,
a Levi subgroup $L$ of $P$, and a complex Lie algebra involution 
$\sigma$ of $\mathfrak{l}$ such that 
\[
\mathfrak{h}=\mathfrak{p}^u\oplus \mathfrak{l}^{\sigma}
\]
\end{rem}

From now on, $H$ will denote a horosymmetric subgroup, and $P$, $L$, 
$\sigma$ will be as in the above definition. We keep the same notation 
$\sigma$ for the induced involution of the Lie algebra $\mathfrak{l}$.
We will also say that $G/H$ is a \emph{horosymmetric homogeneous space}.

Note that $L\cap H\subset N_L(L^{\sigma})$, and we have the following 
description of $N_L(L^{\sigma})$, due to De Concini and Procesi. 
They assume in their paper that $G$ is  semisimple but the proof 
applies to reductive groups just as well.

\begin{prop}[\cite{DCP83}]
\label{prop_normalizer}
The normalizer $N_L(L^{\sigma})$ is equal to the subgroup of all $g$ 
such that $g\sigma(g)^{-1}$ is in the center of $L$. 
\end{prop}

In particular if $L=G$ is semisimple, then $N_L(L^{\sigma})/(L^{\sigma})^0$
is finite. Note also that if in addition $L$ is adjoint, $N_L(L^{\sigma})=L^{\sigma}$ 
and if $L$ is simply connected, then $L^{\sigma}$ is connected. 

\begin{exa}
\label{exa_GFM}
Trivial examples of horosymmetric subgroups are obtained by 
setting $\sigma = \mathrm{id}_L$. Then $H=P$ is a parabolic subgroup 
and $G/H$ is a generalized flag manifold.
Since we will use them later, let us recall a fundamental example 
of flag manifold: the Grassmannian $\mathrm{Gr}_{r,m}$ of $r$-dimensional 
linear subspaces in $\mathbb{C}^m$, under the action of $\mathrm{SL}_m$.
The stabilizer of a point is a proper, maximal (with respect to inclusion) 
parabolic subgroup of $\mathrm{SL}_m$ (for $1\leq r\leq m-1$).
\end{exa}

\begin{exa}
\label{exa_toric}
Assume that $G=(\mathbb{C}^*)^n$, then $P=L=G$. If we consider the 
involution defined by $\sigma(g)=g^{-1}$, which is an honest complex 
algebraic group involution since $G$ is abelian, we obtain 
$\{e\}\subset H \subset \{\pm 1\}^n$ and in any case 
$G/H \simeq (\mathbb{C}^*)^n$. Hence a torus 
may be considered as a horosymmetric homogeneous space. 
\end{exa}

Let $[L,L]$ denote the derived subgroup of $L$ and $Z(L)$ the center 
of $L$. Then $L$ is a semidirect product of these two subgroups, which 
means, at the level of Lie algebras, that 
\[ 
\mathfrak{l} = 
[\mathfrak{l},\mathfrak{l}] \oplus \mathfrak{z}(\mathfrak{l}).
\]
Note that any involution of $L$ preserves this decomposition.

\begin{exa}
\label{exa_horospherical}
A closed subgroup of $G$ is called \emph{horospherical}
if it contains the unipotent radical of a Borel subgroup of $G$.
 
Assume that the involution $\sigma$ of $L$ restricts to the identity  
on $[\mathfrak{l},\mathfrak{l}]$. Then $H$ contains the unipotent 
radical of any Borel subgroup contained in $P$. Hence $H$ is horospherical.

Conversely, if $H$ is a horospherical subgroup of $G$, then taking 
$P:=N_G(H)$ which is a parabolic subgroup of $G$, and letting $L$ be 
any Levi subgroup of $P$, we have 
$\mathfrak{h}=\mathfrak{p}^u\oplus [\mathfrak{l},\mathfrak{l}] \oplus \mathfrak{b}$
where $\mathfrak{b}=\mathfrak{h}\cap \mathfrak{z}(\mathfrak{l})$ (see \cite[Section 2]{Pas08}). 
Choose any complement $\mathfrak{c}$ of $\mathfrak{b}$ in $\mathfrak{z}(\mathfrak{l})$, 
and consider the involution of $\mathfrak{l}$ defined as 
$\mathrm{id}$ on $[\mathfrak{l},\mathfrak{l}]\oplus \mathfrak{b}$ 
and as $-\mathrm{id}$ on $\mathfrak{c}$. This shows that $H$ is a 
horosymmetric subgroup of $G$.
\end{exa}

\begin{exa}
\label{exa_C2minus0}
Consider the linear action of $\mathrm{SL}_2$ on $\mathbb{C}^2\setminus\{0\}$. 
It is a transitive action and the stabilizer of $(1,0)$ is 
the unipotent subgroup $B^u$ of the Borel subgroup formed 
by upper triangular matrices. Under this action, $\mathbb{C}^2\setminus\{0\}$
is a horospherical, hence horosymmetric, homogeneous space.
Alternatively, one may consider the action of $\mathrm{GL}_2$ instead 
of the action of $\mathrm{SL}_2$.
\end{exa}

\begin{exa}
\label{exa_symmetric_subgroup}
Assume $P=L=G$, then $\sigma$ is an involution of $G$, and 
$(G^{\sigma})^0\subset H\subset N_G((G^{\sigma})^0)$. These 
subgroups are commonly known as \emph{symmetric subgroups} and the 
associated homogeneous spaces as (complex reductive) \emph{symmetric spaces}.
\end{exa}

All horosymmetric homogeneous spaces may actually be considered as 
parabolic inductions from symmetric spaces. 
Let us recall the definition of parabolic induction.
\begin{defn}
Let $G$ and $L$ be two reductive algebraic groups, then we say that a 
$G$-variety $X$ is obtained from an $L$-variety $Y$ by 
\emph{parabolic induction} if 
there exists a parabolic subgroup $P$ of $G$, and an surjective group morphism 
$P\rightarrow L$ such that $X=G*_PY$ is the $G$-homogeneous fiber bundle 
over $G/P$ with fiber $Y$.
\end{defn}

In our situation, $G/H$ admits a natural structure of $G$-homogeneous 
fiber bundle over $G/P$, with fiber $P/H$. The action of $P$ on $P/H$ 
factorizes by $P/P^u$ and under the natural isomorphism $L\simeq P/P^u$, 
identifies the fiber with the $L$-variety $L/L\cap H$, which is a 
symmetric homogeneous space.
Conversely, any parabolic induction from a symmetric space is a 
horosymmetric homogeneous space. 
The special case of horospherical homogeneous spaces consists of 
parabolic inductions from tori. 

We will denote by $f$ the $G$-equivariant map $G/H \rightarrow G/P$   
and by $\pi$ the quotient map $G\rightarrow G/H$.

Let us now give more explicit examples of horosymmetric homogeneous spaces, 
starting by examples of symmetric spaces.

\begin{exa}
\label{exa_invol_sl}
Assume $\mathfrak{g}=\mathfrak{sl}_m$ for some $m$. 
Then there are three families of group involutions of $\mathfrak{g}$
up to conjugation \cite[Sections 11.3.4 and 11.3.5]{GW09}. 
For a nicer presentation we work on the group 
$G=\mathrm{SL}_m$.  
For an integer $p>0$, we define the $2p\times 2p$ block diagonal matrix 
$T_p$ by $T_1=(\begin{smallmatrix} 0&1\\-1&0 \end{smallmatrix})$, and  $T_p=\mathrm{diag}(T_1,\ldots, T_1)$.
For an integer $0<r<m/2$, we define the $m\times m$ matrix $J_r$ as follows.
Let $S_r$ be the $r\times r$ matrix with coefficients $(\delta_{j+k,r+1})_{j,k}$, 
and set 
\[
J_r=\begin{pmatrix}
0 & 0 & S_r \\
0 & I_{n-2r} & 0 \\
S_r & 0 & 0
\end{pmatrix}.
\]
The types of involutions are the following:
\begin{enumerate}
\item (Type AI($m$))
Consider the involution of $G$ 
defined by $\sigma(g)=(g^t)^{-1}$ where $\cdot^t$ denotes the transposition
of matrices. Then 
$G^{\sigma}=\mathrm{SO}_m$. 
The symmetric space $G/N_G(G^{\sigma})$ may be identified with the space of 
non-degenerate quadrics in $\mathbb{P}^{m-1}$, equipped with the action 
of $G$ induced by its natural action on $\mathbb{P}^{m-1}$.

\item  (Type AII($p$))
Assume $n=2p$ is even. 
Let $\sigma$ be the involution defined by 
$\sigma(g)=T_p(g^t)^{-1}T_p^t$. Then $G^{\sigma}= \mathrm{Sp}_{2p}$ 
is the 
group of elements that preserve the non-degenerate skew-symmetric bilinear
form $\omega(u,v)=u^tT_pv$ on $\mathbb{C}^{2p}$.

\item (Type AIII($r$, $m$))
Let $\sigma$ be the involution $g\mapsto J_rgJ_r$.
Then $G^{\sigma}$ is conjugate to the subgroup 
$S(\mathrm{GL}_r\times \mathrm{GL}_{m-r})$.

The space $G/G^{\sigma}$ may be considered as the set of pairs $(V_1,V_2)$
of linear subspaces $V_j\subset \mathbb{C}^m$ of dimension $\mathrm{dim}(V_1)=r$,
$\mathrm{dim}(V_2)=m-r$, such that $V_1\cap V_2=\{0\}$. This is an (open dense) orbit 
for the diagonal action of $G$ on the product of Grassmannians 
$\mathrm{Gr}_{r,m}\times \mathrm{Gr}_{m-r,m}$. 
\end{enumerate}
\end{exa}

\begin{exa}
\label{exa_normalizer_AIII}
Let us illustrate the characterization of the normalizer of a symmetric subgroup 
in type AIII case. 
First, since $G=\mathrm{SL}_n$ is simply connected, $G^{\sigma}$ is connected.
Furthermore, it is easy to check here that $N_G(G^{\sigma})$ 
is different from $G^{\sigma}$ if and only if $n$ is even and $r=n/2$, 
in which case $G^{\sigma}$ is of index two in  $N_G(G^{\sigma})$.
For example, if $n=2$ and $r=1$, $N_G(G^{\sigma})$ is generated 
by $G^{\sigma}$ and $\mathrm{diag}(i,-i)$. 
In that situation, $G/N_G(G^{\sigma})$ is the space of unordered pairs 
$\{V_1,V_2\}$ of linear subspaces $V_j\subset \mathbb{C}^m$ of 
dimension $r$ for $j=1$ and $m-r$ for $j=2$, such that $V_1\cap V_2=\{0\}$.
\end{exa}

\begin{exa}
\label{exa_horosym}
Finally, let us give an explicit example of non trivial parabolic induction 
from a symmetric space. 
Consider the subgroup $H$ of $\mathrm{SL}_3$ defined as the set 
of matrices of the form 
\[
\begin{pmatrix} a & b & e \\ b & a & f \\ 0 & 0 & g \end{pmatrix}.
\] 
Then obviously $H$ is contained in the parabolic $P$ composed of 
matrices with zeroes where the general matrix of $H$ has zeroes, 
and contains its unipotent radical, which consists of the matrices 
as above with $a=g=1$ and $b=0$. The subgroup 
$L=\mathrm{S}(\mathrm{GL}_2\times \mathbb{C}^*)$ is then 
a Levi subgroup of $P$ and $L\cap H$ is the subgroup of elements 
of $L$ fixed by the involution $g\mapsto MgM$ where 
\[
M=\begin{pmatrix}0&1&0\\1&0&0\\0&0&1\end{pmatrix}.
\] 
\end{exa}

\subsection{Root systems}

\subsubsection{Maximally split torus}

A torus $T$ in $L$ is \emph{split} if $\sigma(t)=t^{-1}$ for any $t\in T$.
A torus $T$ in $L$ is \emph{maximally split} if $T$ is a 
$\sigma$-stable maximal torus in $L$ which contains a split torus   
$T_s$ of maximal dimension among split tori. 
It turns out that any split torus is contained in a $\sigma$-stable maximal
torus of $L$ \cite{Vus74} hence maximally split tori exist. 
From now on, $T$ denotes a maximally split torus in $L$ with respect to $\sigma$, 
and $T_s$ denotes its maximal split subtorus. If $T^{\sigma}$ denotes the subtorus of 
elements of $T$ fixed by $\sigma$, then $T^{\sigma}\times T_s \rightarrow T$ 
is a surjective morphism, with kernel a finite subgroup. 
The dimension of $T_s$ is called the \emph{rank} of the symmetric space $L/L\cap H$.

\begin{exa}
\label{exa_max_tori}
The ranks and maximal tori for involutions of $\mathrm{SL}_n$ 
are as follows.
\begin{itemize}
\item (Type AI($m$)) For $\sigma:g\mapsto (g^t)^{-1}$, the rank is $m-1$ and the torus 
$T$ of diagonal matrices is a split torus which is also maximal, hence $T_s=T$ in 
this case.
\item (Type AII($p$)) For $\sigma:g\mapsto T_p(g^t)^{-1}T_p^t$, with $m=2p$, the rank 
is $p-1$, and the torus of diagonal matrices provides a maximally split torus. 
The maximal split subtorus $T_s$ is then the subtorus of diagonal matrices of the form 
$\mathrm{diag}(a_1,a_1,a_2,a_2,\ldots,a_p,a_p)$ with $a_1,\ldots,a_{p-1}\in \mathbb{C}^*$
and $a_p=(a_1^2\cdots a_{p-1}^2)^{-1}$, and $T^{\sigma}$ is the subtorus of diagonal 
matrices of the form $\mathrm{diag}(a_1,a_1^{-1},a_2,a_2^{-1},\ldots,a_p,a_p^{-1})$
with $a_1,\ldots,a_p\in \mathbb{C}^*$. 
We record for later use that 
$\sigma(\mathrm{diag}(a_1,\ldots,a_n))=\mathrm{diag}(a_2^{-1},a_1^{-1},a_4^{-1},\ldots,a_{n-1}^{-1})$.
\item (Type AIII($r$, $m$) Finally, for $\sigma:g\mapsto J_rgJ_r$, the rank 
is $r$, and the torus $T$ of diagonal matrices is again maximally split. 
Let $\upsilon$ denote the permutation of $\{1,\ldots ,m\}$ defined by 
$\upsilon(i)=m+1-i$ if $1\leq i\leq r$ or $m+1-r\leq i\leq m$, and $\upsilon(i)=i$ 
otherwise. Then $\sigma$ acts on diagonal matrices as 
$\sigma(\mathrm{diag}(a_1,\ldots,a_m))=\mathrm{diag}(a_{\upsilon(1)},\ldots,a_{\upsilon(m)})$.
We then see that the subtorus $T^{\sigma}$ is the torus of diagonal matrices 
of the form 
$\mathrm{diag}(a_1,a_2,\ldots,a_{m-r},a_r,a_{r-1},\ldots,a_1)$
and that $T_s$ is the subtorus of diagonal matrices of the form 
$\mathrm{diag}(a_1,\ldots,a_{r},1,\ldots,1,a_r^{-1},\ldots,a_1^{-1})$.
\end{itemize}
\end{exa}

\subsubsection{Root systems and Lie algebras decompositions}

We denote by $\mathfrak{X}(T)$ the group of characters of $T$, that is, 
algebraic group morphisms from $T$ to $\mathbb{C}^*$. 
We denote by $\Phi\subset \mathfrak{X}(T)$ the root system of $(G,T)$. 
Recall the root space decomposition of $\mathfrak{g}$:
\[
\mathfrak{g} = \mathfrak{t} \oplus \bigoplus_{\alpha \in \Phi} \mathfrak{g}_{\alpha},
\qquad
\mathfrak{g}_{\alpha} = \{x\in \mathfrak{g} ; \quad 
\mathrm{Ad}(t)(x) = \alpha(t)x \quad \forall t\in T \}
\]
where $\mathrm{Ad}$ denotes the adjoint representation of $G$ on $\mathfrak{g}$. 

\begin{exa}
\label{exa_roots_A}
In our examples we concentrate on the case when $G=\mathrm{SL}_n$, 
and the root system is of type $A_{n-1}$.
Let us recall its root system with respect to the maximal torus of diagonal 
matrices, in order to fix the notations to be used in examples 
throughout the article. 
The roots are the group morphisms $\alpha_{j,k}:T\rightarrow \mathbb{C}^*$, 
for $1\leq j\neq k\leq n$, defined by 
$\alpha_{j,k}(\mathrm{diag}(a_1,\ldots,a_n))=a_j/a_k$.
The root space $\mathfrak{g}_{\alpha_{j,k}}$ is then the set of matrices 
with only one non zero coefficient at the intersection of the 
$j^\mathrm{th}$-line and $k^{\mathrm{th}}$-column.
\end{exa}

We denote by $\Phi_L\subset \Phi$ the root system of $L$ with respect to  $T$, 
by $\Phi_{P^u}\subset \Phi$ the set of roots of $P^u$, so that  
\[
\mathfrak{l} = \mathfrak{t} \oplus \bigoplus_{\alpha \in \Phi_L} \mathfrak{g}_{\alpha},
\qquad 
\mathfrak{p}=\mathfrak{l}\oplus \bigoplus_{\alpha \in \Phi_{P^u}} \mathfrak{g}_{\alpha}
\]
and
\[
\mathfrak{h} = \mathfrak{l}\cap \mathfrak{h} \oplus \bigoplus_{\alpha\in \Phi_{P^u}} \mathfrak{g}_{\alpha}.
\]

\begin{exa}
\label{exa_horosym2}
In the case of Example~\ref{exa_horosym}, 
$G=\mathrm{SL}_3$ and $T$ is the torus of diagonal matrices. 
Using notations from Example~\ref{exa_roots_A}, we have 
$\Phi=\{\pm\alpha_{1,2},\pm\alpha_{2,3},\pm\alpha_{1,3}\}$, 
$\Phi_L=\{\pm\alpha_{1,2}\}$,
$\Phi_{P^u}=\{\alpha_{1,3},\alpha_{2,3}\}$.
\end{exa}

\subsubsection{Restricted root system}

The set of roots in $\Phi_L$ fixed by $\sigma$ is a sub root system denoted by 
$\Phi_L^{\sigma}$.
Let 
$\Phi_s = \Phi_L \setminus \Phi_L^{\sigma}$. Note that $\Phi_s$ is not a root system 
in general.
Let us now introduce the restricted root system of $L/L\cap H$. 
Given $\alpha \in \Phi_L$, we set $\bar{\alpha}=\alpha-\sigma(\alpha)$. 
It is zero if and only if $\alpha\in \Phi_L^{\sigma}$.

\begin{prop}[{\cite[Section~4]{Ric82}}] 
\label{prop_restricted_root_system}
The set 
\[
\bar{\Phi}= \{\bar{\alpha} ; \alpha \in \Phi_s \} \subset \mathfrak{X}(T)
\]
is a (possibly non reduced) root system in the linear subspace of 
$\mathfrak{X}(T) \otimes \mathbb{R}$ it generates.
The Weyl group $\bar{W}$ of the root system $\bar{\Phi}$ may be identified 
with $N_L(T_s)/Z_L(T_s)$ and furthermore any element of $\bar{W}$ admits a 
representant in $N_{(L^{\sigma})^0}(T_s)$. 
\end{prop}

The root system $\bar{\Phi}$ is called the \emph{restricted root system} of the 
symmetric space $L/L\cap H$. We will also say that its elements are 
\emph{restricted roots}, that $\bar{W}$ is the \emph{restricted Weyl group}, 
etc.

Another interpretation of the restricted root system, which justifies the name, 
is obtained as follows. 
For any $t\in T_s$, and $\alpha\in \Phi_L$, we have 
\[
\sigma(\alpha)(t)=\alpha(\sigma(t))=\alpha(t^{-1})=(-\alpha)(t).
\]
As a consequence, $\bar{\alpha}|_{T_s}=2\alpha|_{T_s}$, 
that is, up to a factor two, $\bar{\alpha}$ encodes the restriction 
of $\alpha$ to $T_s$.
More significantly, given $\gamma\in \mathfrak{X}(T_s)$, let 
$\bar{\mathfrak{l}}_{\gamma}$ denote the subset of elements $x$ in 
$\mathfrak{l}$ such that $\mathrm{Ad}(t)(x)=\gamma(t)x$ for 
all $t\in T_s$. Then by simultaneous diagonalization, we check that 
$
\mathfrak{l} = \bigoplus_{\gamma\in \mathfrak{X}(T_s)} \bar{\mathfrak{l}}_{\gamma}.
$
We immediately remark that 
$\bar{\mathfrak{l}}_0$ contains $\mathfrak{t}$ and all 
$\mathfrak{g}_{\alpha}$ for $\alpha \in  \Phi_L^{\sigma}$, and 
that $\bar{\mathfrak{l}}_{\gamma}$ contains $\mathfrak{g}_{\alpha}$
as soon as $\alpha\in \Phi_L$ is such that $\bar{\alpha}|_{T_s}=2\gamma$.
By the usual root decomposition of $\mathfrak{l}$, we check that 
actually 
\[
\bar{\mathfrak{l}}_0=\mathfrak{t}\oplus\bigoplus_{\alpha\in \Phi_L^{\sigma}}\mathfrak{g}_{\alpha},
\qquad 
\bar{\mathfrak{l}}_{\gamma} = \bigoplus_{\bar{\alpha}|_{T_s}=2\gamma}\mathfrak{g}_{\alpha}
\] 
for $\gamma \neq 0$, and 
\[
\mathfrak{l} = \bar{\mathfrak{l}}_0 \oplus 
\bigoplus_{\bar{\alpha}\in \bar{\Phi}} \bar{\mathfrak{l}}_{\bar{\alpha}/2}.
\]
A restricted root $\bar{\alpha}$ is fully determined by its restriction 
to $T_s$ since $T=T_sT^{\sigma}$ and $\bar{\alpha}|_{T^{\sigma}}=0$ 
since $\sigma(\bar{\alpha})=-\bar{\alpha}$.

\begin{exa}
\label{exa_restricted_AI}
In the case of type AI($m$), $\Phi_L^{\sigma}$ is empty and $\Phi_s=\Phi_L=\Phi$.
For any $\alpha\in \Phi$, we have $\sigma(\alpha)=-\alpha$, hence the 
restricted root system is just the double $2\Phi$ of $\Phi$.
\end{exa}

\begin{exa}
\label{exa_restricted_AII}
In the case of type AII($p$), we check that 
\[
\sigma(\alpha_{j,k})=\alpha_{k+(-1)^{k+1},j+(-1)^{j+1}}.
\] 
In fact it is easier to identify the restricted root system by analysing
the restriction of roots to $T_s$. We denote an element of $T_s$ 
by $\mathrm{diag}(b_1,b_1,b_2,b_2,\ldots,b_p,b_p)$.
We check easily that, for $1\leq j\neq k \leq p$,
\[
\alpha_{2j,2k-1}|_{T_s}=\alpha_{2j-1,2k-1}|_{T_s}
=\alpha_{2j-1,2k}|_{T_s}=\alpha_{2j,2k}|_{T_s}
=b_j/b_k.
\]
We deduce that 
$\Phi_L^{\sigma}=\{\pm\alpha_{2j-1,2j};1\leq j\leq p\}$
and that the restricted root system is of type 
$A_{p-1}$, with elements 
$\bar{\alpha}_{2j,2k} : 
\mathrm{diag}(b_1,b_1,b_2,b_2,\ldots,b_p,b_p) \mapsto b_j^2/b_k^2$
for $1\leq j\neq k \leq p$.
\end{exa}

\begin{exa}
\label{exa_restricted_AIII}
In the case of type AIII($r$, $m$), finally, we will also identify the root system \emph{via}
restriction to $T_s$. We will denote an element of $T_s$, which is a diagonal matrix,  
by $\mathrm{diag}(b_1,\ldots,b_r, 1\ldots,1, b_r^{-1},\ldots,b_1^{-1})$.
In the case when $m=2r$, there are no $1$ in the middle and the restricted 
root system will be slightly different.

In general, the restriction $\alpha_{j,k}|_{T_s}$ is trivial if and only 
if $r+1\leq j\neq k\leq m-r$, which proves 
$\Phi_L^{\sigma}$ is the subsystem formed by these roots.

Since $\alpha_{j,k}=-\alpha_{k,j}$, it is obviously enough to consider 
only the case when $j<k$.
For $1\leq j<k \leq r$, we have 
$\alpha_{j,k}|_{T_s}=\alpha_{m-k+1,m-j+1}|_{T_s}=b_j/b_k$.
For $1\leq j \leq r$ and $r+1\leq k\leq m-r$, we have 
$\alpha_{j,k}|_{T_s}=\alpha_{m-k+1,m-j+1}|_{T_s}=b_j$.
Finally, for $1\leq j \leq r$ and $m+1-r\leq k\leq m$, we have 
$\alpha_{j,k}|_{T_s}=\alpha_{m-k+1,m-j+1}|_{T_s}=b_jb_{m+1-k}$.
Remark that in this last case, we may have 
$\alpha_{m-k+1,m-j+1}=\alpha_{j,k}$, namely when $j=m+1-k$. 
In this situation we obtain the 
function $b_j^2$. Hence, whenever $r+1\leq m-r$, or equivalently $r<m/2$
since both $r$ and $m$ are integers, the restricted root system is 
non reduced. It is possible to check 
that it is of type $BC_r$. In the remaining case, that is when 
$n=2r$, the restricted root system is of type $C_r$.
\end{exa}

\subsection{Cartan involution and fundamental domain}

There always exists a Cartan involution of $G$ such that its 
restriction to $L$ commutes with $\sigma$. We fix such a Cartan involution $\theta$. 
Denote by $K = G^{\theta}$ the corresponding 
maximal compact subgroup of $G$.  
Let $\mathfrak{a}_s$ denote the Lie subalgebra $\mathfrak{t}_s \cap i\mathfrak{k}$
of $\mathfrak{t}_s$.

Consider the group $\mathfrak{Y}(T_s)$ of one-parameter subgroups of $T_s$,
that is, algebraic group morphisms $\mathbb{C}^*\rightarrow T_s$.
This group naturally embeds in $\mathfrak{a}_s$: given $\lambda \in \mathfrak{Y}(T_s)$, 
it induces a Lie algebra morphism $d_e\lambda :\mathbb{C} \rightarrow \mathfrak{t}_s$.
Here we identified the Lie algebra of $\mathbb{C}^*$ with $\mathbb{C}$ and the 
exponential map is given by the usual exponential. 
Then $d_e\lambda(1)$ must be an element of $\mathfrak{a}_s$ and it determines 
$\lambda$ completely. This induces an injection of $\mathfrak{Y}(T_s)$ in 
$\mathfrak{a}_s$ which actually allows to identify $\mathfrak{a}_s$ 
with $\mathfrak{Y}(T_s)\otimes \mathbb{R}$.  

Recall that we may either consider the restricted root system $\bar{\Phi}$ 
as in $\mathfrak{X}(T)$, in which case it lies in the subgroup 
$\mathfrak{X}(T/T\cap H)$, or, \emph{via} the restriciton to $T_s$, 
we may consider $\bar{\Phi}$ to be in $\mathfrak{X}(T_s)$.
This allows to define a Weyl chambers in $\mathfrak{a}_s$ 
with respect to the restricted root system. Choose any such Weyl 
chamber, denote it by $\mathfrak{a}_s^+$ and call it the 
\emph{positive restricted Weyl chamber}.

\begin{prop}
\label{prop_fundamental_domain}
The natural map $\mathfrak{a_s} \rightarrow \exp(\mathfrak{a}_s)H/H$
is injective, and the intersection of a $K$-orbit 
in $G/H$ with $\exp(\mathfrak{a}_s)H$ is the image by this map of 
a $\bar{W}$-orbit in $\mathfrak{a}_s$. 
As a consequence, the subset $\exp(\mathfrak{a}_s^+)H/H$ is a fundamental 
domain for the action of $K$ on $G/H$. 
\end{prop} 

\begin{proof} 
Remark that $K$ acts transitively on the base 
$G/P$ of the fibration $f:G/H\rightarrow G/P$, since $P$ is parabolic.
We are then reduced to finding a fundamental domain for the 
action of $K\cap P = K \cap L$ on the fiber $L/L\cap H$.

Flensted-Jensen proves in \cite[Section 2]{FJ80} 
that a fundamental domain is given by the positive Weyl chamber 
of a root system which is in general different from the restricted 
root system described above.  
However, in our situation, the group $L$ and the involution $\sigma$ 
are complex, and this allow to show that the two chambers are the same. 

More precisely, Flensted-Jensen considers the subspace $\mathfrak{l}'$ 
of elements fixed by the involution $\sigma \theta$. The positive Weyl chamber 
he considers is then a positive Weyl chamber for the root system formed 
by the non zero eigenvalues of the action of $\mathrm{ad}(\mathfrak{a}_s)$ 
on $\mathfrak{l}'$.
Now remark that the involution $\sigma \theta$ stabilizes any of the 
subspaces $\bar{\mathfrak{l}}_{\bar{\alpha}/2}$, which we may decompose as 
$\bar{\mathfrak{l}}_{\bar{\alpha}/2}=\bar{\mathfrak{l}}'_{\bar{\alpha}/2} 
\oplus \bar{\mathfrak{l}}''_{\bar{\alpha}/2}$ 
where $\bar{\mathfrak{l}}'_{\bar{\alpha}/2} = \bar{\mathfrak{l}}_{\bar{\alpha}/2} \cap \mathfrak{l}'$ 
and $\bar{\mathfrak{l}}''_{\bar{\alpha}/2}$ is the subspace of elements $x$
such that $\sigma \theta (x) = -x$.
Furthermore, since $\sigma(it)=i\sigma(t)$, multiplication by $i$ induces 
a bijection between $\bar{\mathfrak{l}}'_{\bar{\alpha}/2}$ and 
$\bar{\mathfrak{l}}''_{\bar{\alpha}/2}$, and in particular $\bar{\mathfrak{l}}'_{\bar{\alpha}/2}$ 
is not $\{0\}$ if and only if so is $\bar{\mathfrak{l}}_{\bar{\alpha}/2}$. 
As a consequence, the set of non zero eigenvalues of the action of $\mathrm{ad}(\mathfrak{a}_s)$ 
on $\mathfrak{l}'$ is precisely $\bar{\Phi}$. 

The reader may find a more detailed account of the results of 
Flensted-Jensen and of the structure of the action of $K$ on $G/H$
in \cite[Section 3]{vdB05}. 
\end{proof}

\subsection{Colored data for horosymmetric spaces}

As a parabolic induction from a symmetric space, 
$H$ is a \emph{spherical subgroup} of $G$, that is, 
any Borel subgroup of $G$ acts with an open dense orbit on $G/H$
(see \cite{Bria,Per14,Tim11,Kno91} for general presentations 
of spherical homogeneous spaces, and spherical varieties which 
will appear later). 

Given a choice of Borel subgroup $B$, a spherical homogeneous space $G/H$
is determined by three combinatorial objects (the highly non-trivial theorem 
that these objects fully determine $H$ up to conjugacy was obtained by 
Losev \cite{Los09}). 
\begin{itemize}
\item The first one is its associated lattice $\mathcal{M}$,
defined as the subgroup of characters $\chi\in \mathfrak{X}(B)$
such that there exists a function $f\in \mathbb{C}(G/H)$ with 
$b\cdot f=\chi(b)f$ for all $b\in B$ (where $b\cdot f(x)=f(b^{-1}x)$ by definition). 
Let us call $\mathcal{M}$ 
the \emph{spherical lattice} of $G/H$. Let 
$\mathcal{N}=\mathrm{Hom}_{\mathbb{Z}}(\mathcal{M},\mathbb{Z})$
denote the dual lattice.
\item The second one, the \emph{valuation cone} $\mathcal{V}$, is defined as 
the set of elements of $\mathcal{N}\otimes \mathbb{Q}$ which 
are induced by the restriction of $G$-invariant, $\mathbb{Q}$-valued 
valuations on $\mathbb{C}(G/H)$ to $B$-semi-invariant functions 
as in the definition of $\mathcal{M}$.
\item Finally, the third object needed to characterize the spherical 
homogeneous space $G/H$ is 
the \emph{color map} $\rho : \mathcal{D} \rightarrow \mathcal{N}$,  
as a map from an abstract finite set $\mathcal{D}$ to $\mathcal{N}$, 
that is, we only need to know the image of $\rho$ and the cardinality 
of its fibers. The set $\mathcal{D}$ is actually the set of 
codimension one $B$-orbits in $G/H$, called \emph{colors}, and the map $\rho$ 
is obtained by associating to a color $D$ the element of $\mathcal{N}$ 
induced by the divisorial valuation on $\mathbb{C}(G/H)$ defined by $D$.
\end{itemize}

In the case of horosymmetric spaces (which are parabolic inductions 
from symmetric spaces) these data may mostly be interpreted 
in terms of the restricted root system for a well chosen Borel $B$.

Let $Q$ be the parabolic subgroup of $G$ opposite to $P$ with respect 
to $L$, that is, the only parabolic subgroup of $G$ such that $Q\cap P=L$ and 
$L$ is also a Levi subgroup of $Q$. 
Let $B$ be a Borel subgroup such that $T\subset B \subset Q$. 
Then $B\cap L$ is a Borel subgroup of $L$. 
By \cite[Lemma 1.2]{DCP83},
we may choose a Borel subgroup of $L$, containing $T$, or equivalently 
a positive root system $\Phi_L^+$ in $\Phi_L$, so that for any positive 
root $\alpha\in \Phi_L^+$, either $\sigma(\alpha)=\alpha$ or $-\sigma(\alpha)$
is in $\Phi_L^+$.
Since Borel subgroups of $L$ containing $T$ are conjugate by an element
of $N_L(T)$, we may choose a conjugate of the initially chosen Borel 
subgroup of $G$ so that the resulting Borel $B$ satisfies 
$T\subset B\subset Q$ and that $B\cap L$ satisfies the above property 
with respect to $\sigma$.

We fix such a Borel subgroup and denote by $\Phi^+$ the corresponding 
positive root system of $\Phi$. We will use the notations 
$\Phi_L^+=\Phi^+\cap \Phi_L$ and $\Phi_s^+ := \Phi_L^+ \cap \Phi_s$. 
Note also that $\Phi_{P^u}=-\Phi^+ \setminus \Phi_L$ and 
$\Phi_{Q^u}=-\Phi_{P^u}$. 
Let $S$ denote the set of \emph{simple roots} of $\Phi$ generating $\Phi^+$,
and let $S_L=\Phi_L\cap S$, $S_s=\Phi_s\cap S$.
This induces a natural choice of simple roots in the restricted 
root system: $\bar{S}=\{\bar{\alpha};\alpha\in S_s\}$, and corresponding 
positive roots $\bar{\Phi}^+=\{\bar{\alpha};\alpha\in \Phi_s^+\}$.

Given $\alpha\in \Phi$, recall that the \emph{coroot} $\alpha^{\vee}$ is 
defined as the unique element in $[\mathfrak{g},\mathfrak{g}]\cap \mathfrak{t}$
such that for all $x\in \mathfrak{t}$, 
$\alpha(x)=2\kappa(x,\alpha^{\vee})/\kappa(\alpha^{\vee},\alpha^{\vee})$
where $\kappa$ denotes the Killing form on $\mathfrak{g}$.
Since $\alpha$ is real on $\mathfrak{t}\cap i\mathfrak{k}$, 
the coroot $\alpha^{\vee}$ is in $\mathfrak{a}=\mathfrak{t}\cap i\mathfrak{k}$ 
which we may also identify with $\mathfrak{Y}(T)\otimes \mathbb{R}$.

\begin{exa}
\label{exa_coroot_sl}
In our favorite group $\mathrm{SL}_n$, the coroot $\alpha_{j,k}^{\vee}$ 
is the diagonal matrix with $l^{\mathrm{th}}$-coefficient equal 
to $\delta_{l,j}-\delta_{l,k}$.
\end{exa}

We use also the notion of restricted coroots for the restricted root, 
as defined in \cite[Section 2.3]{Vus90}. 

\begin{defn}
\label{defn_rcoroot}
Given $\alpha\in \Phi_s$ the \emph{restricted coroot} $\bar{\alpha}^{\vee}$ 
is defined as:
\begin{itemize}
\item $\alpha^{\vee}/2$ if $-\sigma(\alpha)=\alpha$ ($\alpha$ is then called a real root), 
\item $(\alpha^{\vee}-\sigma(\alpha^{\vee}))/2=(\alpha^{\vee}+(-\sigma(\alpha))^{\vee})/2$ if $\sigma(\alpha)(\alpha^{\vee})=0$,
\item $(\alpha-\sigma(\alpha))^{\vee}$ if $\sigma(\alpha)(\alpha^{\vee})=1$, in which 
case $\alpha-\sigma(\alpha)\in \Phi_s$. 
\end{itemize}
\end{defn}

The restricted coroots form a root system dual to the restricted root system, 
and we thus call \emph{simple restricted coroots} the basis of this root system 
corresponding to the choice of positive roots $\bar{\alpha}^{\vee}$ 
for $\alpha\in \Phi^+$. 
One has to be careful here: in general the simple restricted corots are not 
the coroots of simple restricted roots. 

\begin{exa}
\label{exa_coroots_AIII}
Consider the example of type AIII($2$, $m>4$). 
Then we already described the restricted root system in 
Example~\ref{exa_restricted_AIII}. 
There are two real roots $\alpha_{1,m}$ and $\alpha_{2,m-1}$. 
The restricted coroots are diagonal matrices of the form 
$\mathrm{diag}(b_1,b_2,0,\ldots,0,-b_2,-b_1)$ and we write 
this more shortly as a point with coordinates $(b_1,b_2)$.
The restricted coroot $\bar{\alpha}_{1,m}^{\vee}$ is then 
$(1/2,0)$, while 
$\bar{\alpha}_{2,m-1}^{\vee}=(0,1/2)$.
The roots $\alpha_{1,2}$ and $\alpha_{1,m-1}$ satisfy $\sigma(\alpha)(\alpha^{\vee})=0$,
hence we have $\bar{\alpha}_{1,2}^{\vee}=(1/2,-1/2)$ 
and $\bar{\alpha}_{1,m-1}^{\vee}=(1/2,1/2)$.
Finally, the roots $\alpha_{1,3}$ and $\alpha_{2,3}$ satisfy 
$\sigma(\alpha)(\alpha^{\vee})=1$, and we have 
$\alpha_{1,3}-\sigma(\alpha_{1,3})=\alpha_{1,m}$ 
and $\alpha_{2,3}-\sigma(\alpha_{2,3})=\alpha_{2,m-1}$, hence 
$\bar{\alpha}_{1,3}^{\vee}=(1,0)$ and 
$\bar{\alpha}_{2,3}^{\vee}=(0,1)$.
We have described that way all (positive) restricted coroots.
Figure~\ref{fig_coroots_AIII} illustrates the positive restricted 
roots and coroots in this example. 
\end{exa}

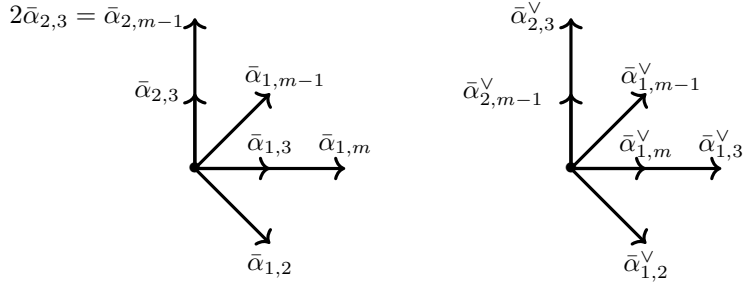
\begin{figure}
\centering
\caption{
Roots and coroots of type AIII($2$, $m>4$)}
\label{fig_coroots_AIII}
\begin{tikzpicture}  
\draw (0,0) node{$\bullet$};
\draw [very thick, ->] (0,0) -- (1,0);
\draw (1,.3) node{$\bar{\alpha}_{1,3}$};
\draw [very thick, ->] (0,0) -- (0,1);
\draw (-.5,1) node{$\bar{\alpha}_{2,3}$};
\draw [very thick, ->] (0,0) -- (2,0);
\draw (2,.3) node{$\bar{\alpha}_{1,m}$};
\draw [very thick, ->] (0,0) -- (0,2);
\draw (-1.3,2) node{$2\bar{\alpha}_{2,3}=\bar{\alpha}_{2,m-1}$};
\draw [very thick, ->] (0,0) -- (1,1);
\draw (1.2,1.2) node{$\bar{\alpha}_{1,m-1}$};
\draw [very thick, ->] (0,0) -- (1,-1);
\draw (1,-1.3) node{$\bar{\alpha}_{1,2}$};
\draw (5,0) node{$\bullet$};
\draw [very thick, ->] (5,0) -- (6,0);
\draw (6,.3) node{$\bar{\alpha}_{1,m}^{\vee}$};
\draw [very thick, ->] (5,0) -- (5,1);
\draw (4.1,1) node{$\bar{\alpha}_{2,m-1}^{\vee}$};
\draw [very thick, ->] (5,0) -- (7,0);
\draw (7,.3) node{$\bar{\alpha}_{1,3}^{\vee}$};
\draw [very thick, ->] (5,0) -- (5,2);
\draw (4.5,2) node{$\bar{\alpha}_{2,3}^{\vee}$};
\draw [very thick, ->] (5,0) -- (6,1);
\draw (6.2,1.2) node{$\bar{\alpha}_{1,m-1}^{\vee}$};
\draw [very thick, ->] (5,0) -- (6,-1);
\draw (6,-1.3) node{$\bar{\alpha}_{1,2}^{\vee}$};
\end{tikzpicture}
\end{figure}

Recall that $f:G/H\rightarrow G/P$ denotes the fibration map. 
The set of colors $\mathcal{D}(G/P)$ of the generalized flag manifold 
$G/P$ is in bijection with the set $\Phi_{Q^u}\cap S$ of  
simple roots that are also roots of $Q^u$, 
and any pre-image by $f$ of a color of $G/P$ is a 
color of $G/H$. Denote by $D_{\alpha}$ the color of $G/P$ 
associated with the root $\alpha\in \Phi_{Q^u}\cap S$. 

We identify $\mathfrak{a}_s$ and $\mathfrak{Y}(T/T\cap H)\otimes \mathbb{R}$.

\begin{prop}[{\cite[Proposition 20.4]{Tim11} and \cite{Vus90}}]
\label{prop_colored_data}
Let $G/H$ be a horosymmetric space.
\begin{itemize}
\item 
The spherical lattice $\mathcal{M}$ is the lattice $\mathfrak{X}(T/T\cap H)$. 
\item 
The valuation cone $\mathcal{V}$ is the 
negative restricted Weyl chamber $-\mathfrak{a}_s^+$.
\item 
The set of colors may be decomposed as a union of two sets 
$\mathcal{D}=\mathcal{D}(L/L\cap H) \cup f^{-1}\mathcal{D}(G/P)$.
The image of the color $f^{-1}(D_{\alpha})$ by $\rho$   
is the restriction $\alpha^{\vee}|_{\mathcal{M}}$ of the coroot $\alpha^{\vee}$ 
for $\alpha\in \Phi_{Q^u}\cap S$.   
The image $\rho(\mathcal{D}(L/L\cap H)$
on the other hand is the set of simple restricted coroots. 
\end{itemize}
\end{prop}

\begin{rem}
If $G=L$ is semisimple and simply connected, then $\mathcal{M}$ is a lattice between the lattice of restricted 
weights and the lattice of restricted roots determined by the restricted root system \cite{Vus90}.
More precisely, it is the lattice of restricted weights if and only 
if $H=G^{\sigma}$ and it is the lattice of restricted roots if and only if 
$H=N_G(G^{\sigma})$. 
\end{rem}

Remark that the proposition does not give here a complete description 
of $\rho$ in general as it does not give the cardinality of all orbits. 
There is however a rather general case where the discussion is simply 
settled. Say that the symmetric space $L/L\cap H$ has no Hermitian factor if 
$[L,L]\cap Z_L(L\cap H)$ is finite. Then Vust proved the following 
full caracterization of $\rho$:
\begin{prop}[\cite{Vus90}]
\label{prop_injective_no_hermitian}
If $L/L\cap H$ has no Hermitian factor, then $\rho_B$
is injective on $\mathcal{D}_B(L/L\cap H)$.
\end{prop}

Note, and this is a general fact for parabolic inductions, that 
the images of colors in $f^{-1}\mathcal{D}(G/P)$ by $\rho$ all 
lie in the valuation cone $\mathcal{V}$. Indeed, for any two 
simple roots $\alpha$ and $\beta$, $\kappa(\alpha,\beta)\leq 0$. 
Given $\alpha\in \Phi_{Q^u}\cap S$, this implies that 
$\kappa(\alpha,\beta)\leq 0$ for any $\beta\in \Phi_L^+$ and 
thus $\kappa(\alpha,\bar{\beta})\leq 0$ for $\beta\in S_s$.

\begin{exa}
\label{exa_colored_data_AIII}
We draw here Figure~\ref{fig_colored_data_AIII} 
as an example of colored data for the symmetric space 
of type AIII($2$, $m>4$). 
The dotted grid represents the dual of the spherical lattice (which 
coincides here with the lattice generated by restricted coroots), 
the cone delimited by the dashed rays represents the valuation cone
(the negative restricted Weyl chamber), and the circles 
are centered on the points in the image of the color map 
(the simple restricted coroots), the number 
of circles reflecting the cardinality of the fiber.
\end{exa}

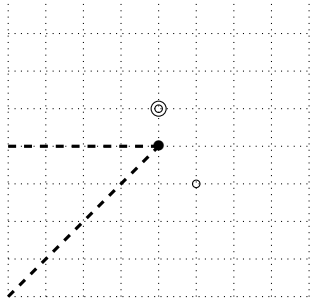
\begin{figure}
\centering
\caption{Colored data for type AIII($2$, $m>4$)}
\label{fig_colored_data_AIII}
\begin{tikzpicture}
\draw (0,0) node{$\bullet$};
\draw [dotted] (-4/2,-4/2) grid[xstep=1/2,ystep=1/2] (4/2,4/2); 
\draw [very thick, dashed] (-4/2,0) -- (0,0);  
\draw [very thick, dashed] (-4/2,-4/2) -- (0,0);
\draw (0,1/2) circle (.05);
\draw (0,1/2) circle (.1);
\draw (1/2,-1/2) circle (.05);
\end{tikzpicture}
\end{figure}

\section{Curvature forms}

\label{sec_curvature}

We now begin the study of Kähler geometry on horosymmetric spaces. We first 
recall how linearized line bundles on homogeneous spaces are encoded 
by their isotropy character, then we consider $K$-invariant hermitian 
metrics. We associate two functions to a hermitian metric: the quasipotential 
and the toric potential. We express the curvature form of the metric in terms 
of the isotropy character and toric potential, using the quasipotential as a tool 
in the proof. 

For this section, we use the letter $q$ to denote a metric, as the 
letter $h$ denotes elements of the group $H$.
Recall that given a hermitian metric $q$ on a line bundle $\mathcal{L}$, its 
curvature form $\omega$ may be defined locally as follows. Let $s$ be a local 
trivialization of $\mathcal{L}$ and let $\varphi$ denote the function defined 
by $\varphi=-\ln|s|^2_q$. Then the curvature form is the globally defined form 
which satisfies locally $\omega=i\partial\bar{\partial}\varphi$.

\subsection{Linearized line bundles on horosymmetric homogeneous spaces}

Let $\mathcal{L}$ be a $G$-linearized line bundle on $G/H$.  
The pulled back line bundle $\pi^*\mathcal{L}$ on $G$ is trivial, and we 
denote by $s$ a $G$-equivariant trivialization of $\pi^*\mathcal{L}$ on $G$.  
Denote by $\chi$ the character of $H$ defined by 
$h\cdot \xi = \chi(h) \xi$ for any $\xi$ in the fiber $\mathcal{L}_{eH}$. 
It fully determines the 
$G$-linearized line bundle $\mathcal{L}$.
The line bundle is trivializable on $G/H$ if and only if 
$\chi$ is the restriction 
of a character of $G$.

\begin{exa}
\label{exa_isotropy_ac}
The anticanonical line bundle admits a natural linearization, induced 
by the linearization of the tangent bundle. We may determine the 
isotropy character $\chi$ from the isotropy representation of $H$ 
on the tangent space at $eH$. If one identify this tangent space with 
$\mathfrak{g}/\mathfrak{h}$, then, working at the level of Lie algebras, 
the isotropy representation is given by 
$h\cdot (\xi+\mathfrak{h})= [h,\xi]+\mathfrak{h}$ 
for $h\in \mathfrak{h}$, $\xi\in \mathfrak{g}$.
Taking the determinant of this representation, we obtain for 
a horosymmetric homogeneous space $G/H$ that the isotropy character 
for the anticanonical line bundle is the restriction of the character 
$\sum_{\alpha\in \Phi_{Q^u}} \alpha$ of $P$ to $H$. 
\end{exa} 

\begin{exa}
\label{exa_isotropy_AIII}
On a Hermitian symmetric space, there may be non-trivial line bundles 
on $G/H$, as there may exist characters of $H$ which are not restrictions 
of characters of $G$. Let us illustrate this with our favorite type AIII 
example. 
Consider the matrix 
\[
M_r=\begin{pmatrix} 
\frac{1}{\sqrt{2}}I_r&0&\frac{1}{\sqrt{2}}S_r\\
0&I_{n-2r}&0\\
-\frac{1}{\sqrt{2}}S_r&0&\frac{1}{\sqrt{2}}I_r
\end{pmatrix}
\quad
\text{so that} 
\quad
M_rJ_rM_r^{-1}=\begin{pmatrix} -I_{r}&0\\0&I_{n-r}\end{pmatrix},
\]
then $M_rHM_r^{-1}=S(GL_{r}\times GL_{n-r})$.
This group obviously has non trivial 
characters not induced by a character of the (semisimple) group $G$, for example
$
(\begin{smallmatrix} A&B\\C&D \end{smallmatrix}) \mapsto \det(A).
$
Write an element of $H$ as 
\[
h=\begin{pmatrix}A_{11}&A_{12}&A_{13}\\A_{21}&A_{22}&A_{21}S_r\\S_rA_{13}S_r&S_rA_{12}&S_rA_{11}S_r\end{pmatrix}
\]
then composing with conjugation by $M$ we obtain the non-trivial 
character 
\[
\chi:h\mapsto \det(A_{11}+A_{13}S_r).
\]
\end{exa}

\begin{exa}
\label{exa_isotropy_P1xP1}
Consider the simplest example of type AIII, that is, 
$\mathbb{P}^1\times \mathbb{P}^1\setminus \mathrm{diag}(\mathbb{P}^1)$
equipped with the diagonal action of $SL_2$, and with base 
point $([1:1],[-1:1])$.
Then we have naturally linearized line bundles given by the 
restriction of $\mathcal{O}(k,m)$ for $k,m\in \mathbb{N}$.
The character associated to the line bundle $\mathcal{O}(k,m)$ is $\chi^{k-m}$ 
with $\chi$ as above, which translates here as 
$\chi:(\begin{smallmatrix}a&b\\b&a\end{smallmatrix}) \mapsto a+b$.
In particular we recover that it is trivial if and only if $k=m$.
\end{exa}

\subsection{Quasipotential and toric potential}
Let $q$ be a smooth $K$-invariant metric on $\mathcal{L}$. 

\begin{defn} \mbox{}
\begin{itemize}
\item The \emph{quasipotential} of $q$ is the function $\phi$ on $G$ defined by 
\[
\phi(g)=-2\ln|s(g)|_{\pi^*q}.
\]
\item The \emph{toric potential} of $q$ is the function 
$u: \mathfrak{a}_s \rightarrow  \mathbb{R}$ defined by 
\[
u(x)=\phi(\exp(x)).
\]
\end{itemize}
\end{defn}

\begin{prop}
\label{prop_equivariance_quasipotential}
The function $\phi$ satisfies the following equivariance relation:
\[
\phi(kgh)=\phi(g) - 2\ln|\chi(h)|.
\]
In particular $\phi$ is fully determined by $u$. 
\end{prop}

\begin{proof}
First, by $G$-invariance of $s$, we have 
\begin{align*}
\phi(kgh)& =-2\ln|kgh\cdot s(e)|_{\pi^*q} \\
& = -2\ln|k\cdot g\cdot h\cdot \pi_*s(e)|_q \\
\intertext{by equivariance of $\pi$}
& = -2\ln|g\cdot \chi(h) \pi_*(s(e))|_q \\
\intertext{by $K$-invariance of $q$ and by definition of $\chi$}
& = -2\ln|g\cdot s(e)|_{\pi^*q}- 2\ln|\chi(h)|
\end{align*}
Hence the equivariance relation.

Recall from Proposition~\ref{prop_fundamental_domain} that any $K$-orbit 
on $G/H$ intersects the image of $\mathfrak{a}_s$, 
so in view of the equivariance formula for $\phi$, we see that 
$\phi$, hence $q$, is fully determined by $u$.
\end{proof}

\subsection{Reference $(1,0)$-forms}

We choose root vectors $0\neq e_{\alpha} \in \mathfrak{g}_{\alpha}$ 
for $\alpha\in \Phi$ such that $[\theta(e_{\alpha}),e_{\alpha}]=\alpha^{\vee}$.
Using these root vectors, we can give a more explicit decomposition of $\mathfrak{h}$: 
\[
\mathfrak{h} = \bigoplus_{\alpha\in \Phi_{P^u}} \mathbb{C}e_{\alpha} \oplus \mathfrak{t}^{\sigma} 
\oplus \bigoplus_{\alpha\in \Phi_L^{\sigma}} \mathbb{C}e_{\alpha} \oplus 
\bigoplus_{\alpha\in \Phi_s^+} \mathbb{C}(e_{\alpha}+\sigma(e_{\alpha}))
\]

Choose a basis $(l_1, \ldots, l_r)$ of the real vector space $\mathfrak{a}_s$.
Let us further add the vectors $e_{\alpha}$ for $\alpha\in \Phi_{Q^u}$ and 
$\tau_{\beta}= e_{\beta} -\sigma(e_{\beta})$, 
for $\beta \in \Phi_s^+$. Then we obtain 
a family which is the complex basis of a complement of $\mathfrak{h}$ in 
$\mathfrak{g}$.  
This also defines local coordinates
\[
g\exp\left(\sum_j z_jl_j +\sum_{\alpha}z_{\alpha}e_{\alpha} +\sum_{\beta} z_{\beta}\tau_{\beta}\right)H
\]
 near a point $gH$ in $G/H$, depending on the 
choice of $g$. Let $\gamma_{\diamondsuit}^g$ denote the element of $\Omega^{(1,0)}_{gH} G/H$ 
defined by these coordinates, where $\diamondsuit$ is either some $j$, some $\alpha$ or some $\beta$.
Then $x\mapsto \gamma_{\diamondsuit}^{\exp(x)}$ provides a well defined 
(by Proposition~\ref{prop_fundamental_domain})
$\exp(\mathfrak{a}_s)$-invariant smooth $(1,0)$-form on $\exp(\mathfrak{a}_s)H/H$. 
From now on we denote by $\gamma_{\diamondsuit}$ the corresponding $(1,0)$-form 
and by $\omega_{\diamondsuit,\bar{\heartsuit}}$ the $(1,1)$-form 
$i\gamma_{\diamondsuit}\wedge\bar{\gamma}_{\heartsuit}$. 

\subsection{Reference volume form and integration}

We introduce a reference volume form on $G/H$.
Recall from example~\ref{exa_isotropy_ac} that 
the naturally linearized canonical line bundle $K_{L/L\cap H}$ 
on the symmetric space $L/L\cap H$ is $L$-trivial, 
that is, there exists a nowhere vanishing section 
$s_0 : L/L\cap H \rightarrow K_{L/L\cap H}$ which is $L$-equivariant.
We can further assume that $s_0$ coincides with 
$\bigwedge_{j} \gamma_j \wedge \bigwedge_{\beta} \gamma_\beta$
on $\exp(\mathfrak{a}_s)H/H$ where $j$ runs from $1$ to $r$ and 
$\beta$ runs over the set $\Phi_s^+$. 

Recall that $f$ denotes the map $G/H\rightarrow G/P$. Let 
$K_f = K_{G/H} - f^*K_{G/P}$ denote the relative canonical bundle. 
Then the section $s_0$ above may be considered as a trivialization 
of $K_f$ on the fiber above $eP\in G/P$. Since the map $f$ is 
$G$-equivariant, $K_f$ admits a natural $G$-linearization, 
and we may use the action of the maximal compact group $K$ to build 
a $K$-equivariant trivialization $s_f$ of $K_f$ on $G/H$. 
Setting $|s_f|_{q_f}=1$ provides a smooth $K$-invariant metric $q_f$ on $K_f$.
Let $q_P$ denote the smooth $K$-invariant metric on $K_{G/P}$ which 
satisfies $|f_*(\bigwedge_{\alpha} \gamma_{\alpha}^{e})|_{q_P}=1$, 
where $\alpha$ runs over the set $\Phi_{Q^u}$.
Pulling it pack provides a smooth $K$-invariant metric on $f^*K_{G/P}$.

The two metrics together provide a smooth reference metric 
$q_H=q_f\otimes f^*q_P$ on 
$K_{G/H}=K_f\otimes f^*K_{G/P}$, which is $K$-invariant.
We denote by $dV_H$ the associated smooth volume form on $G/H$. 
It is defined pointwise as follows: if $\xi$ is an element of the fiber  
of $K_{G/H}$ at $gH$, then 
$(dV_H)_{gH} = i^{n^2}|\xi|_{q_{H}}^{-2} \xi \wedge \bar{\xi}$.

\begin{prop}
\label{prop_relation_volume_forms}
Let $a\in \mathfrak{a}_s$, then 
\[
(dV_H)_{\exp(a)H} = e^{2\sum_{\alpha} \alpha(a)}
\left(\bigwedge_{\diamondsuit} 
\omega_{\diamondsuit,\bar{\diamondsuit}}\right)_{\exp(a)H}
\]
\end{prop}

\begin{proof}
At a point $\exp(a)H$ for $a\in \mathfrak{a}_s$, we can choose 
\[
\xi = \bigwedge_{\diamondsuit} \gamma_{\diamondsuit}^{\exp(a)} 
= \exp(-a)^*\cdot \bigwedge_{\diamondsuit} \gamma_{\diamondsuit}^{e},
\] 
and we get 
\begin{align*}
(dV_H)_{\exp(a)H} & = |\xi|_{q_{H}}^{-2} i^{n^2}\xi \wedge \bar{\xi} \\
& = |\exp(-a)^*\cdot \bigwedge_{\alpha}\gamma_{\alpha}^e|_{f^*q_P}^{-2} i^{n^2}\xi \wedge \bar{\xi} \\
\intertext{by definition of $q_H$ and $q_f$,}
& = |\exp(-a)^*\cdot f_*(\bigwedge_{\alpha}\gamma_{\alpha}^e)|_{q_P}^{-2} i^{n^2}\xi \wedge \bar{\xi} \\
& = e^{2\sum_{\alpha} \alpha(a)}i^{n^2}\xi \wedge \bar{\xi} \\
\intertext{because $P$ acts on the fiber at $eP$ of $K_{G/P}$ \emph{via} the 
character $-\sum_{\alpha\in \Phi_{Q^u}}\alpha$,}
& = e^{2\sum_{\alpha} \alpha(a)}i^{n}(-1)^{n(n-1)/2}\xi \wedge \bar{\xi}.
\end{align*}
\end{proof}

Remark that $dV_H$ depends on the precise choice of basis of the complement 
of $\mathfrak{h}$ in $\mathfrak{g}$ only by a multiplicative constant, as 
it only changes the element of the fiber of $K_{G/P}$ at $eP$ where $q_P$ 
takes value one, and the element of the fiber of $K_f$ at $eH$ where 
$q_f$ takes value one. 

Combining fiber integration with respect to the fibration $f$, and 
the formula for integration on symmetric spaces from 
\cite[Theorem 2.6]{FJ80}, we obtain a formula that reduces integration 
of a $K$-invariant function on $G/H$ with respect to $dV_{G/H}$ to 
integration of its restriction to $\exp(\mathfrak{a}_s^+)$ with respect 
to an explicit measure. 

Let $J_H$ denote the function on $\mathfrak{a}_s$ defined by 
\[
J_H(x) = \prod_{\alpha \in \Phi_s^+} |\sinh(2\alpha(x))|.  
\]
Another possible expression of the function $J_H$ is:
\[
J_H(x) = \prod_{\bar{\alpha} \in \bar{\Phi}^+} |\sinh(\bar{\alpha}(x))^{m_{\bar{\alpha}}}|
\]
where $m_{\bar{\alpha}} = \mathrm{dim}(\bar{l}_{\bar{\alpha}/2})$ is the 
number of $\beta \in \Phi_s$ such that $\bar{\beta}=\bar{\alpha}$. 

\begin{prop}
\label{prop_integration_horosymmetric_space}
There exists a constant $C_H>0$ such that
for any $K$-invariant function $\psi$ on $G/H$ which is integrable with respect to 
$dV_H$, we have 
\[
\int_{G/H} \psi dV_H = C_H \int_{\mathfrak{a}_s^+} \psi(\exp(x)H) J_H(x) dx
\] 
where $dx$ is a fixed Lebesgue measure on $\mathfrak{a}_s$.
\end{prop} 

Here again, a more detailed account on the integration formula for 
symmetric spaces may be found in \cite[Section 3]{vdB05}.

\subsection{Preparation for curvature form}

To shorten the formulas, we start using the following notations, 
for $y\in \mathfrak{g}$, 
\[
\Re(y)=\frac{y-\theta(y)}{2}\in i\mathfrak{k}
\qquad \mathrm{and} \qquad
\Im(y)=\frac{y+\theta(y)}{2}\in \mathfrak{k}.
\]
For $y\in \mathfrak{l}$, we will also use the notations 
\[
\mathcal{H}(y)=\frac{y+\sigma(y)}{2}\in \mathfrak{h}
\qquad \mathrm{and} \qquad
\mathcal{P}(y)=\frac{y-\sigma(y)}{2}.
\]
Remark that $\tau_{\beta}=2\mathcal{P}(e_{\beta})$ and define 
$\mu_{\beta}=2\mathcal{H}(e_{\beta})$.

\begin{lem}
\label{lem_prep1}
Let $a\in \mathfrak{a}_s$ be such that $\beta(a)\neq 0$ for all 
$\beta \in \Phi_s$.
Consider an element $D$ in $\mathfrak{g}$ and write 
\[
D = \sum_{1\leq j\leq r}z_jl_j + \sum_{\alpha\in \Phi_{Q^u}}z_{\alpha}e_{\alpha}
+\sum_{\beta \in \Phi_s^+}z_{\beta}\tau_{\beta}
+h
\] 
where $h\in \mathfrak{h}$, and $z_j$ for $1\leq j\leq r$, 
$z_{\alpha}$ for $\alpha\in \Phi_{Q^u}$, and 
$z_{\beta}$ for $\beta \in \Phi_s^+$ denote 
complex numbers.
Then we may write 
$D=A_D+B_D+C_D$ with $A_D\in \mathrm{Ad}(\exp(-a))(\mathfrak{k})$, 
$B_D \in \mathfrak{a}_1$ and $C_D\in \mathfrak{h}$ as follows.
\begin{align*}
A_D & = \sum_{1\leq j\leq r} \Im(z_jl_j) + \exp(\mathrm{ad}(-a)) \Big{\{} 
\sum_{\beta \in \Phi_s^+}
			\left( \frac{\Im(z_{\beta}\tau_{\beta})}{\cosh(\beta(a))} -
			\frac{\Im(z_{\beta}\mu_{\beta})}{\sinh(\beta(a))} \right)
\\ & \qquad 
+  \sum_{\alpha\in\Phi_{Q^u}} 2e^{\alpha(a)} \Im(z_{\alpha}e_{\alpha})  \Big{\}}
			\\
B_D & = \sum_{1\leq j\leq r}\Re(z_jl_j)  \\
C_D & = h  + 
		\sum_{\beta \in \Phi_s^+} \big{\{}
		\tanh(\beta(a))\Re(z_{\beta}\mu_{\beta})+ \coth(\beta(a))\Im(z_{\beta}\mu_{\beta}) \big{\}}
\\ & \qquad
		 +\sum_{\alpha\in \Phi_{Q^u}}-e^{2\alpha(a)}\theta(z_{\alpha}e_{\alpha})
\end{align*}
\end{lem}

\begin{proof}
This is a straightforward rewriting, using the following relations.
For $\alpha\in \Phi_{Q^u}$, 
\[
\exp(\mathrm{ad}(-a))(z_{\alpha}e_{\alpha} + \theta(z_{\alpha}e_{\alpha})) = 
e^{\alpha(-a)}z_{\alpha}e_{\alpha} + e^{-\alpha(-a)}\theta(z_{\alpha}e_{\alpha})
\] 
where we remark that 
$z_{\alpha}e_{\alpha} + \theta(z_{\alpha}e_{\alpha})\in \mathfrak{k}$
and 
$\theta(z_{\alpha}e_{\alpha})\in \mathfrak{h}$. 
For the terms in $\tau_{\beta}$, we use the relations 
\begin{align*}
z_{\beta}\tau_{\beta} & = \Re(z_{\beta}\tau_{\beta}) + \Im(z_{\beta}\tau_{\beta}), \\
\exp(\mathrm{ad}(-a))(\Im(z_{\beta}\tau_{\beta}))  
& = 
\cosh(\beta(a))\Im(z_{\beta}\tau_{\beta}) - 
\sinh(\beta(a))\Re(z_{\beta}\mu_{\beta}),  \\
\exp(\mathrm{ad}(-a))(\Im(z_{\beta}\mu_{\beta}))  
& = 
\cosh(\beta(a))\Im(z_{\beta}\mu_{\beta}) - 
\sinh(\beta(a))\Re(z_{\beta}\tau_{\beta}). 
\end{align*}
Note that the relations hold because $a\in \mathfrak{a}_s$, hence 
$\sigma(\beta)(a)=-\beta(a)$. 
\end{proof}

Let $a\in \mathfrak{a}_s$ be such that $\beta(a)\neq 0$ for all 
$\beta \in \Phi_s$, and consider now the function 
\[
D=D(\underline{z})=\sum_{1\leq j\leq r}z_jl_j + \sum_{\alpha\in \Phi_P^+}z_{\alpha}e_{\alpha}
+\sum_{\beta \in \Phi_L^+\setminus \Phi_L^{\sigma}}z_{\beta}\tau_{\beta},
\] 
where $\underline{z}$ denotes the tuple obtained by merging the tuples $(z_j)_j$, 
$(z_{\alpha})_{\alpha}$ and $(z_{\beta})_{\beta}$.
Let $A_D$, $B_D$, $C_D$ be the elements provided by Lemma~\ref{lem_prep1}
applied to $D$.
Let 
\[
E=E(\underline{z}):= ([B_D,D]+[C_D,B_D]+[C_D,D])/2
\] 
and introduce 
also $A_E$, $B_E$, $C_E$ the elements provided by Lemma~\ref{lem_prep1}
applied to $E$. 

\begin{lem}
\label{lem_prep_2}
For small enough values of $\underline{z}$, we have   
\[
\exp(D) = 
\exp(-a)k\exp(a+y+O) 
\exp(h),
\] 
where $O=O(\underline{z})\in \mathfrak{g}$ is of order strictly higher than two in $\underline{z}$,
$k=k(\underline{z})\in K$, $y=y(\underline{z})\in \mathfrak{a}_s$,             
and $h=h(\underline{z})\in \mathfrak{h}$. 
Furthermore,
\[
y=B_D+B_E
\]
and 
\[
\exp(h)=\exp(C_E)\exp(C_D).
\]
\end{lem}

\begin{proof} 
Throughout the proof, $O$ denotes an element of $\mathfrak{g}$ for 
$\underline{z}$ small enough, 
of order stricly higher than two in $\underline{z}$, which may change from line to line.

We first write $D= A_D + B_D + C_D$, with 
$A_D \in \mathrm{Ad}(\exp(-a))(\mathfrak{k})$, 
$B_D \in \mathfrak{a}_s$ and 
$C_D \in \mathfrak{h}$ given by Lemma~\ref{lem_prep1}.
Remark that they are all of order one in $\underline{z}$

Using the Baker-Campbell-Hausdorff formula \cite[Theorem X.3.1]{Hoc65} twice,  
we obtain that 
\[
\exp(-A_D)\exp(D)\exp(-C_D)=\exp(B_D+\frac{1}{2}([C_D,B_D]+[C_D,A_D]+[B_D,A_D])+O).
\]
Writing $A_D=D-B_D-C_D$ we easily check that $\frac{1}{2}([C_D,B_D]+[C_D,A_D]+[B_D,A_D])$
is equal to the $E$ introduced before.
We may then decompose again $E$ as $A_E+B_E+C_E$ 
where $A_E \in \mathrm{Ad}(\exp(-a))(\mathfrak{k})$, 
$B_E \in \mathfrak{a}_s$ and $C_E \in \mathfrak{h}$ 
given by Lemma~\ref{lem_prep1} and all terms are of order two in $\underline{z}$. 
Using again the Baker-Campbell-Hausdorff formula, we get 
\[
\exp(D)=\exp(A_D)\exp(A_E)\exp(B_D+B_E+O)\exp(C_E)\exp(C_D).
\]
The lemma is thus proved with a final application of the Baker-Campbell-Hausdorff
formula to $\exp(a)\exp(y+O)$.
\end{proof}

\subsection{Expression of the curvature form}

Given a function $u:\mathfrak{a}_s \rightarrow \mathbb{R}$ we may consider 
its differential $d_au \in \mathfrak{a}_s^*$ at a given point $a\in \mathfrak{a}_s$ 
as an element of $\mathfrak{a}^*$ by setting $d_au(x)=d_au\circ\mathcal{P}(x)$ and 
identifying $\mathfrak{a}^*_s$ with $\mathfrak{X}(T/T\cap H)\otimes \mathbb{R}$. 

Let $\mathcal{L}$ be a $G$-linearized line bundle corresponding to 
the character $\chi$ of $H$. We also denote by $\chi$ the corresponding 
Lie algebra character $\mathfrak{h}\rightarrow \mathbb{C}$. Hoping it 
will cause no confusion, we will also denote by $\chi$ the restriction 
of $\chi$ to $\mathfrak{a}\cap \mathfrak{h}$ and consider it as an element
of $\mathfrak{a}^*$ by setting $\chi(x)=\chi\circ\mathcal{H}(x)$ for 
$x\in \mathfrak{a}$.

Let $q$ be a smooth $K$-invariant metric on $\mathcal{L}$ with toric potential $u$, 
and let $\omega$ denote the curvature form of $q$. 

\begin{thm}
\label{thm_curv}
Let $a\in \mathfrak{a}_s$ be such that $\beta(a)\neq 0$ for all 
$\beta \in \Phi_s$. Then 
\[
\omega_{\exp(a)H} = \sum \Omega_{\diamondsuit,\bar{\heartsuit}} \omega_{\diamondsuit,\bar{\heartsuit}}
\]
where the sum runs over the indices $j$, $\alpha$, $\beta$, and 
the coefficients are as follows.
Let $1\leq j,j_1,j_2\leq r$, $\alpha$, $\alpha_1$, $\alpha_2\in \Phi_{Q^u}$, and 
$\beta$, $\beta_1$, $\beta_2\in \Phi_s^+$ with $\beta_1\neq \beta_2$ and $\alpha_2-\alpha_1\in \Phi_s$, then 
\[
\Omega_{j_1,\bar{j_2}} =\frac{1}{4}d^2u(l_{j_1},l_{j_2}), \quad \qquad 
\Omega_{j,\bar{\beta}} = \beta(l_j)(1-\tanh^2(\beta))\chi(\theta(\mu_{\beta})),
\]
\[
\Omega_{\alpha,\bar{\alpha}} = 
\frac{-e^{2\alpha}}{2}(du-2\chi)(\alpha^{\vee}), \quad \qquad 
\Omega_{\alpha_1,\bar{\alpha}_2} = 
\frac{2\chi([\theta(e_{\alpha_2}),e_{\alpha_1}])}{e^{-2\alpha_1}+e^{-2\alpha_2}},
\]
\begin{align*}
\Omega_{\beta_1,\bar{\beta}_2} = & 
\frac{\tanh(\beta_2-\beta_1)}{2}\Big{\{}\frac{1}{\sinh(2\beta_1)}
-\frac{1}{\sinh(2\beta_2)}\Big{\}} \chi([\theta(e_{\beta_2}),e_{\beta_1}])  
\\
 & +\frac{\tanh(\beta_1+\beta_2)}{2}\Big{\{}\frac{1}{\sinh(2\beta_2)}
+\frac{1}{\sinh(2\beta_1)}\Big{\}} \chi([\theta(e_{\beta_2}),\sigma(e_{\beta_1})]) 
\end{align*}
and 
\[
\Omega_{\beta,\bar{\beta}} =  
\frac{du(\beta^{\vee})}{\sinh(2\beta)} 
-\frac{2}{\cosh(2\beta)}\chi\circ\Re([\theta\sigma(e_{\beta}),e_{\beta}])
\]
where all quantities are evaluated at $a$.
Finally, the remaining coefficients except obviously the symmetric of those above are zero.
\end{thm}

This very involved description drastically simplifies if the 
restriction of $\chi$ to $L\cap H$ is trivial on $[L,L]\cap H$.
It is equivalent to the fact that it coincides with the restriction 
of a character of $L$ to $L\cap H$, or also to the fact that the 
corresponding line bundle is trivial on the symmetric 
fiber $L/L\cap H$.
This particular case in fact covers a wealth of examples, as it is the 
case for any choice of line bundle whenever the symmetric fiber has 
no Hermitian factor.
In the Hermitian case there are still plenty of line bundle which 
satisfy this extra assumption. A remarkable example is the 
anticanonical line bundle. 

\begin{cor} 
\label{cor_curv}
Assume that the restriction of $\mathcal{L}$ to the symmetric 
fiber $L/L\cap H$ is trivial.  
Let $a\in \mathfrak{a}_s$ be such that $\beta(a)\neq 0$ for all 
$\beta \in \Phi_s$. Then $\omega_{\exp(a)H}$ may compactly be written 
as 
\[
\frac{1}{4}d^2_au(l_{j_1},l_{j_2})\omega_{j_1,\bar{j_2}}
+  
\frac{-e^{2\alpha}}{2}(d_au-2\chi)(\alpha^{\vee}) \omega_{\alpha,\bar{\alpha}}
+  
\frac{d_au(\beta^{\vee})}{\sinh(2\beta(a))} \omega_{\beta,\bar{\beta}}
\]
where it is implicit that the three summands are actually sums over 
$\{1,\ldots,r\}$, $\Phi_{Q^u}$ and $\Phi_s^+$ respectively.
\end{cor}

\begin{exa}
\label{exa_curv_C2minus0}
Consider the example of $\mathbb{C}^2\setminus\{0\}$, viewed as a 
horospherical space under the natural action of $SL_2$. 
Then $\mathfrak{a}_s$ is one-dimensional and we may choose 
$\Phi_{Q^u}=\{\alpha_{2,1}\}$ ($\Phi_s^+$ is obviously empty).
We choose $l_1=\alpha_{2,1}^{\vee}$ 
as basis of $\mathfrak{a}_s$ and consider $u$ as a one real variable 
function, writing $d_au=u'(y)\alpha_{2,1}$  for $a=yl_1$, so that 
$d^2_a(l_1,l_1)=4u''(y)$. 
Then since $H=\mathrm{Stab}(1,0)$ has no characters, we have 
at $\exp(yl_1)H$,
\[
\omega= u''(y)\omega_{1,\bar{1}} + e^{-4y}u'(y)\omega_{\alpha_{2,1},\bar{\alpha_{2,1}}}. 
\]
\end{exa}

One major application of this general computation of curvature forms, 
but not the only one, will be through the \emph{Monge-Ampère operator}, 
which reads, with respect to the reference volume form, as follows.

\begin{cor}
\label{cor_MA}
Assume that the restriction of $\mathcal{L}$ to the symmetric 
fiber $L/L\cap H$ is trivial.  
Let $a\in \mathfrak{a}_s$ be such that $\beta(a)\neq 0$ for all 
$\beta \in \Phi_s$. Then at $\exp(a)H$, $\omega^n/dV_H$ is equal 
to  
\[
\frac{n!}{2^{2r+|\Phi_{Q^u}|}} \frac{\mathrm{det}(((d_a^2u)(l_j,l_k))_{j,k})}{J_H(a)} 
\prod_{\alpha\in \Phi_{Q^u}} 
(2\chi-d_au)(\alpha^{\vee}) 
\prod_{\beta \in \Phi_s^+} |d_au(\beta^{\vee})|
 \]
\end{cor}

\begin{exa}
\label{exa_MA_AIII}
Consider the example of symmetric space of type AIII($2$, $m>4$). 
We choose as basis $l_1, l_2$ the basis dual to $(\bar{\alpha}_{1,2},\bar{\alpha}_{2,3})$.
We write $a=a_1l_1+a_2l_2$ and $d_au=u_1(a)\bar{\alpha}_1+u_2(a)\bar{\alpha}_2$.
We check easily that 
$\mathcal{P}(\alpha_{1,2}^{\vee})=\bar{\alpha}_{1,2}^{\vee}$, 
$\mathcal{P}(\alpha_{1,m-1}^{\vee})=\bar{\alpha}_{1,m-1}^{\vee}$,
$\mathcal{P}(\alpha_{1,m}^{\vee})=\bar{\alpha}_{1,3}^{\vee}$,
$\mathcal{P}(\alpha_{2,m-1}^{\vee})=\bar{\alpha}_{2,3}^{\vee}$, 
$\mathcal{P}(\alpha_{1,k}^{\vee})=\bar{\alpha}_{1,m}^{\vee}$ and 
$\mathcal{P}(\alpha_{2,k}^{\vee})=\bar{\alpha}_{2,m-1}^{\vee}$ 
for $3\leq k \leq n-2$. Hence we may compute, under the assumption 
that $\chi$ is zero, that at $\exp(a)H$, $\omega^n/dV_H$ is equal 
to $n!/2$ times 
\[
\frac{(u_{1,1}u_{2,2}-u_{1,2}^2)(2u_1-u_2)^2u_1^{2m-7}u_2^2(u_2-u_1)^{2m-7}}
{\sinh(a_1)^2\sinh(a_1+a_2)^{2m-8}\sinh(2a_1+2a_2)\sinh(a_1+2a_2)^2\sinh(a_2)^{2m-8}\sinh(2a_2)}
\]
\end{exa}

Let us now illustrate on examples how the other terms in the curvature form 
may appear. 

\begin{exa}
\label{exa_curv_P1xP1}
Consider again $\mathbb{P}^1\times \mathbb{P}^1\setminus \mathrm{diag}(\mathbb{P}^1)$
equipped with the diagonal action of $SL_2$, and with the linearized line bundle 
$\mathcal{O}(k,m)$.
Then we can take $\beta=\alpha_{1,2}$, 
$e_{\beta}=(\begin{smallmatrix}0&1\\0&0\end{smallmatrix})$ 
and $l_1=\beta^{\vee}=(\begin{smallmatrix} 1&0\\0&-1\end{smallmatrix})$. 
We may further consider $u$ as a function of a single variable $t$ by 
writing $a=(\begin{smallmatrix}t&0\\0&-t\end{smallmatrix})$ and we get 
\[
\omega_{\exp(a)H}=\frac{u''(t)}{4}\omega_{1,\bar{1}}
+2(m-k)(1-\tanh^2(2t))(\omega_{1,\bar{\beta}}+\omega_{\beta,\bar{1}})
+\frac{u'(t)}{\sinh(4t)}\omega_{\beta,\bar{\beta}}
\]
\end{exa}

\begin{exa}
Consider the symmetric space of type AIII($1$, $3$). 
It admits the non-trivial character $\chi:(a_{i,j})\mapsto a_{1,1}+a_{1,3}$. 
For a metric on the line bundle corresponding to this character, 
we get for example 
\[
\mathcal{R}([\theta\sigma(e_{\alpha_{1,3}}),e_{\alpha_{1,3}}]=
\begin{pmatrix}0&0&1/2\\0&0&0\\1/2&0&0\end{pmatrix}
\]
hence a non-trivial contribution in $\Omega_{\alpha_{1,3},\bar{\alpha}_{1,3}}$ 
which is equal to 
\[
\frac{d_au(\alpha_{1,3}^{\vee})}{\sinh(2\alpha_{1,3}(a))}-\frac{1}{\cosh(2\alpha_{1,3}(a))}.
\]
\end{exa}

\begin{exa}
Consider the symmetric space $G/G^{\sigma}$ of type AIII($2$, $4$). 
It admits the non-trivial character $\chi:(a_{i,j})\mapsto a_{1,1}+a_{2,2}+a_{2,3}+a_{1,4}$, 
at the Lie algebra level and for $(a_{i,j})\in \mathfrak{h}$. 
For a metric on the line bundle corresponding to this character, 
we have for example 
$[\theta(e_{\alpha_{3,4}}),e_{\alpha_{2,4}}]=e_{\alpha_{2,3}}$,  
$[\theta(e_{\alpha_{3,4}}),\sigma(e_{\alpha_{2,4}})]=-e_{\alpha_{4,1}}$
and $\chi(e_{\alpha_{2,3}})=\chi\circ\mathcal{H}(e_{\alpha_{2,3}})=1/2$,
$\chi(-e_{\alpha_{4,1}})=\chi\circ\mathcal{H}(-e_{\alpha_{4,1}})=-1/2$,
hence, writing $a=\mathrm{diag}(t_1,t_2,-t_2,-t_1)$ we have 
\[
\Omega_{\alpha_{2,4},\bar{\alpha}_{3,4}}=
\frac{-1}{2\cosh(2t_1)\cosh(2t_2)}.
\]
\end{exa}

\begin{exa}
Consider again example~\ref{exa_horosym}. Using the same notations, 
we have a non-trivial character $\chi$ which associates $a+b$ to 
any element of $H$.
We have 
$[\theta(e_{\alpha_{2,3}}),e_{\alpha_{1,3}}]=e_{\alpha_{1,2}}$, 
and $\chi(e_{\alpha_{1,2}})=\chi\circ\mathcal{H}(e_{\alpha_{1,2}})=1/2$, 
hence we have 
\[\Omega_{\alpha_{1,3},\bar{\alpha}_{2,3}}=\frac{1}{2\cosh(2t)}\]
at the point $\exp(\mathrm{diag}(t,-t,0))H$.
We  check also that 
\[
\Omega_{\alpha_{2,3},\bar{\alpha}_{2,3}}=e^{2t}(2-u'(t))/4.
\]
\end{exa}

\subsection{Proof of Theorem~\ref{thm_curv}}

\mbox{}

\textbf{Step 1}

Recall that $\pi$ denotes the quotient map $G\rightarrow G/H$. 
By definition of the quasipotential $\phi : G \rightarrow \mathbb{R}$ of $q$, 
$i\partial \overline{\partial} \phi$ is the curvature form of $\pi^*q$. 
Furthermore, this curvature form coincides with $\pi^*\omega$.

Let $f_{\diamondsuit}\in \mathfrak{g}$ be any of the elements $l_j$, $e_{\alpha}$ or $\tau_{\beta}$ 
for $1\leq j\leq r$, $\alpha\in \Phi_{Q^u}$ or $\beta\in \Phi_s^+$. 
Identifying $\mathfrak{g}$ with $T^{(1,0)}_eG$, we build a global $G$-invariant 
$(1,0)$ holomorphic vector fields $\eta_{\diamondsuit}$ by setting 
$(\eta_{\diamondsuit})_g=g_*f_{\diamondsuit}\in T^{(1,0)}_gG$. 
Then 
\[
\pi^*\omega_g(\eta_{\diamondsuit},\bar{\eta}_{\heartsuit}) = 
\left.i \frac{\partial^2}{\partial z_{\diamondsuit} \partial \bar{z}_{\heartsuit}}\right|_0
\phi(g\exp(z_{\diamondsuit}f_{\diamondsuit} + z_{\heartsuit}f_{\heartsuit})).
\]

By definition, the set of all direct images $\pi_*\eta_{\diamondsuit}$ at $\exp(a)$ 
provides a basis of $T^{(1,0)}_{\exp(a)H}G/H$ which coincides with the dual basis to
the basis formed by the $(\gamma_{\diamondsuit})_{\exp(a)H}$ in 
$\Omega^{(1,0)}_{\exp(a)H}G/H$.
We thus have 
\begin{align*}
\Omega_{\diamondsuit,\bar{\heartsuit}} & = -i \omega_{\exp(a)H}(\pi_*\eta_{\diamondsuit},\pi_*\bar{\eta}_{\heartsuit}) \\
& = - i (\pi^*\omega)_{\exp(a)}(\eta_{\diamondsuit},\bar{\eta}_{\heartsuit}) \\
& = \left. \frac{\partial^2}{\partial z_{\diamondsuit} \partial \bar{z}_{\heartsuit}}\right|_0
\phi(\exp(a)\exp(z_{\diamondsuit}f_{\diamondsuit} + z_{\heartsuit}f_{\heartsuit})).
\end{align*}

\textbf{Step 2}

Set $D=z_{\diamondsuit}f_{\diamondsuit} + z_{\heartsuit}f_{\heartsuit}$.
Using Lemma~\ref{lem_prep_2}, we write 
\[
\exp(D)= \exp(-a)k\exp(a+y+O)\exp(h)
.\]
Then 
\begin{align*}
\phi(\exp(a)\exp(D)) & = \phi(k\exp(a+y+O)\exp(h)) \\
\intertext{by the equivariance property of the quasipotential (Proposition~\ref{prop_equivariance_quasipotential}), 
this is}
& = \phi(\exp(a+y+O))-2\ln|\chi(\exp(h))|
\end{align*}
Recall from Lemma~\ref{lem_prep_2} and the notations introduced before this lemma 
that $y=B_D+B_E$ and $\exp(h)=\exp(C_E)\exp(C_D)$ where 
$E=\frac{1}{2}([B_D,D]+[C_D,B_D]+[C_D,D])$
and $B_D$, $C_D$, $B_E$, $C_E$ are provided by Lemma~\ref{lem_prep1}.  
Note that 
\begin{align*}
\ln|\chi(\exp(C_E)\exp(C_D))| &= \ln|\chi(\exp(C_E))|+\ln|\chi(\exp(C_D))| \\
& = \ln|e^{\chi(C_E)}|+\ln|e^{\chi(C_D)}| \\
\intertext{where we still denote by $\chi$ the Lie algebra character 
$\mathfrak{h}\rightarrow \mathbb{C}$ induced by $\chi$,}
& = \mathrm{Re}(\chi(C_E)+\chi(C_D)) \\
& = \mathrm{Re}(\chi(C_E+C_D)). 
\end{align*}
We may now write 
\begin{align*}
\Omega_{\diamondsuit,\bar{\heartsuit}} & = 
\left.\frac{\partial^2}{\partial z_{\diamondsuit} \partial \bar{z}_{\heartsuit}}\right|_0 
\phi(\exp(a+B_D+B_E+O))-2 \ln|\chi(\exp(C_E)\exp(C_D))| \\
& = 
\left.\frac{\partial^2}{\partial z_{\diamondsuit} \partial \bar{z}_{\heartsuit}}\right|_0 
\phi(\exp(a+B_D+B_E))-2 \ln|\chi(\exp(C_E)\exp(C_D))| \\
& = 
\left.\frac{\partial^2}{\partial z_{\diamondsuit} \partial \bar{z}_{\heartsuit}}\right|_0 
(u(a+B_D+B_E)-2\ \mathrm{Re}(\chi(C_E+C_D)).
\end{align*}
Note that here the term $O$ denoted terms of order strictly higher than two 
in $(z_{\diamondsuit},z_{\heartsuit})$, which become negligeable in our computation.
Actually, other terms will be negligible and we will now denote by $O$ a sum of terms 
(which may change from line to line) each with a factor among 
$z_{\diamondsuit}^2$, $\bar{z}_{\diamondsuit}^2$, $z_{\diamondsuit}\bar{z}_{\diamondsuit}$,
$z_{\heartsuit}^2$, $\bar{z}_{\heartsuit}^2$, $z_{\heartsuit}\bar{z}_{\heartsuit}$, 
$z_{\diamondsuit}z_{\heartsuit}$ or $\bar{z}_{\diamondsuit}\bar{z}_{\heartsuit}$.

\textbf{Step 3}

The case by case computation follows.

1) 
Consider the case $D=z_1l_{j_1}+z_2l_{j_2}$, then 
we have $B_D=(z_1+\bar{z}_1)l_{j_1}/2+(z_2+\bar{z}_2)l_{j_2}/2$ and 
$B_E=C_E=C_D=0$ hence 
\[
\Omega_{j_1,\bar{j}_2} =
\frac{1}{4} d^2_a u(l_{j_1},l_{j_2}). 
\]

2) 
Consider the case $D=z_1e_{\alpha}+z_2l_j$.
By Lemma~\ref{lem_prep1}, we have $B_D= \Re(z_2l_j)$ and 
$C_D=-e^{2\alpha(a)}\theta(z_1 e_{\alpha})$.
We now compute $E=([B_D,D]-[B_D,C_D]+[C_D,D])/2$: 
\[ 
2[B_D,D]=O-z_1\bar{z}_2\alpha(\theta(l_j))e_{\alpha},
\]
\[
2[B_D,C_D]=z_2\bar{z}_1\alpha(l_j)e^{2\alpha(a)}\theta(e_{\alpha})+O,
\]
\[
[C_D,D]=-z_2\bar{z}_1\alpha(l_j)e^{2\alpha(a)}\theta(e_{\alpha})+O,
\]
hence
\[
E = \frac{1}{4}z_1\bar{z}_2\alpha(l_j)e_{\alpha} - \frac{3}{4}
z_2\bar{z}_1\alpha(l_j)e^{2\alpha(a)}\theta(e_{\alpha})+O.
\]
Using Lemma~\ref{lem_prep1} again we check that $B_E=O$ is negligible and 
\[
C_E=-2z_2\bar{z}_1\alpha(l_j)e^{2\alpha(a)}\theta(e_{\alpha})+O.
\]
Since $\theta(e_{\alpha})$ is in the Lie algebra of the unipotent radical of 
$H$, we have $\chi(\theta(e_{\alpha}))=0$, we may thus end the 
computation and obtain
\[
\Omega_{j,\bar{\alpha}}=\Omega_{\alpha,\bar{j}}=0.
\]

3) Consider the case $D=z_1e_{\alpha_1}+z_2e_{\alpha_2}$.
We have $B_D=0$ and 
\[
C_D=-e^{2\alpha_1(a)}\theta(z_1e_{\alpha_1})-e^{2\alpha_2(a)}\theta(z_2e_{\alpha_2}).
\]
Then $E=[C_D,D]/2$ is equal to 
\[
E = O -e^{2\alpha_1(a)}\bar{z}_1z_2[\theta(e_{\alpha_1}),e_{\alpha_2}]/2 
 -e^{2\alpha_2(a)}z_1\bar{z}_2[\theta(e_{\alpha_2}),e_{\alpha_1}]/2 
\]
We then need to treat several cases separately, depending on $\alpha_1-\alpha_2$. 

3.i) 
If $\alpha_1=\alpha_2=\alpha$ then we have 
\[
E = -\frac{1}{2}e^{2\alpha(a)}(\bar{z}_1z_2+z_1\bar{z}_2)\alpha^{\vee} +O
\]
hence 
\[
B_E =-\frac{1}{2}e^{2\alpha(a)}(\bar{z}_1z_2+z_1\bar{z}_2)\mathcal{P}(\alpha^{\vee}) +O
\] 
and 
\[
C_E=-\frac{1}{2}e^{2\alpha(a)}(\bar{z}_1z_2+z_1\bar{z}_2)\mathcal{H}(\alpha^{\vee}) +O. 
\]
Since $\chi$ is trivial on the unipotent radical of $H$, we have 
\[
\mathrm{Re}(\chi(C_D+C_E))=\mathrm{Re}(\chi(C_E))=
\chi\circ \Re(C_E)=\chi(C_E).
\]
We then end the computation to obtain 
\begin{align*}
\Omega_{\alpha,\bar{\alpha}} & = \frac{-1}{2}e^{2\alpha(a)}(d_au(\alpha^{\vee})
-2\chi(\alpha^{\vee})) 
\end{align*}

3.ii) 
If $\alpha_2-\alpha_1 \in \Phi_L^{\sigma}$ then we get $B_E=O$ and $C_E=O+E$.
Furthermore, $[\theta(e_{\alpha_1}),e_{\alpha_2}]\in 
\mathfrak{g}_{\alpha_2-\alpha_1}\subset [\mathfrak{h},\mathfrak{h}]$ 
(consider $[[\theta(e_{\alpha_2-\alpha_1}),e_{\alpha_2-\alpha_1}],e_{\alpha_2-\alpha_1}]$)
hence $\chi([\theta(e_{\alpha_1}),e_{\alpha_2}])=0$, and the same holds for 
$[\theta(e_{\alpha_2}),e_{\alpha_1}]$. 
We can then end the computation and obtain  
\[
\Omega_{\alpha_1,\bar{\alpha_2}}=0.
\]

3.iii) 
If $\alpha_2-\alpha_1 \in \Phi_L\setminus \Phi_L^{\sigma}$ then 
we have $B_E=O$ and 
\begin{align*}
C_E  = & O -e^{2\alpha_1(a)}\bar{z}_1z_2\mathcal{H}([\theta(e_{\alpha_1}),e_{\alpha_2}])/2 \\
& -e^{2\alpha_2(a)}z_1\bar{z}_2\mathcal{H}([\theta(e_{\alpha_2}),e_{\alpha_1}])/2 \\
&  +\tanh((\alpha_2-\alpha_1)(a))\Re(-e^{2\alpha_1(a)}\bar{z}_1z_2\mathcal{H}([\theta(e_{\alpha_1}),e_{\alpha_2}])/2)) \\
&  +\coth((\alpha_2-\alpha_1)(a))\Im(-e^{2\alpha_1(a)}\bar{z}_1z_2\mathcal{H}([\theta(e_{\alpha_1}),e_{\alpha_2}])/2)) \\
&  +\tanh((\alpha_1-\alpha_2)(a))\Re(-e^{2\alpha_2(a)}\bar{z}_2z_1\mathcal{H}([\theta(e_{\alpha_2}),e_{\alpha_1}])/2)) \\
&  +\coth((\alpha_1-\alpha_2)(a))\Im(-e^{2\alpha_2(a)}\bar{z}_2z_1\mathcal{H}([\theta(e_{\alpha_2}),e_{\alpha_1}])/2)). 
\end{align*}
As a consequence, 
\begin{align*}
\Re(C_E) = & O -e^{2\alpha_1(a)}\Re(\bar{z}_1z_2\mathcal{H}([\theta(e_{\alpha_1}),e_{\alpha_2}]))/2 \\
& -e^{2\alpha_2(a)}\Re(z_1\bar{z}_2\mathcal{H}([\theta(e_{\alpha_2}),e_{\alpha_1}]))/2 \\
&  +\tanh((\alpha_2-\alpha_1)(a))\Re(-e^{2\alpha_1(a)}\bar{z}_1z_2\mathcal{H}([\theta(e_{\alpha_1}),e_{\alpha_2}])/2)) \\
&  +\tanh((\alpha_1-\alpha_2)(a))\Re(-e^{2\alpha_2(a)}\bar{z}_2z_1\mathcal{H}([\theta(e_{\alpha_2}),e_{\alpha_1}])/2)). 
\end{align*}
We then check by computation that 
\begin{align*} 
\Omega_{\alpha_1,\bar{\alpha}_2} & = 
\frac{1}{2}\Big{(}
e^{2\alpha_2(a)}(1+\tanh((\alpha_1-\alpha_2)(a))) \\ 
& \qquad 
+e^{2\alpha_1(a)}(1+\tanh((\alpha_2-\alpha_1)(a))) 
\Big{)}
\chi\circ\mathcal{H}([\theta(e_{\alpha_2}),e_{\alpha_1}])  \\
& = \frac{ 2 \chi([\theta(e_{\alpha_2},e_{\alpha_1}])}{(e^{-2\alpha_1(a)}+e^{-2\alpha_2(a)})}.
\end{align*}

3.iv) 
If $\alpha_2-\alpha_1 \in \Phi\setminus \Phi_L$, say 
$\alpha_1-\alpha_2\in \Phi_{P^u}$ for example, then $B_E=O$ 
and 
\begin{align*}
C_E & = O + \frac{-e^{2\alpha_2(a)}}{2}\bar{z}_2z_1[\theta(e_{\alpha_2}),e_{\alpha_1}] \\
& \quad - e^{2(\alpha_1-\alpha_2)(a)}\theta(\frac{-e^{2\alpha_1(a)}}{2}\bar{z}_1z_2[\theta(e_{\alpha_1}),e_{\alpha_2}]).
\end{align*}
Since $[\theta(e_{\alpha_2}),e_{\alpha_1}]$ is in the Lie algebra of the 
unipotent radical of $H$ we have $\chi([\theta(e_{\alpha_2}),e_{\alpha_1}])=0$
hence 
\[
\Omega_{\alpha_1,\bar{\alpha}_2}=0.
\]

3.v) 
Finally, if $\alpha_2-\alpha_1 \notin \Phi$, then we have 
$[\theta(e_{\alpha_1}),e_{\alpha_2}] = [\theta(e_{\alpha_2}),e_{\alpha_1}] =0$
hence $B_E=O$ and $C_E=O$, and we deduce 
\[
\Omega_{\alpha_1,\bar{\alpha}_2}=0.
\]

4) 
Consider now the case $D=z_1l_j+z_2\tau_{\beta}$.
Then $B_D=\Re(z_1l_j)$ and 
\[
C_D=\tanh(\beta(a))\Re(z_2\mu_{\beta}) +
\coth(\beta(a))\Im(z_2\mu_{\beta}).
\]
We compute 
\[
[B_D,D]=O+\bar{z}_1z_2\beta(l_j)\mu_{\beta}/2
\]
\begin{align*}
[C_D,B_D]=& O+\frac{z_1\bar{z}_2}{4}\beta(l_j)(\coth(\beta(a))-\tanh(\beta(a)))\theta(\tau_{\beta}) \\
	& -\frac{\bar{z}_1z_2}{4}\beta(l_j)(\tanh(\beta(a))+\coth(\beta(a))\tau_{\beta}
\end{align*}
and 
\[
[C_D,D]=O+\frac{z_1\bar{z}_2}{2}\beta(l_j)(\coth(\beta(a)-\tanh(\beta(a))\theta(\tau_{\beta}).
\]
From these computations we deduce 
\begin{align*}
E = & O + \bar{z}_1z_2 \frac{\beta(l_j)}{2}(\mu_{\beta}-\frac{\tanh(\beta(a)+\coth(\beta(a))}{2}\tau_{\beta}) \\
& +z_1\bar{z}_2\frac{3\beta(l_j)}{4}(\coth(\beta(a))-\tanh(\beta(a)))\theta(\tau_{\beta})).
\end{align*}
We then have $B_E=O$ and 
\begin{align*}
\Re(C_E) & = O+\frac{\beta(l_j)}{2} \Re(\bar{z}_1z_2\mu_{\beta}) \\
& \quad + \tanh(\beta(a))\Re(\bar{z}_1z_2\frac{-\beta(l_j)}{4}(\tanh(\beta(a))+\coth(\beta(a)))\mu_{\beta}) \\
& \quad + \tanh(-\beta(a))\Re(z_1\bar{z}_2\frac{3\beta(l_j)}{4}(\coth(\beta(a))-\tanh(\beta(a)))\theta(\mu_{\beta})) \\
& = O+\beta(l_j)(1-\tanh^2(\beta(a)))\Re(\bar{z}_1z_2\mu_{\beta}) 
\end{align*}
since $\Re(z_1\bar{z}_2\theta(\mu_{\beta}))=-\Re(\theta(z_1\bar{z}_2\theta(\mu_{\beta}))=-\Re(\bar{z}_1z_2\mu_{\beta})$.
We may thus finish the computation to obtain
\[ 
\Omega_{j,\bar{\beta}} = \beta(l_j)(1-\tanh^2(\beta(a))) \chi(\theta(\mu_{\beta})).
\]

5) 
Consider the case $D=z_1\tau_{\beta}+z_2e_{\alpha}$.
Then we have $B_D=0$ and  
\[
C_D=-e^{2\alpha(a)}\theta(z_2e_{\alpha})+
\tanh(\beta(a))\Re(z_1\mu_{\beta}) +
\coth(\beta(a))\Im(z_1\mu_{\beta}).
\]
Then $E=[C_D,D]/2$, 
which is equal to
\[
\frac{z_1\bar{z}_2}{2}e^{2\alpha(a)}[\tau_{\beta},\theta(e_{\alpha})]
+\frac{-\bar{z}_1z_2}{4}(\coth(\beta(a))-\tanh(\beta(a)))[e_{\alpha},\theta(\mu_{\beta})] +O.
\]
We then remark that 
$[\tau_{\beta},\theta(e_{\alpha})]\in 
\mathfrak{g}_{-\alpha+\beta}\oplus\mathfrak{g}_{-\alpha+\sigma(\beta)}$
and 
$[e_{\alpha},\theta(\mu_{\beta})]\in 
\mathfrak{g}_{\alpha-\beta}\oplus\mathfrak{g}_{\alpha-\sigma(\beta)}$.
It is impossible for $-\alpha+\beta$ as well as for 
$-\alpha+\sigma(\beta)$ to lie in $\Phi_L$ (write these roots in a basis of 
simple roots adapted to $P$ to check this assertion).
Hence we obtain that $C_E$ is a sum of a negligible term and 
a term in 
$\mathfrak{g}_{-\alpha+\beta}\oplus\mathfrak{g}_{-\alpha+\sigma(\beta)}
\oplus \mathfrak{g}_{\alpha-\beta)}\oplus\mathfrak{g}_{\alpha-\sigma(\beta)} \cap \mathfrak{h}$.
Since this last space is contained in the Lie algebra of the unipotent 
radical of $H$, we obtain that $\chi\circ\Re(C_E)$ is negligible,
hence
\[
\Omega_{\beta,\bar{\alpha}}=0.
\]

6) 
Consider finally the case $D=z_1\tau_{\beta_1}+z_2\tau_{\beta_2}$.
Then we have $B_D=0$ and  
\begin{align*}
C_D = &
\tanh(\beta_1(a))\Re(z_1\mu_{\beta_1}) +
\coth(\beta_1(a))\Im(z_1\mu_{\beta_1}) \\
& +\tanh(\beta_2(a))\Re(z_2\mu_{\beta_2}) +
\coth(\beta_2(a))\Im(z_2\mu_{\beta_2}).
\end{align*}
Then in view of the relation $\coth(x)-\tanh(x)=2/\sinh(2x)$, we have 
\[E= 
\frac{\bar{z}_2z_1[\theta(\mu_{\beta_2}),\tau_{\beta_1}]}{2\sinh(2\beta_2(a))} 
 +\frac{\bar{z}_1z_2[\theta(\mu_{\beta_1}),\tau_{\beta_2}]}{2\sinh(2\beta_1(a))} +O
\]
We separate in two cases the end of the computation.

6.i) 
If $\beta_1\neq \beta_2$ then, note that for $1\leq j\neq k\leq 2$, 
we have 
\[
[\theta(\mu_{\beta_k}),\tau_{\beta_j}]=2\mathcal{P}([\theta(e_{\beta_k}),e_{\beta_j}])+
2\mathcal{P}([\theta\sigma(e_{\beta_k}),e_{\beta_j}])
\]
where $[e_{\beta_j},\theta(e_{\beta_k})]\in \mathfrak{g}_{\beta_j-\beta_k}$
and $[e_{\beta_j},\theta\sigma(e_{\beta_k})]\in\mathfrak{g}_{\beta_j-\sigma(\beta_k)}$.
It shows that $B_E$ is negligible, and that $\Re(C_E)$ is equal to 
the sum of a negligible term and 
\begin{align*}
&\frac{\tanh(\beta_2(a)-\beta_1(a))}{\sinh(2\beta_1(a))}\Re(\bar{z}_1z_2\mathcal{H}([\theta(e_{\beta_1}),e_{\beta_2}])) \\
&+ \frac{\tanh(\beta_2(a)-\sigma(\beta_1)(a))}{\sinh(2\beta_1(a))}\Re(\bar{z}_1z_2\mathcal{H}([\theta\sigma(e_{\beta_1}),e_{\beta_2}])) \\
&+\frac{\tanh(\beta_1(a)-\beta_2(a))}{\sinh(2\beta_2(a))}\Re(\bar{z}_2z_1\mathcal{H}([\theta(e_{\beta_2}),e_{\beta_1}])) \\
&+ \frac{\tanh(\beta_1(a)-\sigma(\beta_2)(a))}{\sinh(2\beta_2(a))}\Re(\bar{z}_2z_1\mathcal{H}([\theta\sigma(e_{\beta_2}),e_{\beta_1}])).
\end{align*}
We may rewrite this as 
\begin{align*}
&\tanh(\beta_2-\beta_1)(\frac{1}{\sinh(2\beta_1(a)}-\frac{1}{\sinh(2\beta_2(a)})\Re(\bar{z}_1z_2\mathcal{H}([\theta(e_{\beta_1}),e_{\beta_2}])) \\
&+ \tanh(\beta_1+\beta_2)(\frac{1}{\sinh(2\beta_1(a)}+\frac{1}{\sinh(2\beta_2(a)})\Re(\bar{z}_1z_2\mathcal{H}([\theta\sigma(e_{\beta_1}),e_{\beta_2}])).
\end{align*}

We compute now 
\begin{align*}
\Omega_{\beta_1,\bar{\beta}_2} & =
\frac{1}{2}\tanh(\beta_2(a)-\beta_1(a))(\frac{1}{\sinh(2\beta_1(a)}-\frac{1}{\sinh(2\beta_2(a)})\chi([\theta(e_{\beta_2}),e_{\beta_1}]) \\
& \quad + \frac{1}{2}\tanh(\beta_1+\beta_2)(\frac{1}{\sinh(2\beta_1(a)}+\frac{1}{\sinh(2\beta_2(a)})\chi([\theta(e_{\beta_2}),\sigma(e_{\beta_1})]) 
\end{align*}

6.ii) 
If $\beta_1 = \beta_2 = \beta$ then we have
\[
E=\frac{z_1\bar{z}_2+\bar{z}_1z_2}{\sinh(2\beta)}(\mathcal{P}(\beta^{\vee})+\mathcal{P}([\theta\sigma(e_{\beta}),e_{\beta}]))
\]
and thus 
\[
B_E = \frac{z_1\bar{z}_2+\bar{z}_1z_2}{\sinh(2\beta)}\mathcal{P}(\beta^{\vee})
\] 
and 
\begin{align*}
\Re(C_E)&=\frac{z_1\bar{z}_2+\bar{z}_1z_2}{\sinh(2\beta)}\tanh(\beta-\sigma(\beta))\Re\circ\mathcal{H}([\theta\sigma(e_{\beta}),e_{\beta}]) \\
&=\frac{z_1\bar{z}_2+\bar{z}_1z_2}{\cosh(2\beta)}\mathcal{H}\circ\Re([\theta\sigma(e_{\beta}),e_{\beta}]).
\end{align*} 
Hence 
\begin{align*}
\Omega_{\beta,\bar{\beta}}&= \frac{d_au(\beta^{\vee})}{\sinh(2\beta)} 
-\frac{2}{\cosh(2\beta)}\chi\circ\Re([\theta\sigma(e_{\beta}),e_{\beta}]).
\end{align*}

\section{Horosymmetric varieties}

\label{sec_varieties}

We move on to introduce horosymmetric varieties. We provide several examples and 
present the classification theory inherited from that of spherical varieties.
We then check the property that a $G$-invariant irreducible codimension one 
subvariety in a horosymmetric variety is still horosymmetric. 

\subsection{Definition and examples}

\begin{defn}
A \emph{horosymmetric variety} $X$ is a normal $G$-variety such that 
$G$ acts with an open dense orbit which is a \emph{horosymmetric homogeneous 
space}. 
\end{defn}

\begin{exa} 
By Example~\ref{exa_horospherical}, any 
\emph{horospherical variety} (see \cite{Pas08}) 
may be considered as a horosymmetric variety. 
It includes in particular generalized flag manifolds, toric varieties 
and homogeneous toric bundles.  
\end{exa}

\begin{exa}
\label{exa_P2_horospherical}
Consider the projective plane $\mathbb{P}^2$ equipped with the action 
of $\mathrm{SL}_2$ or $\mathrm{GL}_2$ induced by a choice of affine chart 
$\mathbb{C}^2$ in $\mathbb{P}^2$. There are three orbits under this action: 
the fixed point $\{0\}$, the open dense orbit $\mathbb{C}^2\setminus \{0\}$ 
and the projective line at infinity $\mathbb{P}^1$. The $\mathrm{GL}_2$-variety 
$\mathbb{P}^2$ is hence a horospherical variety by Example~\ref{exa_C2minus0}. 
We may further consider the blow up of $\mathbb{P}^2$ at the fixed point 
$\{0\}$ and lift the action of $\mathrm{GL}_2$ to check that this blow up is also
a horospherical variety.  
More generally, Hirzebruch surfaces have structures of $\mathrm{GL}_2$-horospherical 
varieties refining their toric structure.
\end{exa}

Assume $G=L$ is semisimple and $H=N_G(G^{\sigma})$.
Then the \emph{wonderful compactifications} of $G/H$ constructed by De Concini 
and Procesi \cite{DCP83} is a horosymmetric variety.
It is a particularly nice compactification of $G/H$ caracterized by 
the following properties. 
Let $r$ denote the rank of $G/H$ and set $I=\{1,\ldots,r\}$.

\begin{thm}[{\cite{DCP83}}]
\label{thm_dcp83}
The wonderful compactification $X$ of $G/H$ is the unique smooth $G$-equivariant 
compactification of $X$ such that:
\begin{itemize} 
\item $G$-orbit closures $X_J$ in $X$ are in bijection with subsets $J\subset I$ and 
\item all $X_J$ are smooth and intersect transversely, with $X_{J}=\bigcap_{j\in J} X_{\{j\}}$.
\end{itemize}  
Furthermore, for each $J$, there exists a parabolic subgroup $P_J$ of $G$, with $\sigma$-stable 
semisimple Levi factor $L'_J$, and an equivariant fibration $X_J\rightarrow G/P_J$ 
with fiber the wonderful compactification of $L'_J/N_{L'_J}((L'_J)^{\sigma})$.
\end{thm}

\begin{exa}
\label{exa_wonderful_AIII}
Consider the symmetric space $G/H$ of type AIII($2$, $m>4$).
Recall from Example~\ref{exa_normalizer_AIII} that $H=N_G(H)$. Using the description of $G/H$ as a dense 
orbit in the product of Grassmannians 
$X_0=\mathrm{Gr}_{2,m}\times \mathrm{Gr}_{m-2,m}$ as in Example~\ref{exa_invol_sl},
we obtain a first example of (horo)symmetric variety with open orbit 
$G/H$: this product of Grassmannians $X_0$ itself. 
It contains three orbits under the action of $G=\mathrm{SL}_m$: the dense 
orbit of pairs of linear subspaces in direct sum, the closed 
orbit of flags $(V_1,V_2)$ with $V_1\subset V_2$, and the codimension one orbit 
of pairs $(V_1,V_2)$ with $\mathrm{dim}(V_1\cap V_2)=1$.
One can blow up $X_0$ along the closed orbit to obtain another 
$G$-equivariant compactification $X$ of $G/H$. This new 
compactification $X$ is none other than the wonderful compactification 
of $G/H$. 
\end{exa}

In higher ranks, one may still obtain the wonderful compactifications 
from the product of Grassmannians, but this requires a more 
involved sequence of blow-ups.

From Theorem~\ref{thm_dcp83}, one sees the first 
examples of horosymmetric varieties that are neither horospherical 
nor symmetric. Indeed, the description of orbits in a wondeful 
compactification show that orbit closures in wonderful compactifications 
are all horosymmetric, with the only symmetric being the open orbit and 
the only horospherical being the closed one (actually a generalized flag 
manifold).

Since it applies only to $H=N_G(G^{\sigma})$, the construction of 
De Concini and Procesi does not exhaust the compactifications of 
symmetric spaces satisfying the properties of Theorem~\ref{thm_dcp83}, 
still called wonderful compactifications. 
The simplest example of wonderful compactification which is not 
in the examples studied by De Concini and Procesi is the following.
\begin{exa}
\label{exa_wond_P1xP1}
Consider the variety $\mathbb{P}^1\times \mathbb{P}^1$ equipped 
with the diagonal action of $\mathrm{SL}_2$. There are two orbits under 
this action: the diagonal embedding of $\mathbb{P}^1$ and its 
complement. The complement is open dense and isomorphic to the 
symmetric space $\mathrm{SL}_2/T$ where $T$ is a maximal torus of 
$\mathrm{SL}_2$. 

The wonderful compactification constructed by De Concini and Procesi 
and corresponding to this involution on the other hand is $\mathbb{P}^2$ 
seen as a compactification of $\mathrm{SL}_2/N_{\mathrm{SL}_2}(T)$ by 
a adding a quadric.
\end{exa}

\begin{exa}
Several papers expanded results valid on the wonderful compactification 
of a symmetric space to so-called \emph{complete symmetric varieties}
(see \emph{e.g.} \cite{DCP85,Bif90}), 
that is, smooth $G$-equivariant compactifications of $G/H$ that 
dominate the wonderful compactification. 
Any complete symmetric variety is a horosymmetric variety.
\end{exa}

\subsection{Combinatorial description of horosymmetric varieties}

\subsubsection{Colored fans}

As spherical varieties, horosymmetric varieties with open orbit $G/H$ 
are classified by colored fans for the spherical homogeneous space $G/H$, 
which are defined in terms of the combinatorial data $\mathcal{M}$, 
$\mathcal{V}$ and $\rho:\mathcal{D}\rightarrow \mathcal{N}$ 
(Recall these were described in Proposition~\ref{prop_colored_data}). 

\begin{defn}\mbox{}
\begin{itemize}
\item A \emph{colored cone} is a pair $(\mathcal{C},\mathcal{R})$, where 
$\mathcal{R}\subset \mathcal{D}$, $0\notin \rho(\mathcal{R})$, and 
$\mathcal{C}\subset \mathcal{N} \otimes \mathbb{Q}$ is a strictly convex cone 
generated by $\rho(\mathcal{R})$ and finitely many elements of $\mathcal{V}$
such that the intersection of the relative interior of $\mathcal{C}$ with 
$\mathcal{V}$ is not empty.
\item Given two colored cones $(\mathcal{C},\mathcal{R})$ and 
$(\mathcal{C}_0,\mathcal{R}_0)$, we say that $(\mathcal{C}_0,\mathcal{R}_0)$ is a 
\emph{face} of $(\mathcal{C},\mathcal{R})$ if $\mathcal{C}_0$ is a face of 
$\mathcal{C}$ and $\mathcal{R}_0=\mathcal{R}\cap \rho^{-1}(\mathcal{C}_0)$.
\item A \emph{colored fan} is a non-empty finite set $\mathcal{F}$ of colored cones 
such that the face of any colored cone in $\mathcal{F}$ is still in $\mathcal{F}$, 
and any $v\in \mathcal{V}$ is in the relative interior of at most one cone.
\end{itemize}
\end{defn}

An equivariant \emph{embedding} $(X,x)$ of $G/H$ is the data of a horosymmetric variety $X$ 
and a base point $x\in X$ such that $\overline{G\cdot x}=X$ and $\mathrm{Stab}_G(x)=H$. 

\begin{thm}[{\cite[Theorem 3.3 and Theorem 4.2]{Kno91}}]
\label{thm_class}
There is a bijection $(X,x) \longmapsto \mathcal{F}_X$ between embeddings of $G/H$ up to 
$G$-equivariant isomorphism and colored fans. 
There is a bijection $Y\mapsto (\mathcal{C}_Y, \mathcal{R}_Y)$ 
between the orbits of $G$ in $X$, and the colored cones in $\mathcal{F}_X$. 
An orbit $Y$ is in the closure of another orbit $Z$ in $X$ if and only if 
the colored cone $(\mathcal{C}_Z, \mathcal{R}_Z)$ is a face of 
$(\mathcal{C}_Y, \mathcal{R}_Y)$.
The variety $X$ is complete if and only if the support 
$|\mathcal{F}_X|=\bigcup_{(\mathcal{C},\mathcal{R})\in\mathcal{F}_X}\mathcal{C}$ 
contains the valuation cone $\mathcal{V}$.
\end{thm}

\begin{exa}
\label{exa_wond_colored_fan}
Assume $G$ is a semisimple group, and $H$ is a symmetric subgroup of $G$. 
Then there is a natural choice of colored fan given by the negative 
Weyl chamber and its faces. 
If $H=N_G(H)$ then the corresponding variety is the wonderful compactification 
of $G/H$.   
\end{exa}

More generally if the valuation cone is strictly convex, then the embedding 
corresponding to the colored fan given by the valuation cone and its faces 
is called \emph{wonderful} it it is in addition smooth.
There are criterions of smoothness for spherical varieties \cite{Bri91}, and 
some simpler criterions for the case of horospherical \cite{Pas08}
and symmetric \cite[Section 3]{Ruz11} varieties. It would certainly be possible and 
useful to derive such a simpler criterion for the class of horosymmetric 
varieties. 
In the case of toroidal horosymmetric varieties, which are introduced in the 
next section, the criterion is very simple, as it is the case for toroidal 
spherical varieties in general. 

\subsubsection{Toroidal horosymmetric varieties}

Given an embedding $(X,x)$ of $G/H$ we denote by $\mathcal{F}_X$ its 
colored fan and we call the elements of $\mathcal{D}(X)= 
\bigcup_{(\mathcal{C},\mathcal{R})\in \mathcal{F}}\mathcal{R}
\subset \mathcal{D}$
the \emph{colors of} $X$. 
It should be noted that the set of colors does not depend on the base 
point $x$, but $\mathcal{F}_X$ does. We however omit this dependence in the
notation.

\begin{defn}
An embedding is \emph{toroidal} if $\mathcal{D}(X)$ is empty, else 
it is \emph{colored}.
\end{defn}

A toroidal horosymmetric variety is globally a parabolic induction from 
a symmetric variety. More generally, we record the following elementary 
statement, easily seen by the classification of 
horosymmetric varieties by colored fan, and the fact that the colored 
fan of a parabolic induction is the same as the colored of the 
embedding one starts with. 

\begin{prop}
\label{prop_criterion_parabolic_induction}
A horosymmetric variety with set of colors $\mathcal{D}_X$ is globally 
a parabolic induction from a symmetric variety if and 
only if $\mathcal{D}_X\cap f^{-1}\mathcal{D}(G/P)=\emptyset$.
\end{prop}

\begin{exa}
The horospherical variety $\mathbb{P}^2$ under the action of $\mathrm{SL}_2$ is 
not a global parabolic induction (in particular it is not toroidal), 
but the blow up of $\mathbb{P}^2$ is. 
\end{exa}

The following result shows that, in a toroidal horosymmetric 
variety, there is a well identified toric subvariety which 
will play an important role in later applications.

\begin{prop}[{\cite[Corollary 8.3 and paragraph after Corollary 6.3]{Kno94}}]
\label{prop_toric_subvar}
Let $(X,x)$ be a toroidal embedding of the horosymmetric space $G/H$, 
with colored fan $\mathcal{F}_X$.  
Then the closure $Z$ of $T\cdot x$ in $X$ is the $T/T\cap H$-toric 
variety whose fan (as a spherical variety) 
consists of the images, by elements of the restricted Weyl group $\bar{W}$, 
of the cones $\mathcal{C}$ in the colored cone $\mathcal{F}_X$.
\end{prop}

\begin{rem}
We insist here that we obtain the fan of $Z$ as a spherical variety 
under the action of $T/T\cap H$. It does not exactly coincide in 
general with the fan of $Z$ as a toric variety with the classical conventions, 
but to the opposite of this fan. We refer to Pezzini \cite[Section 2]{Pez10}
for details, but the short explanation is that a character $\lambda$
of a torus $T$ may be interpreted as a regular $B=T$-semi-invariant 
function on $T$ with weight $-\lambda$:
$\lambda(b^{-1}t)=(-\lambda)(b)\lambda(t)$.
This difference is important to get the correct expression for the 
asymptotic behavior of metrics in Section~\ref{sec_metrics}. 
This fact was overlooked in previous work of the author, 
fortunately with no serious consequences. 
\end{rem}

\begin{exa}
Assume $X$ is the wonderful compactification of a symmetric space 
then the fan of its toric subvariety $Z$ is the fan obtained by considering 
the collection of all restricted Weyl chambers for $\bar{\Phi}$ and their 
faces.  
\end{exa}

Finally, let us mention the 
criterion of smoothness for toroidal horosymmetric varieties:
\begin{prop}
\label{prop_smoothness}
A toroidal horosymmetric embedding $(X,x)$ is smooth if and only if the 
toric subvariety $Z$ is smooth, that is, if and only if every cone in the 
colored fan is generated by a subset of a basis of $\mathcal{N}$.  
\end{prop}

\subsection{Facets of a horosymmetric variety}

\label{subsec_facets}

\begin{defn}
Let $X$ be a horosymmetric variety under the action of $G$. A 
\emph{facet} of $X$ is a $G$-stable irreducible codimension one subvariety in $X$.
\end{defn}
The goal of this section is to prove the following result.
\begin{prop}
\label{prop_facets}
Let $X$ be a horosymmetric variety under the action of $G$, then 
any facet of $X$ is also a horosymmetric 
variety under the action of $G$.
\end{prop}
We will actually obtain more precise statements describing the corresponding 
horosymmetric homogeneous spaces, using \cite{Bri90}. 
Let us first introduce some terminology. 

\begin{defn}
An \emph{elementary embedding} $(E,x)$ of $G/H$ is an embedding such 
that the complement of $G/H$ in $E$ is a single codimension one orbit 
of $G$. Equivalently, it is an embedding whose colored fan consists 
of a single ray $\mathcal{C}_E\subset -\mathfrak{a}_s^+$ with no colors. 
\end{defn}

Elementary embeddings are in bijection with indivisible one parameter subgroups 
in $-\mathfrak{a}_s^+\cap \mathfrak{Y}(T_s)$ by selecting the only such one parameter subgroup 
in the ray associated to the elementary embedding. Given an indivisible 
$\mu\in -\mathfrak{a}_s^+\cap \mathfrak{Y}(T_s)$ we denote by $(E_{\mu},x)$ the corresponding elementary 
embedding.
Furthermore 
$x_{\mu}:=\lim_{z\rightarrow 0} \mu(z)\cdot x$ exists in $E_{\mu}$ and is in 
the open $B$-orbit of the codimension one $G$-orbit \cite[Section~2.10]{BP87}. 
We will use \cite{Bri90} to describe the Lie algebra of the isotropy subgroup 
of $x_{\mu}$. 

We fixed since Section~\ref{sec_homogeneous}
a maximal torus $T$ of $G$ and a Borel subgroup $B$ containing $T$. 
Recall that parabolic subgroups containing $B$ are classified by subsets 
of the set of simple roots $S$. More precisely, given a subset $I\subset S$, 
there is a unique parabolic subgroup $Q_I$ of $G$ containing $B$ such that 
$\Phi_{Q_I^u}\cap S=S\setminus I$. It further contains a unique Levi subgroup $L_I$
containing $T$, and $\Phi_{L_I}\cap S=I$. We denote by $P_I$ the 
parabolic subgroup opposite to $Q_I$ with respect to $L_I$.

The subgroup $B\cap L$ is a Borel subgroup of $L$ containing $T$, and we have 
the same correspondence between subsets $I$ of $S_L$ and parabolic 
subgroups $Q_I^L$ of $L$ containing $B\cap L$. We have the obvious 
relation $Q_I^L=Q_I\cap L$, and all of these parabolics are contained 
in $Q_{S_L}=Q$. The Levi subgroup of $Q^L_I$ containing $T$ is none 
other than $L_I$. 

Given a one parameter subgroup $\mu\in \mathfrak{Y}(T)$, we obtain a subset 
of $S_L$ by setting $I(\mu)=\{\alpha\in S_L; \alpha(\mu)=0\}$. 
Then the 
Levi subgroup $L_I$ of $Q^L_{I(\mu)}$ containing $T$ coincides with 
the centralizer $Z_L(\mu)$ of $\mu$. 
In particular, $\mu$ is contained in the radical of the Levi subgroup $L_{I(\mu)}$.
Pay attention to the fact that $Z_L(\mu)$ may be different from $Z_G(\mu)$ here.

Take $\mu\in \mathfrak{Y}(T_s)$. Then the restriction of $\sigma$ to 
$L_{I(\mu)}=Z_L(\mu)$ is still a well defined involution. 
Since $\mu$ is contained in the radical of $L_{I(\mu)}$, we may 
choose a $\sigma$-stable complement $\mathfrak{m}$ of $\mathbb{C}\mu$ 
in $\mathfrak{l}_{I(\mu)}$ which contains the derived Lie algebra 
$[\mathfrak{l}_{I(\mu)},\mathfrak{l}_{I(\mu)}]$.
Define a new involution $\sigma_{\mu}$ on $L_{I(\mu)}$ by setting, 
at the level of Lie algebras, 
\[
\sigma_{\mu}(z\mu+m)=z\mu+\sigma(m).
\]
This is a well defined involution of Lie algebras thanks to our choice of 
complement $\mathfrak{m}$.

The following proposition provides a more precise version of  
Proposition~\ref{prop_facets}. Indeed, given a facet $Y$ of $X$, the union of 
the open $G$-orbits in $X$ and $Y$ form an elementary embedding of $G/H$.

\begin{prop}
\label{prop_isotropy_facets}
Let $\mu\in \mathfrak{Y}(T_s)$ indivisible. Then the isotropy subgroup 
$H_{\mu}$ of $x_{\mu}$ is horosymmetric as follows:
\[
\mathfrak{h}_{\mu} = \mathfrak{p}_{I(\mu)}^u \oplus 
\mathfrak{l}_{I(\mu)}^{\sigma_{\mu}}.
\]
\end{prop}

\begin{proof}
An elementary embedding is toroidal, hence, 
by Proposition~\ref{prop_criterion_parabolic_induction}, it is a global parabolic induction 
$E_L \hookrightarrow E_{\mu}\rightarrow G/P$ where $(E_L,x)$ is the elementary 
embedding of $L/L\cap H$ associated with the same one parameter subgroup $\mu$.
Since $x_{\mu} \in E_L$, we obtain that $P^u\subset H_{\mu}\subset P$ and 
$H_{\mu}$ is determined by $L\cap H_{\mu}$. 

We now use the results of \cite{Bri90} applied to the case of symmetric spaces 
to obtain a description of $\mathfrak{l}\cap \mathfrak{h}_{\mu}$.
By Proposition~2.4 in \cite{Bri90} and the remarks about the case of 
symmetric spaces in Section~2.2 of the same paper, we have 
\[
\mathfrak{l}\cap \mathfrak{h}_{\mu} = 
\mathbb{C}\mu\oplus \mathfrak{t}^{\sigma} \oplus 
\bigoplus_{\alpha\in \Phi_L^{\sigma}} \mathbb{C}e_{\alpha}
\oplus \bigoplus_{\alpha\in \Phi_s^+; \bar{\alpha}(\mu)=0}\mathbb{C}(e_{-\alpha}+\sigma(e_{-\alpha}))
\oplus \bigoplus_{\alpha\in \Phi_s^+; \bar{\alpha}(\mu)\neq 0}\mathbb{C}e_{-\alpha}.
\]
Since $\mu\in \mathfrak{Y}(T_s)$, we have $\bar{\alpha}(\mu)=2\alpha(\mu)$, so we may write 
the above expression as 
\[
\mathfrak{l}\cap \mathfrak{h}_{\mu} = 
\mathfrak{l}_{I(\mu)}^{\sigma_{\mu}}\oplus \bigoplus_{\alpha\in \Phi_{P_{I(\mu)}^u}\cap \Phi_L}\mathfrak{g}_{\alpha}
\]
Putting both results together, we get 
\[
\mathfrak{l} = 
\mathfrak{l}_{I(\mu)}^{\sigma_{\mu}}\oplus \bigoplus_{\alpha\in \Phi_{P_{I(\mu)}^u}}\mathfrak{g}_{\alpha}
\]
hence the statement.
\end{proof}

\begin{exa}
\label{exa_facets_AIII}
Consider the symmetric space of type AIII($2$, $m>4$). 
Take $\mu \in -\mathfrak{a}_s^+$ an indivisible one parameter subgroup. Then 
there are three possibilities for $I(\mu)$: we have 
$I(\bar{\alpha}_{1,m}^{\vee})=S\setminus\{\alpha_{1,2},\alpha_{m-1,m}\}$, 
$I(\bar{\alpha}_{1,m-1}^{\vee})=S\setminus\{\alpha_{2,3},\alpha_{m-2,m-1}\}$, 
and $I(\mu)=S^{\sigma}$ in the other cases. 
In the first situation, $[\mathfrak{l}_{I(\mu)},\mathfrak{l}_{I(\mu)}]$ is 
isomorphic to the simple Lie algebra $\mathfrak{sl}_{m-2}$ and the induced 
involution is still of type AIII, but now of rank one. 
In the second situation, $[\mathfrak{l}_{I(\mu)},\mathfrak{l}_{I(\mu)}]$ splits 
as a direct sum of three summands, one fixed by $\sigma$ and the other two, 
each isomorphic to $\mathfrak{sl}_2$, exchanged by $\sigma$. 
Finally, in the third scenario, the isotropy group is in fact horospherical.
\end{exa}

\begin{rem}
It is in fact very general that if the ray is in the 
interior of the valuation cone, then the closed orbit in the corresponding 
elementary embedding is horospherical, with an explicit description of its 
Lie algebra \cite[Proposition 3.10]{BP87}.
\end{rem}

\begin{rem}
Gagliardi and Hofscheier \cite{GH15_hsd} obtained a description of the combinatorial 
data associated to any orbit in a spherical variety. We could in principle 
(neglecting the difficulty of describing the color map in general) have 
used this to show that facets of horosymmetric varieties are horosymmetric. 
However their result identifies only the conjugacy class of the isotropy subgroup, 
while we will need the precise knowledge of the isotropy group (actually Lie algebra)
of a specific point in the orbit. On the other hand, it is possible to identify 
precisely the isotropy group of the point we consider, and not only its Lie algebra, 
by combining our result with that of \cite{GH15_hsd}. 
This could be used to fully recover the description of orbit closures in 
wonderful compactifications given by De Concini and Procesi. 
We mention here, as it could be useful for other applications in Kähler geometry, 
that the work of Gagliardi and Hofscheier \cite{GH15_hsd} further allows to identify 
the colored fan of a facet of a symmetric variety. 
\end{rem}

\section{Linearized line bundles on horosymmetric varieties}

\label{sec_bundles}

In this section, we consider $G$-linearized line bundles on a 
$G$-horosymmetric variety $X$. We explain how to associate to 
such a line bundle $\mathcal{L}$ a priviledged $B$-invariant 
$\mathbb{Q}$-divisor, and several convex polytopes. 
For example, one can associate to $\mathcal{L}$ its (algebraic) 
moment polytope, and in the case $X$ is toroidal, the moment 
polytope of the restriction of $\mathcal{L}$ to the toric subvariety. 
We determine the relations between these polytopes, and illustrate 
these notions on some examples, including the anticanonical line bundle.

Note that if $G$ is simply connected, all line bundles on $X$ are 
$G$-linearized \cite{KKV89,KKLV89}. 
Else, if $\mathcal{L}$ is not linearized, there exists 
a tensor power $\mathcal{L}^k$ which is linearized. Finally, recall 
that any two linearizations of the same line bundle differ by a character 
of $G$.

\subsection{Special function associated to a linearized line bundle}

\label{subsec_special_function}

Let $X$ be a horosymmetric embedding of $G/H$.
Let $\mathcal{L}$ be a $G$-linearized line bundle on $X$. 
The action of $G$ on $\mathcal{L}$ induces an action of $G$ on 
meromorphic sections of $\mathcal{L}$, given explicitly by 
$(g\cdot s)(x)=g\cdot (s(g^{-1}\cdot x))$. Such a section is called 
$B$-semi-invariant if it is an eigenvector for the action of $B$, 
that is, there exists a character $\lambda$ of $B$ such that 
$b\cdot s=\lambda(b)s$.

A meromorphic section $s$ of $\mathcal{L}$ defines a Cartier divisor 
$D_s=\mathrm{div}(s)$ representing $\mathcal{L}$. If $s$ is 
$B$-semi-invariant, then $d_s$ is $B$-invariant.
There are two types of irreducible $B$-stable Weil divisor 
on $X$: the closure of colors $D\in \mathcal{D}$ of $G/H$, and the 
facets of $X$. 
Denote by $\mathcal{I}^G(X)$ the set of facets of $X$.
Since $D_s$ is by definition Cartier, it writes, by \cite[Proposition 3.1]{Bri89},
\[
D_s=\sum_{Y\in \mathcal{I}^G(X)} v_s(\mu_Y) Y + \sum_{D\in \mathcal{D}_X} v_s(\rho(D)) \overline{D} 
+\sum_{D\in \mathcal{D}\setminus \mathcal{D}_X}n_D \overline{D}
\]
for some integers $n_D$ and a piecewise linear integral function $v_s$ 
on the fan $\mathcal{F}_X$. 

In general there may not be any priviledged $B$-semi-invariant 
meromorphic section of $\mathcal{L}$. Given such a section $s$, 
with associated weight $\lambda_s \in \mathfrak{X}(B)$, the others 
are obtained by multiplying by eigenvectors $f$ for the action 
of $B$ on $\mathbb{C}(G/H)$, with eigenvalue an element 
$\lambda_f$ of $\mathcal{M}$.
Assume however that $\lambda_s|_{T_s}\in \mathfrak{X}(T_s)$ is 
induced by an element of $\mathcal{M}$ under the epimorphism 
$\mathfrak{X}(T_s)\rightarrow \mathfrak{X}(T/T\cap H)$. Then 
we can choose $f$ so that $(\lambda_s\lambda_f)|_{T_s}$ is trivial. 

\begin{defn}
Assume there exists a section $s_{\mathcal{L}}$ such that 
$\lambda_s|_{T_s}$ is trivial.
Then we call this section, 
well defined up to multiplicative scalar, 
the \emph{special section} of $\mathcal{L}$, and 
$D_{\mathcal{L}}=\mathrm{div}(s_{\mathcal{L}})$ the 
\emph{special divisor} representing $\mathcal{L}$.
\end{defn}

In the general case, we may still define $D_{\mathcal{L}}$ 
as a $\mathbb{Q}$-divisor. Indeed, let $s$ be a $B$-semi-invariant 
section of $\mathcal{L}$ with weight $\lambda$ and consider the 
$B$-semi-invariant section $s\otimes s$ of the $G$-linearized
line bundle $\mathcal{L} \otimes \mathcal{L}$. It has weight $2\lambda$, 
and since $T_s\cap H$ consists of elements of order two, the 
weight $2\lambda|_{T_s}$ is induced by an element of $\mathcal{M}$.
Thus the previous paragraph defines the special divisor 
$D_{\mathcal{L} \otimes \mathcal{L}}$.

\begin{defn}
The \emph{special divisor} of $\mathcal{L}$ is the $\mathbb{Q}$-divisor 
$D_{\mathcal{L}}=D_{\mathcal{L} \otimes \mathcal{L}}/2$.
\end{defn}

Using the fact that $D_{\mathcal{L} \otimes \mathcal{L}}$ is Cartier, 
we may write  
\[
D_{\mathcal{L}}=\sum_{Y\in \mathcal{I}^G(X)} v_{\mathcal{L}}(\mu_Y) Y 
+ \sum_{D\in \mathcal{D}_X} v_{\mathcal{L}}(\rho(D)) \overline{D} 
+\sum_{D\in \mathcal{D}\setminus \mathcal{D}_X}n_{\mathcal{L},D} \overline{D}.
\]
for some piecewise linear function on $\mathcal{F}_X$, such that 
$2v_{\mathcal{L}}$ is integral.

\begin{defn}
The function $v_{\mathcal{L}}$ will be referred to as the 
\emph{special function} associated to $\mathcal{L}$. 
\end{defn}

Remark that it takes non integral values if $\mathcal{L}$
admits no special section.
Note that all of these notions are relative to the choice of a 
Borel subgroup $B$.  

\begin{defn}
The unique $\bar{W}$-invariant function $v_{\mathcal{L}}^t$ defined 
by $v_{\mathcal{L}}^t = v_{\mathcal{L}}$ on $-\mathfrak{a}_s^+$, 
is the \emph{toric special function} of $\mathfrak{L}$.
\end{defn}

\subsection{Toroidal case}

When $X$ is toroidal, we may identify the restriction of a $G$-linearized 
line bundle to the toric subvariety $Z$ in terms of the objects defined 
in Section~\ref{subsec_special_function}. 

\begin{prop}
\label{prop_restr_tor}
Let $\mathcal{L}$ be a $G$-linearized line bundle on a toroidal horosymmetric 
variety $X$. For a divisible enough integer $m$, the restriction  
$\mathcal{L}^m|_Z$ defines a $T/T\cap H\rtimes \bar{W}$-linearized toric 
line bundle on $Z$, such that the divisor associated with the 
$T/T\cap H\rtimes \bar{W}$-invariant section coincides with 
\[
\sum_{F\in \mathcal{I}^{T/T\cap H}(Z)} mv_{\mathcal{L}}^t(u_F)F.
\]  
\end{prop}

\begin{proof}
The restriction $\mathcal{L}|_Z$ inherits a $T_s$-linearization as well as a 
$N_K(T_s)\cap H$-linearization. 
The line bundle $\mathcal{L}^{2k}$ for some $k\geq 1$ admits a 
special section $s$.
By definition of a special section, $s$ is in particular $T_s\cap H$-invariant, 
hence $T_s\cap H$ acts trivially on $\mathcal{L}|_Z$.
Thus the $T_s$-linearization of $\mathcal{L}^{2k}$ actually comes from a 
$T/T\cap H$-linearization. 
The section $s|_Z$ is obviously $T/T\cap H$-invariant.

Let $n_w$ be a representant in $N_K(T_s)\cap H$ of $w\in \bar{W}$.
Then for any $t\in T_s$ 
\begin{align*}
n_w\cdot s(n_w^{-1}\cdot t\cdot x) &= n_w\cdot s(n_w^{-1}tn_w\cdot x) \\
\intertext{since $n_w\in H$,}
&= n_wn_w^{-1}tn_w\cdot s(x) \\
\intertext{since $n\in N_K(T_s)$ and $s$ is $T_s$-invariant,}
&= \chi(n_w)t\cdot s(x) \\
\intertext{where $\chi$ is the character of $H$ associated to $\mathcal{L}^{2k}$}
&= \chi(n_w) s(t\cdot x).
\end{align*}
Since there is a finite number of $n_w$ and they are in $K$, we may choose 
$k$ so that $\chi(n_w)=1$ for all $w\in \bar{W}$.

We now use the local structure of spherical varieties \cite[Proposition 3.4]{BP87}. 
Consider $\Delta=\bigcup_{D\in\mathcal{D}}D\subset G/H$ and set 
$U=X\setminus \bar{\Delta}$ and $V=Z\cap U$. 
Then $V$ is the toric subvariety associated to the subfan contained 
in $-\mathfrak{a}_s^+$, and the toric divisors in $V$ are precisely the 
$Y\cap V$ for $Y\in \mathcal{I}^G_X$.
By \cite[Section 3.2]{Bri89}, the restriction of $d_{\mathcal{L}^{2k}}$ 
to $V$ is 
\[
\mathrm{div}(s)\cap V= \sum_{Y\in \mathcal{I}^G(X)} 2k v_{\mathcal{L}}(\mu_Y)(Y\cap V).
\]
By $n_w$-invariance of $s$, we obtain 
\[
\mathrm{div}(s)\cap w\cdot V = \sum_{Y\in \mathcal{I}^G(X)} 2k v_{\mathcal{L}}(\mu_Y)w\cdot(Y\cap V).
\]
We deduce that 
\[
\mathrm{div}(s|_Z) = \sum_{F\in \mathcal{I}^{T/T\cap H}(Z)} 2k v_{\mathcal{L}}^t(u_F)F.
\] 
Remark that our reasoning with the representants $n_w$ did not endow $\mathcal{L}^{2k}|_Z$
with a $\bar{W}$-linearization \emph{a priori}. 
However, $s|_Z$ is the (up to multiplicative scalar) $T/T\cap H$-invariant section of $\mathcal{L}^{2k}|_Z$
and $\mathrm{div}(s|_Z)$ is $\bar{W}$-invariant, hence $\mathcal{L}^{2k}|_Z$ admits a 
natural $\bar{W}$-linearization such that $w\cdot s|_Z= \mu(w) s|_Z$ for all $w\in \bar{W}$ and 
some character $\mu:\bar{W}\rightarrow \mathbb{S}^1$ of $\bar{W}$. The group $\bar{W}$ being 
finite, we may take a multiple $m$ of $2k$ such that the character $\frac{m}{2k}\mu$ is trivial, 
and obtain that the $T/T\cap H$-invariant section is also $\bar{W}$-invariant. 
\end{proof}

\subsection{Polytopes}

To a $G$-linearized line bundle $\mathcal{L}$ on a complete horosymmetric variety $X$, 
we may associate several different convex 
polytopes. The first one is obtained directly from the special divisor of $\mathcal{L}$. 

\begin{defn}
The \emph{special polytope} $\Delta_{\mathcal{L}}$ of $\mathcal{L}$ is the convex polytope 
in $\mathcal{M}\otimes \mathbb{R}$ defined by the inequalities 
$m+v_{\mathcal{L}}\geq 0$, and  
$m(\rho(D))+n_{D,\mathcal{L}}\geq 0$ for all $D\in \mathcal{D}\setminus \mathcal{D}_X$.
\end{defn}

\begin{defn}
The \emph{toric polytope} $\Delta_{\mathcal{L}}$ of $\mathcal{L}$ is the convex polytope 
defined by 
\[
\Delta^t_{\mathcal{L}}=\{m\in \mathcal{M}\otimes \mathbb{R} ; 
m+v^t_{\mathcal{L}}\geq 0 \}.
\]
\end{defn}

Remark that the toric polytope is $\bar{W}$-invariant (and independent of the choice of a 
Borel subgroup $B$ containing $T$).

\begin{defn}
The \emph{moment polytope} $\Delta^+_{\mathcal{L}}$ is the set defined as the closure 
in $\mathfrak{X}(T)\otimes \mathbb{R}$ of the set of all 
$\sfrac{\lambda}{k}$ such that there exists a non-zero $B$-semi-invariant 
global holomorphic section $s$ of $\mathcal{L}^k$ with weight $\lambda$ 
(that is, $b\cdot s=\lambda(b)s$ for all $b\in B$).
\end{defn}

Note that all of these sets are multiplicative with respect to tensor powers, 
that is $\Delta^{\heartsuit}_{\mathcal{L}^m}=m\Delta^{\heartsuit}_{\mathcal{L}}$
for any positive integer $m$.

It was proved by Brion that the moment polytope is indeed a convex polytope. 
More precisely, we have the following relation between the special 
polytope and the moment polytope. 

\begin{prop}
\label{prop_moment_vs_special}
Let $\chi\in \mathfrak{X}(H)$ be the character associated to the restriction of 
$\mathcal{L}$ to $G/H$. Consider $\chi$ as before as an element of 
$\mathfrak{X}(T/T_s)\otimes \mathbb{R} \subset \mathfrak{X}(T)\otimes \mathbb{R}$
via its restriction to $T\cap H$.
Then 
\[
\Delta^+_{\mathcal{L}}=\chi + \Delta_{\mathcal{L}}. 
\]
\end{prop}

\begin{proof}
By multiplicativity, we may as well prove the result for $\mathcal{L}^2$.
This has two consequences: $\mathcal{L}$ has a special section, and 
$2\chi|{T_s\cap H}=0$.  

Let $s$ denote the special section of $\mathcal{L}^2$, and denote by 
$\lambda\in \mathfrak{X}(B)=\mathfrak{X}(T)$ its character.  
By \cite[Proposition 3.3]{Bri89} (see also \cite[Section 5.3]{Bria}), 
we have 
\[
\Delta^+_{\mathcal{L}^2}=\lambda + \Delta_{\mathcal{L}^2}.
\]
We know that $\lambda|_{T_s}=0$ hence we may consider $\lambda$ as an 
element of $\mathfrak{X}(T/T_s)\subset \mathfrak{X}(T)$.

From the other consequence of considering $\mathcal{L}^2$, we see that 
we may also consider $2\chi$ as an element of $\mathfrak{X}(T\cap H/T_s\cap H)$.
The natural epimorphism $T\cap H\rightarrow T/T_s$ identifies 
$T\cap H/T_s\cap H$ with $T/T_s$.
Let $t\in T\cap H$. We have, by definition of $s$, 
\begin{align*}
\lambda(t)s(eH) &= t\cdot s(t^{-1}H) \\
&= t\cdot s(eH) \\
\intertext{since $t\in H$}
&= (2\chi)(t) s(eH)
\end{align*}
by definition of $\chi$, hence the theorem.
\end{proof}

Recall the caracterization of ample and globally generated line bundles 
proved by Brion. Given a maximal cone $\mathcal{C}$ contained in $\mathcal{F}_X$,   
let $m_{\mathcal{C}}$ denote the element of 
$\sfrac{1}{2}\mathcal{M}$ such that $v_{\mathcal{L}}(y)=m_{\mathcal{C}}(y)$ 
for $y\in \mathcal{C}$.

\begin{prop}[{\cite[Th\'eor\`eme 3.3]{Bri89}}]
\label{prop_ample}
The $G$-linearized line bundle $\mathcal{L}$ is \emph{globally generated} 
if and only if 
\begin{itemize}
\item The function $v_{\mathcal{L}}$ is convex and 
\item $n_{D,\mathcal{L}} \geq m_{\mathcal{C}}(\rho(D))$ for all $D\in \mathcal{D} \setminus \mathcal{D}_{X}$ 
and maximal cone $\mathcal{C} \in \mathcal{F}_X$. 
\end{itemize}
It is \emph{ample} if and only if it is globally generated and furthermore 
\begin{itemize}
\item $m_{\mathcal{C}_1} \neq m_{\mathcal{C}_2}$ if 
$\mathcal{C}_1\neq \mathcal{C}_2\in \mathcal{F}_X$ are two maximal cones, 
\item $n_{D,\mathcal{L}} \neq m_{\mathcal{C}}(\rho(D))$ for all $D\in \mathcal{D} \setminus \mathcal{D}_{X}$ 
and maximal cone $\mathcal{C} \in \mathcal{F}_X$.
\end{itemize}
\end{prop}

\begin{defn}
The \emph{support function}
$w_{\Delta} : V^* \longrightarrow \mathbb{R}$ of a convex 
polytope $\Delta$ in a real vector space $V$ is defined by 
\[
w_{\Delta}(x)=\mathrm{sup} \{ m(x); m\in \Delta \}.
\] 
\end{defn}

One may recover the convex polytope $\Delta$ from its support function by checking
$\Delta=\{m\in V; m\leq w_{\Delta} \}$. As a consequence from this definition, 
we have:

\begin{cor}
\label{cor_gg}
If $\mathcal{L}$ is globally generated, 
then $v_{\mathcal{L}}(y)=w_{\Delta_{\mathcal{L}}}(-y)$ for $y\in |\mathcal{F}_X|$. 
\end{cor}

Let $C^+$ denote the positive Weyl chamber in 
$\mathfrak{a}^*=\mathfrak{X}(T)\otimes\mathbb{R}$, 
which may be defined as
\[
C^+=\{p\in \mathfrak{a}^*; p(\alpha^{\vee})\geq 0, \forall \alpha\in \Phi^+\}.
\]
Similarly, we define the positive restricted Weyl chamber 
in $\mathfrak{a}_s^*$ as 
\[
\bar{C}^+=\{p\in \mathfrak{a}_s^*; p(\bar{\alpha}^{\vee})\geq 0, \forall \bar{\alpha}\in \bar{\Phi}^+\}.
\]

\begin{prop}
\label{prop_toric_vs_special}
The polytope $\Delta_{\mathcal{L}}$ is a translate 
by an element of $\bar{C}^+$ of a polytope which is the intersection of a 
$\bar{W}$-invariant polytope with $\bar{C}^+$. In particular, 
$\Delta_{\mathcal{L}}\subset \bar{C}^+$
and 
\[
\Delta^t_{\mathcal{L}} = \mathrm{Conv}(\bar{W}\cdot \Delta_{\mathcal{L}}).
\]
\end{prop}

\begin{proof}
By definition, the polytope $\Delta_{\mathcal{L}}$ has outer normal along 
codimension one faces which are given by some elements of the valuation cone 
and the images by $\rho$ of some colors.
Since the only images of colors that are not in the valuation cone 
are simple restricted coroots by Proposition~\ref{prop_colored_data}
we obtain at once that $\Delta_{\mathcal{L}}$ is a translate of a polytope 
which is the intersection of a $\bar{W}$-invariant polytope 
with $\bar{C}^+$. 

To check that it is a translation by an element 
of $\bar{C}^+$, it is enough to check that $\Delta_{\mathcal{L}}$ is 
included in $\bar{C}^+$ itself. 
This is a direct consequence of the relation 
$\Delta^+_{\mathcal{L}}=\chi + \Delta_{\mathcal{L}}$ 
together with the fact that $\Delta^+_{\mathcal{L}}\in C^+$ 
by definition. 
Indeed, given $p\in \Delta_{\mathcal{L}}$ and $\alpha\in \Phi_s^+$, we have 
\[
p(\bar{\alpha}^{\vee})= 
(p+\chi)(\bar{\alpha}^{\vee})-\chi(\bar{\alpha}^{\vee}) 
\]
which is positive since 
$\chi$ is zero on $\mathfrak{a}_s$, $p+\chi\in\Delta_{\mathcal{L}}^+\subset C^+$
and $\bar{\alpha}^{\vee}$ is a positive multiple of either a positive coroot 
or the sum of two positive 
coroots by Definition~\ref{defn_rcoroot}. 
\end{proof}

Recall that the linear part of a cone containing the origin is the largest 
linear subspace included in the cone. 

\begin{cor}
\label{cor_ASSE}
The following conditions are equivalent:
\begin{enumerate}
\item $\mathcal{L}^m$ admits a global holomorphic $Q$-semi-invariant 
section for some $m\in \mathbb{N}^*$, 
\item $\Delta_{\mathcal{L}}^+\cap \mathfrak{X}(T/T\cap [L,L])\otimes \mathbb{R} \neq \emptyset$         
\item $\Delta_{\mathcal{L}}$ intersects the linear part of $\bar{C}^+$,
\item $\Delta_{\mathcal{L}}^t\cap \bar{C}^+=\Delta_{\mathcal{L}}=-\chi+\Delta_{\mathcal{L}}^+$. 
\end{enumerate}
\end{cor}

\begin{proof}
The first condition translates directly into a condition on $\Delta_{\mathcal{L}}^+$:
it is equivalent to the fact that  some $\mathcal{L}^m$ admits a global holomorphic 
$B$-semi-invariant section whose weight is in $\mathfrak{X}(T/T\cap [L,L])$,
that is, $\Delta_{\mathcal{L}}^+\cap \mathfrak{X}(T/T\cap [L,L]) \neq \emptyset$.

One checks easily that the linear part of $\bar{C}^+\subset \mathfrak{X}(T_s)\otimes \mathbb{R}$ 
is $\mathfrak{X}(T_s/T_s\cap [L,L])\otimes \mathbb{R}$, 
and coincides also with the linear subspace of $\bar{W}$-invariant elements of  
$\mathfrak{X}(T_s)\otimes \mathbb{R}$. 
Now since $\Delta_{\mathcal{L}}^+=\Delta_{\mathcal{L}}+\chi$ and 
$\chi\in \mathfrak{X}(T/(([L,L]\cap T)T_s))\otimes \mathbb{R}
\subset \mathfrak{X}(T/[L,L]\cap T)\otimes \mathbb{R}$, 
the first condition is equivalent to 
$\Delta_{\mathcal{L}}\cap \mathfrak{X}(T_s/T_s\cap [L,L])\otimes \mathbb{R}\neq \emptyset$.

Finally, thanks to Proposition~\ref{prop_toric_vs_special}, 
we obtain the equivalence with the last condition.
\end{proof}

\subsection{The anticanonical line bundle}

\label{subsec_ac}

Recall the Weil divisor representing the anticanonical class obtained by Brion 
on any spherical variety.

\begin{prop}[{\cite[Sections 4.1 and 4.2]{Bri97}}]
\label{prop_Brion_anticanonical}
The horosymmetric variety $X$ admits an anticanonical Weil divisor 
\[
-K_X=\sum_{Y\in \mathcal{I}^G(X)}Y+\sum_{D\in\mathcal{D}}m_D\overline{D}
\]
where the $m_j$ are positive integers with an explicit description in terms 
of the colored data.
\end{prop}

More precisely, if one considers the subvariety $\hat{X}\subset X$ which 
consists of all the orbits of codimension strictly less than two in $X$, 
it is a smooth variety with a well defined anticanonical line bundle, 
and this anticanonical line bundle admits a $B$-semi-invariant section $s$ 
with weight $\lambda$ whose divisor is the divisor $-K_X$ above.
In our horosymmetric situation, the weight $\lambda$ is equal to 
$\sum_{\alpha\in \Phi_{Q^u}\cup \Phi_s^+} \alpha$.

We may reason as if $\lambda|_{T_s}=\sum_{\alpha\in \Phi_{Q^u}\cup \Phi_s^+} \alpha\circ \mathcal{P}$ 
is induced by an element of $\mathfrak{X}(T/T\cap H)$, up to passing 
to $K_{\hat{X}}^{-2}$ if necessary. 
Let $h\in \mathbb{C}(\hat{X})$ be a $B$-semi-invariant function with 
weight $-\sum_{\alpha\in \Phi_{Q^u}\cup \Phi_s^+} \alpha\circ\mathcal{P}$. Then 
$hs$ is the special section of $K_{\hat{X}}^{-1}$. 
Its $B$-weight is 
$\sum_{\alpha\in \Phi_{Q^u}\cup \Phi_s^+}\alpha\circ\mathcal{H}$.
Note that this is equal to $\sum_{\alpha\in \Phi_{Q^u}}\alpha\circ\mathcal{H}$
since $\sum_{\alpha\in \Phi_s^+} \alpha\circ\mathcal{P}=\sum_{\alpha\in \Phi_s^+} \alpha$.
The special divisor $D^{ac}$ of $K_{\hat{X}}^{-1}$ is thus  
\begin{align*}
D^{ac}= &  \sum_{Y\in \mathcal{I}_X^G}\left(1-\sum_{\alpha\in \Phi_{Q^u}\cup \Phi_s^+}\alpha\circ\mathcal{P}(\mu_Y)\right)Y \\
& +\sum_{D\in\mathcal{D}}\left(m_D-\sum_{\alpha\in \Phi_{Q^u}\cup \Phi_s^+}\alpha\circ\mathcal{P}(\rho(D))\right)\overline{D}.
\end{align*}

Note that $\hat{X}$ is a global parabolic induction with respect to the 
morphism $f:\hat{X}\rightarrow G/P$ extending the natural morphism $G/H\rightarrow G/P$. 
We accordingly decompose
the anticanonical line bundle as 
$K_{\hat{X}}^{-1}=K_f^{-1}\otimes f^*K_{G/P}^{-1}$.
By definition, the special section is the product of special sections of these
naturally linearized line bundles and the divisor $D^{ac}$ on $\hat{X}$ is the sum of 
their respective divisors $D^{ac}_f$ and $D^{ac}_P$. 
The special section of $f^*K_{G/P}^{-1}$ is obviously $Q$-semi-invariant, with weight 
precisely equal to $\sum_{\alpha\in \Phi_{Q^u}}\alpha\circ\mathcal{H}$. 
As a consequence, the special section of $K_f^{-1}$ is $G$-invariant.
The special section of $K_{\hat{X}}^{-1}$ is thus $Q$-semi-invariant.
It follows from this discussion that the coefficients of colors coming 
from $L/L\cap H$ must vanish. 

By normality of $X$, the special divisors $D^{ac}_f$ of $K_f^{-1}$ and 
$D^{ac}_{P}$ of $f^*K_{G/P}^{-1}$, defined on $\hat{X}$, extend to 
$X$ as Weil divisors and $D^{ac}=D^{ac}_f+D^{ac}_{P}$ holds on $X$. 
We further have explicitly 
\[
D^{ac}_f= \sum_{Y\in \mathcal{I}_X^G}\left(1-\sum_{\alpha\in \Phi_{Q^u}\cup \Phi_s^+}\alpha\circ\mathcal{P}(\mu_Y)\right)Y
\] 
and 
\[ 
D^{ac}_{P}= \sum_{\alpha\in \Phi_{Q^u}\cap S} 
\left(m_{D_{\alpha}}-\sum_{\beta\in \Phi_{Q^u}\cup \Phi_s^+}\beta\circ\mathcal{P}(\rho(D_{\alpha}))\right)\overline{D_{\alpha}}.
\]

\subsection{Examples}

\begin{exa}
\label{exa_pic_P1xP1}
We consider again the variety $X=\mathbb{P}^1\times \mathbb{P}^1$ equipped 
with the diagonal action of $\mathrm{SL}_2$. 
The line bundles on $X$ are the $\mathcal{O}(k,m)$ for $k,m\in \mathbb{Z}$.
They admit natural $\mathrm{SL}_2\times \mathrm{SL}_2$-linearization hence also a natural 
linearization under the diagonal action. 
There are two colors $D^+$ and $D^-$ with same image $\bar{\alpha_{1,2}}^{\vee}$ via the 
color map, the fan of $X$ is the negative Weyl chamber, a single 
ray generated by $-\bar{\alpha}_{1,2}^{\vee}$ corresponding to the 
orbit $Y=\mathrm{diag}(\mathbb{P}^1)$.
The line bundle corresponding to $D^+$ is, say, $\mathcal{O}(1,0)$ and 
$\mathcal{O}(0,1)$ corresponds to $D^-$, while the line bundle 
correponding to $Y$ is obviously $\mathcal{O}(1,1)$. 

The special divisor corresponding to $\mathcal{O}(k,m)$ is 
$\frac{k+m}{2}Y+\frac{k-m}{2}(D^+-D^-)$, its moment polytope 
is the set of all $t\alpha_{1,2}$ for 
$\frac{|k-m|}{2}\leq t\leq \frac{k+m}{2}$. It is the same as the special polytope. 
The toric polytope is the set of all $t\alpha_{1,2}$ for 
$|t| \leq \frac{k+m}{2}$.
\end{exa}

\begin{exa} 
\label{exa_pic_wondAIII}
Consider the wonderful compactification $X$ of Type AIII($2$, $m>4$), under 
the action of $\mathrm{SL}_m$. Since this group is simple and simply connected, 
all line bundles admit a unique linearization. 
It follows from \cite{Bri89} that the Picard group of $X$ is the free abelian 
group generated by the three colors $D_1^+$, $D_1^-$ and $D_2$ whose images 
under the color map are $\rho(D_1^{\pm})=\bar{\alpha}_{2,m-1}^{\vee}$ and 
$\rho(D_2)=\bar{\alpha}_{1,2}^{\vee}$. 
The two $G$-invariant prime divisors $Y_1$, corresponding to the ray generated by 
$\mu_1=-\bar{\alpha}_{1,m-1}^{\vee}$ and $Y_2$, corresponding to the ray generated by 
$\mu_2=-\bar{\alpha}_{1,m}^{\vee}$ write in this basis as 
$Y_1=D_1^++D_1^--D_2$ and 
$Y_2=2D_2-D_1^+-D_1^-$. 
Given a line bundle corresponding to the divisor $k_1^+D_1^++k_1^-D_1^-+k_2D_2$, 
the corresponding special divisor is 
$b_1Y_1+b_2Y_2+b^{\pm}(D_1^++D_1^-)$ where $b_1=k_1^++k_1^-+k_2$, 
$b_2= \frac{k_1^++k_1^-}{2}+k_2$ and $b^{\pm}=\frac{k_1^+-k_1^-}{2}$.

Assume $b^{\pm}=0$.
The polytope $\Delta^t$ is the convex hull of the images by $\bar{W}$ of the 
point $b_2\bar{\alpha}_{1,2}+b_1\bar{\alpha}_{2,3}$, provided it is in the 
positive restricted Weyl chamber, that is $2b_2\geq b_1$ and $b_1\geq b_2$.
Note that Brion's ampleness criterion translates here as the fact that this 
point is in the interior of the positive restricted Weyl chamber. 
The polytope $\Delta^+=\Delta$ is the intersection of $\Delta^t$ with the 
positive chamber. 
If $b^{\pm}$ is non-zero, then we must intersect with another half-plane 
to get the polytope. 
\end{exa} 

\section{Metrics on line bundles}

\label{sec_metrics}

We will now use the objects introduced in the previous section to study 
hermitian metrics on $G$-linearized line bundles. 
Given a hermitian metric $q$ on $\mathcal{L}$, recall that its local 
potentials are the functions $\psi:y\mapsto -\ln|s(y)|^2$ where $s$ 
is a local trivialization of $\mathcal{L}$. 
We allow for now singular hermitian metrics, that is the local potentials 
are only required to be locally integrable. 
The metric is called locally bounded/continuous/smooth if and only 
if its local potentials are. 
It is called non-negatively curved (in the sense of currents) if 
the local potentials are plurisubharmonic functions and positively 
curved if its local potentials are strictly plurisubharmonic.

\subsection{Asymptotic behavior of toric potentials}

\begin{prop}
\label{prop_loc_bounded_behavior}
Let $G/H\subset X$ be a complete horosymmetric variety, 
and $\mathcal{L}$ a $G$-linearized line bundle on $X$, with special function $v_{\mathcal{L}}$.
Let $q$ be a $K$-invariant locally bounded metric on $\mathcal{L}$ 
with toric potential $u$.
Then the function 
\[
x\mapsto u(x)-2v_{\mathcal{L}}^t(x)
\]
is bounded on $\mathfrak{a}_s$.
\end{prop} 

The proof will use the process of discoloration, which allows to reduce to 
the case of a toroidal variety.

\begin{defn}
Let $(X,x)$ be an embedding of $G/H$ with colored fan $\mathcal{F}_X$.
Then the \emph{discoloration} $(X',x')$ of $(X,x)$ is the embedding 
of $G/H$ whose colored fan $\mathcal{F}_{X'}$ is obtained by 
taking the collection of all colored cones of the form 
$(\mathcal{C}\cap \mathcal{V},\emptyset)$ for 
$(\mathcal{C},\mathcal{R})\in\mathcal{F}_X$ and their faces.
\end{defn}

The discoloration $(X',x')$ of $(X,x)$ is equipped with a $G$-equivariant 
birational proper morphism $d':X'\rightarrow X$ sending $x'$ to $x$. The 
simplest example of discoloration is given by the blow up of $\mathbb{P}^2$ 
at $0$, seen as a horospherical variety under the action of 
$\mathrm{SL}_2$.

\begin{proof}
Note that the toric potential is defined only up to an additive scalar, 
but this does not affect the statement. The choice of toric potential is 
determined by the choice of a non zero element $\xi\in \mathcal{L}_{eH}$.
We fix such a choice here.

Since the special function of $\mathcal{L}^m$ is $m v_{\mathcal{L}}$ and 
the toric potential of $q^{\otimes m}$ is $m u$, replacing $\mathcal{L}$ 
by a power of $\mathcal{L}$ will not affect the result. For example 
we can already assume that $\mathcal{L}$ admits a special section. 
 
Consider the pullback $\mathcal{L}'$ of $\mathcal{L}$ under the discoloration 
morphism $d':X'\rightarrow X$, equipped with the metric 
$d'^*q$. The special function $v_{\mathcal{L}'}$ coincides with 
the special function $v_{\mathcal{L}}$ \cite[Proof of Lemma 5.3]{Pas17}.
Furthermore, by Proposition~\ref{prop_restr_tor} and up to replacing 
$\mathcal{L}$ by a power of itself, the restriction of $\mathcal{L}'$ to 
the toric subvariety $Z'\subset X'$ is a $T/T\cap H$-linearized line 
bundle with divisor 
\[
\sum_{F\in \mathcal{I}^{T/T\cap H}(Z')} v_{\mathcal{L}}^t(u_F)F.
\]
Consider the Batyrev-Tschinkel metric associated to this line bundle \cite[Section 3.3]{Mai00}.  
It is a compact torus invariant, continuous metric on $\mathcal{L}'|_{Z'}$ 
with toric potential $u_{BT}:x\in \mapsto -2\ln|\exp_s(x)\cdot \xi|_{BT}$ equal to 
\[
x\mapsto 2v_{\mathcal{L}}^t(x).
\]
Beware that here $\exp_s$ denotes the exponential map for the Lie group 
$T/T\cap H$, which does not coincide with the exponential map for $G$.
Here however, since the $T/T\cap H$-linearization of $\mathcal{L}'|_{Z'}$ 
was obtained \emph{via} factorization of the $T_s$-linearization, 
we have $\exp_s(x)\cdot \xi = \exp(x) \cdot \xi$.
We then have 
\[
(u-u_{BT})(x)=-2\ln\frac{|\exp(x)\cdot \xi|_{BT}}{|\exp(x)\cdot \xi|_{q}}.
\]
Since the Batyrev-Tschinkel metric is continuous and the metric $q$ 
is locally bounded, we obtain that the above function is globally bounded, 
hence the statement.
\end{proof}

\subsection{Positive metrics on globally generated line bundles}

In this section, $X$ is a horosymmetric variety and $\mathcal{L}$ is a globally generated 
and big line bundle on $X$. 

\begin{prop}
\label{prop_positive_metrics}
Let $q$ be a non-negatively curved, $K$-invariant, locally bounded 
hermitian metric on $\mathcal{L}$ with toric potential $u$.
Assume in addition that $q$ is locally bounded, and that its restriction to 
$\mathcal{L}|_{G/H}$ is smooth and positively curved. Then 
\begin{enumerate}
\item $u$ is a smooth, strictly convex, $\bar{W}$-invariant function,
\item there exists a constant $C$ such that 
$w_{-2\Delta^t}-C\leq u\leq w_{-2\Delta^t}+u(0)$,
\item $a\mapsto d_au$ defines a diffeomorphism from $\mathfrak{a}_s$ onto 
$\mathrm{Int}(-2\Delta^t)$,
\item $a\mapsto d_au$ defines a diffeomorphism from $\mathrm{Int}(\mathfrak{a}_s^+)$ onto 
$\mathrm{Int}(-2\Delta^t\cap \bar{C}^+)$.
\end{enumerate}
\end{prop}

\begin{proof}
Without loss of generality, we may assume that $X$ is toroidal \emph{via}
the discoloration procedure. Then the first property directly 
follows from restriction to the toric subvariety. 
The second property is a translation of Proposition~\ref{prop_loc_bounded_behavior}, 
with the additional input that by convexity and since $w_{-2\Delta^t}$ is 
piecewise linear, 
we can take $u(0)$ as constant on one side.
Then the third property is a consequence of the second, and the fourth 
follows by $\bar{W}$-invariance.
\end{proof}

\begin{rem}
Note that the open dense orbit is contained in the ample locus of 
any big line bundle on $X$, hence there are hermitian metrics as 
in the statement of Proposition~\ref{prop_positive_metrics}.
\end{rem}

The \emph{convex conjugate} $u^*:\mathfrak{a}_s^*\rightarrow \mathbb{R}\cup\{+\infty\}$ of $u$ 
is the convex function defined by 
\[
u^*(p)=\sup_{y\in\mathfrak{a}_s}(p(y)-u(y)).
\]
If $u$ is the toric potential of a metric $q$ as in Proposition~\ref{prop_positive_metrics}, 
then $u^*$ is $W$-invariant and $u^*=+\infty$ on $\mathfrak{a}_s^*\setminus -2\Delta^t$. 
Furthermore, we have 
$u^*(p)=p(a)-u(a)$ whenever $p=d_au\in \mathrm{Int}(-2\Delta^t)$. 

\subsection{Metric induced on a facet}

\label{subsec_metric_facet}

Let $\mathcal{L}$ be a $G$-linearized line bundle on a horosymmetric 
embedding $(X,x)$ of $G/H$. 
Assume that $\mathcal{L}$ admits a special section $s$ and write the 
special divisor as  
$D_{\mathcal{L}}=\sum_Yn_yY+\sum_Dn_D\overline{D}$.
For every facet $Y$ of $X$, let $\mu_Y\in \mathfrak{Y}(T_s)$ denote 
the indivisible generator of the ray corresponding to $Y$ in the 
colored fan of $X$, denote by $E_Y\subset X$ the corresponding 
elementary embedding and let $x_Y=\lim \mu_Y(z)\cdot x$.

For each facet $Y$, we choose a complement $\mathfrak{a}_Y$ of 
$\mathbb{R}\mu_Y$ in $\mathfrak{a}_s$ as in Section~\ref{subsec_facets}, 
corresponding to a torus $T_{s,Y}$. 
Note that the torus $T_{s,Y}$ is a maximal split torus for the 
involution associated to $Y$ as in Section~\ref{subsec_facets}. 

Let $h$ be a hermitian metric on $\mathcal{L}$ and assume that it is 
smooth on the elementary embedding $E_Y$. 
Denote by $\psi: b\mapsto -2\ln|s(b)|_h$ the potential of $h$ 
with respect to the special section $s$. 

There exists a unique $\lambda\in \mathfrak{X}(T/(T\cap H)T_Y)\otimes \mathbb{Q}$ 
such that $\lambda(\mu_Y)=-n_Y$. 
Up to taking a tensor power of $\mathcal{L}$, we may thus find 
a rational function $f\in \mathbb{C}(X)$ such that 
$\mathrm{ord}_Y(f)=-n_Y$ and $f(x)=1$.
Let $s_{\lambda}=fs$ denote a $B$-semi-invariant section obtained by multiplying 
the section $s$ by $f$.
Then the section $s_{\lambda}$ does not vanish identically on $Y$, 
and its restriction to $Y$ is further a special section for $\mathcal{L}|_Y$.
The potential $\psi_{\lambda}$ of $h$ with respect to $s_{\lambda}$ 
is defined on the whole $E_Y$ and satisfies, for $b\in B$, 
$\psi_{\lambda}(bH)=\psi(bH)-2\ln \lambda(b)$.

The toric potential of $h$ is 
$u(a)=\psi(\exp(a)\cdot x)$ 
and the toric potential of the restriction of $h$ to $\mathcal{L}|_Y$ is 
the function $u_Y$ defined by 
$u_Y(b)=\psi_{\lambda}(\exp(b)\cdot x_Y)$
for $b\in \mathfrak{a}_Y$.
Let us also define the function $u_{\lambda}$ on $\mathfrak{a}_s$ 
by 
$u_{\lambda}(a)=\psi_{\lambda}(\exp(a)\cdot x)=u(a)-2\lambda(a)$.

\begin{prop}
\label{prop_limits_to_facets}
Let $(t_j)$, $(\tilde{t}_j)$ be sequences of real numbers 
and $(b_j)$, $(\tilde{b}_j)$ be sequences of elements of $\mathfrak{a}_Y$ 
such that $\lim t_j=-\infty$, $(\tilde{t}_je^{t_j})$ is bounded, 
$\lim b_j=b\in \mathfrak{a}_Y$ and 
$\lim \tilde{b}_j=\tilde{b}\in \mathfrak{a}_Y$.
Then 
\[
\lim u_{\lambda}(t_j \mu_Y+b_j)=u_Y(b)
\] 
and 
\[ 
\lim d_{t_j \mu_Y+b_j}u_{\lambda}(\tilde{t}_j \mu_Y+\tilde{b}_j)=
d_bu_Y(\tilde{b}).
\]
\end{prop}

\begin{proof}
The first statement follows directly from the smoothness of $h$, 
hence of $\psi_{\lambda}$.

For the second limit, the result holds because 
\[
\lim \tilde{t}_j d_{t_j \mu_Y+b_j}u_{\lambda}(\mu_Y)=0.
\]
This is easily verified by noticing 
$\exp(-t)d_{t \mu_Y+b}u_{\lambda}(\mu_Y)=
d_{t \mu_Y+b}\psi_{\lambda}(\partial/\partial t)$
where $\partial/\partial t$ is the constant real direction vector field 
of $\mathbb{C}$ identified with the $\mathbb{C}$-factor in the toric 
subvariety $Z\simeq \mathbb{C} \times T_Y$ in $E_Y$.
\end{proof}

Assume now that $h$ is positively curved on the elementary embedding $E_Y\subset X$. 
Then we may consider the convex conjugates $u^*$, $u_{\lambda}^*$ and $u_Y^*$. 

\begin{prop}
\label{prop_legendre_facet}
Let $b\in \mathfrak{a}_Y$. Then at $p=d_bu_Y$ we have  
\[
u^*_Y(p)=u_{\lambda}^*(p)=u^*(p+2\lambda).
\]
\end{prop}

\begin{proof}
The second inequality follows from elementary properties of 
the convex conjugate.
For the other equality, we have 
\begin{align*}
u_Y^*(d_bu_Y) & = d_bu_Y(b)-u_Y(b) \\
& = \lim d_{t_j \mu_Y+b_j}u_{\lambda}(t_j \mu_Y+b_j)-u_{\lambda}(t_j \mu_Y+b_j) \\
\intertext{for any sequences such that $\lim t_j=-\infty$ and $\lim b_j=b$ by Proposition~\ref{prop_limits_to_facets}}
& =\lim u_{\lambda}^*(d_{t_j \mu_Y+b_j}u_{\lambda}).  
\end{align*}
Note that the convex conjugate $u_{\lambda^*}$ is not \emph{a priori}
continuous up to the boundary, hence it is not enough to conclude. 
On the other hand, if we choose $t\in \mathbb{R}$ we have 
\[
u_{\lambda}^*(d_bu_Y) = \lim_{s\rightarrow 1}
u_{\lambda}^*(sd_bu_y+(1-s)d_{t\mu_Y+b}u_{\lambda}).
\]
We may find $t_s$ and $b_s$ such that 
$sd_bu_y+(1-s)d_{t\mu_Y+b}u_{\lambda}=d_{t_s\mu_Y+b_s}u_{\lambda}$.
The fact that $\lim_{s\rightarrow 1}d_{t_s\mu_Y+b_s}u_{\lambda}=d_bu_Y$
ensures that $\lim t_s=-\infty$ and that $\lim b_s=b$ (it certainly 
ensures that $b_s$ is bounded, then considering converging subsequences 
yields the limit since $u_Y$ is smooth and strictly convex), 
hence the statement.
\end{proof}
 
It is clear from the point of view of the toric subvariety $Z$ that 
the domain of $u_Y^*$ is the toric polytope $-2\Delta^t_Y+2\lambda$  
of 
the restriction of $\mathcal{L}$ to $Y$, translated, 
which is the facet of 
$-2\Delta^t+2\lambda$ whose outer normal is $\mu_Y$.  
Concerning moment polytopes, we have:

\begin{prop}[{\cite{Bria}}]
\label{prop_moment_facet}
The moment polytope of $Y$ is the codimension one face of $\Delta^+$ whose 
outer normal in the affine space $\chi+\mathcal{M}_{\mathbb{R}}$ is $-\mu_Y$.
\end{prop}

The special polytope $\Delta_Y$ of $\mathcal{L}|_Y$ is then 
$\Delta^+_Y-\chi-\lambda$ 
and we have $\chi_Y=\chi+\lambda$ under the usual identifications.

\subsection{Volume of a polarized horosymmetric variety}

Before applying the results from this section combined with our computation 
of the Monge-Ampère operator to get an integration formula on horosymmetric 
varieties, we recall the formula for the volume of a horosymmetric variety, 
consequence of a general result of Brion. 

\begin{prop}[{\cite[Théorème 4.1]{Bri89}}] 
\label{prop_volume}
Let $X$ be a projective horosymmetric variety, and $\mathcal{L}$ be a $G$-linearized 
ample line line bundle on $X$. Then 
\[
\mathcal{L}^n= n! \int_{\Delta^+_{\mathcal{L}}} 
\prod_{\Phi^+\setminus E}\frac{\kappa(\alpha,p)}{\kappa(\alpha,\rho)} dp 
\]
where $E$ is the set of roots $\alpha\in \Phi^+$ that are orthogonal to  
$\Delta^+_{\mathcal{L}}$ with respect to $\kappa$ and $dp$ is the Lebesgue 
measure on the affine span of $\Delta^+_{\mathcal{L}}$, normalized by the 
translated lattice $\chi+\mathfrak{X}(T/T\cap H)$.
\end{prop}

In the case of horosymmetric varieties, the set $E$ is exactly $\Phi^+\cap \Phi_L^{\sigma}$. 
We will use the notation  
\[
P_{DH}(p)=\prod_{\Phi_{Q^u}\cup \Phi_s^+}\frac{\kappa(\alpha,p)}{\kappa(\alpha,\rho)}
=\prod_{\Phi_{Q^u}\cup \Phi_s^+}\frac{\kappa(\alpha,\alpha)}{2\kappa(\alpha,\rho)} p(\alpha^{\vee})
\]

\subsection{Integration on horosymmetric varieties}

Let $\mathcal{L}$ be a globally generated and big 
$G$-linearized line bundle on a 
horosymmetric variety $X$. 
We assume in this subsection that $h$ is a locally bounded, 
non-negatively curved metric on $\mathcal{L}$, 
smooth and positive on the restriction of $\mathcal{L}$ to $G/H$. 
We denote by $\omega$ its curvature current. 
Assume furthermore that $\chi$ vanishes on $[\mathfrak{l},\mathfrak{l}]$. 

Let $\psi$ denote a $K$-invariant function on $X$,  
integrable with respect to $\omega^n$, 
and continuous on $G/H$. 
To simplify notations, we denote by $\psi(a)$ the image by $\psi$ 
of $\exp(a)H$ for $a\in \mathfrak{a}_s$. 
Let $\Delta'$ denote the polytope $-2\Delta^t\cap \bar{C}^-$.

\begin{prop}
\label{prop_inthoro}
Let $dq$ denote the Lebesgue measure on the affine span of $\Delta^+$,
normalized by the lattice $\chi+\mathcal{M}$, let $dp$ denote 
the Lebesgue measure on $\mathcal{M}\otimes \mathbb{R}$ normalized by 
$\mathcal{M}$. Then there exist a constant 
$C_H'$, independent of $h$ and $\psi$, such that 
\begin{align*}
\int_X\psi\omega^n & =\frac{C_H'}{2^n}\int_{\Delta'} \psi(d_pu^*)
P_{DH}(2\chi-p) dp
\\ & = 
C_H'\int_{\chi+\Delta^t\cap \bar{C}^+} \psi(d_{2\chi-2q}u^*)
P_{DH}(q) dq. 
\end{align*}
\end{prop} 

\begin{proof}
Since $\omega^n$ puts no mass on $X\setminus G/H$, we may first note that 
\[
\int_X\psi\omega^n = \int_{G/H} \psi \omega^n 
\]
Then by $K$-invariance and 
Proposition~\ref{prop_integration_horosymmetric_space}, 
this is equal to 
\[
C_H\int_{-\mathfrak{a}_s^+}\psi(a)J_H(a)\frac{\omega^n}{dV_H}(\exp(a)H)da. 
\]
Now by definition of $J_H$ and Corollary~\ref{cor_MA}, this is equal to 
\[
C_H\int_{-\mathfrak{a}_s^+}  \frac{n!}{2^{2r+|\Phi_{Q^u}|}} \psi(a)
\prod_{\alpha\in\Phi_{Q^u}}(2\chi-d_au)(\alpha^{\vee})   
\prod_{\beta\in\Phi_s^+}(-d_au)(\beta^{\vee}) \det(d^2u)(a)da 
\]
We then use the change of variables $2p=d_au$ and thanks to 
Proposition~\ref{prop_positive_metrics} 
we obtain 
\[
\int_X\psi\omega^n = 
C_H\int_{(-\Delta^t)\cap \bar{C}^-}  n! 2^{|\Phi_s^+|-r} \psi(d_{2p}u^*)
\prod_{\alpha\in\Phi_{Q^u}}(\chi-p)(\alpha^{\vee})   
\prod_{\beta\in\Phi_s^+}(-p)(\beta^{\vee}) dp 
\]
where $dp$ is for the moment a Lebesgue measure independent of $\psi$. 
Actually by considering a constant $\psi$, hence the volume of $\mathcal{L}$, 
we see that the Lebesgue measure is further independent of the choice 
of $q$. 

The assumption that $\chi$ vanishes on $[\mathfrak{l},\mathfrak{l}]$ 
ensures that $\chi(\beta^{\vee})=0$ for all $\beta\in \Phi_s^+$, 
hence using the change of variables $q=\chi-p$, we get 
\[
\int_X\psi\omega^n = 
n! 2^{|\Phi_s^+|-r} C_H\int_{\chi-(-\Delta^t)\cap \bar{C}^-}  \psi(d_{2\chi-2q}u^*)
\prod_{\alpha\in\Phi_{Q^u}\cup \Phi_s^+}q(\alpha^{\vee}) dq. 
\]
Taking the constant $C_H'$ to be the covolume of the lattice 
$\mathfrak{X}(T/T\cap H)$ under $dp$ times the constant 
\[
n! 2^{2|\Phi_s^+|+|\Phi_{Q^u}|-r} C_H \prod_{\alpha\in\Phi_{Q^u}\cup \Phi_s^+}
\frac{\kappa(\alpha,\rho)}{\kappa(\alpha,\alpha)}
\]
we obtain the result.
\end{proof}

\begin{cor}
\label{cor_identify_constant}
Assume that $\mathcal{L}^m$ admits a global $Q$-semi-invariant 
section for some $m>0$. Then the constant $C_H'$ in Proposition~\ref{prop_inthoro} 
is equal to $n!(2\pi)^n$ and the integration is over $\Delta^+$ in the 
second equality.
\end{cor}

\begin{proof}
It follows from applying Proposition~\ref{prop_inthoro} to the constant function 
$\psi=\frac{1}{(2\pi)^n}$, using Corollary~\ref{cor_ASSE} to check that 
the integration is over $\Delta^+$, and comparing 
with Brion's formula for the degree of a line bundle 
(Proposition~\ref{prop_volume}). 
\end{proof}

\section{Mabuchi Functional and coercivity criterion}
\label{sec_mabuchi}

\subsection{Setting}

We fix in this section a $\mathbb{Q}$-line bundle $\mathcal{L}$ on a 
$n$-dimensional smooth horosymmetric variety $X$. We assume that there 
exists a positive integer $m$ such that $\mathcal{L}^m$ is an ample 
line bundle, with a fixed $G$-linearization, and that it admits a 
global holomorphic $Q$-semi-invariant section. Let $\Delta^+$ 
denote the moment polytope of $\mathcal{L}$, and 
$\Delta'=-2(\Delta^+-\chi)$ where $\chi$ is the isotropy character 
of $\mathcal{L}$. 
As a consequence of Corollary~\ref{cor_ASSE}, we may fix a point $\lambda_0$ 
in the relative interior of $\Delta^+\cap \mathfrak{X}(T/T\cap [L,L])\otimes \mathbb{R}$. 
The point $2(\chi-\lambda_0)$ is then in the interior of  
$\Delta'\cap \mathfrak{X}(T_s/T_s\cap [L,L])\otimes \mathbb{R}$. 

We will make the following additional assumptions:
\begin{itemize}
\item[(T)] the horosymmetric variety $X$ is toroidal, and 
whenever a facet of $\Delta^+$ intersects a Weyl 
wall, either the facet is fully contained in the wall or its 
normal belongs to the wall, furthermore, $\Delta^+$ intersects 
only walls defined by roots in $\Phi_L$,
\item[(R)] for any restricted Weyl wall, there are at least 
two roots in $\Phi^+_s$ that vanish on this Weyl wall.
\end{itemize}

Unlike the assumption that $\mathcal{L}$ is trivial on the symmetric 
fibers, these assumptions are very likely not meaningful. We use these 
to provide a rather general application of the setting we developped 
for Kähler geometry on horosymmetric varieties in a paper with reasonnable 
length. We have no claim of giving the most general statement, and 
expect that at least assumption (R) can be removed 
without too much difficulties. 
Once this is achieved, removing assumption (T) should require an analysis similar to that 
given by Li-Zhou-Zhu in \cite{LZZ} to treat non-toroidal group 
compactifications. 
Finally, removing the assumption that $\mathcal{L}$ is trivial on the symmetric 
fibers appears to be a much more challenging problem in view of the 
expression of the curvature form, as convex 
conjugacy in this generality seems to be a bit less helpful.

Note that these assumptions are satisfied in a large variety of situations. 
We expect that the second part of assumption (T), in terms of the moment 
polytope, is actually implied by the assumption that $X$ is toroidal. 
This is true at least for symmetric varieties and horospherical varieties.  
Assumption (R) is satisfied for example 
when the symmetric fiber is of group type, of 
type AIII($r$, $n>2r$), of type AII($p$), but unfortunately not 
when the symmetric fiber is of type AI($m$). It is obviously satisfied is 
the variety $X$ is horospherical. 

Recall that we fixed an anticanonical $\mathbb{Q}$-divisor
$D^{ac}$ with a decomposition $D^{ac}=D_f^{ac}+D_{P}^{ac}$ in 
Section~\ref{subsec_ac} such that 
$\mathcal{O}(D_{P}^{ac})=f^*K^{-1}_{G/P}$ on the complement $\hat{X}$ of 
codimension $\geq 2$ orbits, and 
$D_f^{ac}=\sum_Y n_Y Y$ 
where 
\[
n_Y=1-\sum_{\alpha\in\Phi_{Q^U}\cup \Phi_s^+}\alpha\circ \mathcal{P}(\mu_Y).
\]

Let $\Theta$ be a $G$-stable $\mathbb{Q}$-divisor on $X$ with 
simple normal crossing support. 
Write $\Theta=\sum c_YY$ and set
\[
n_{Y,\Theta}=-c_Y+1-\sum_{\alpha\in\Phi_{Q^U}\cup \Phi_s^+}\alpha\circ \mathcal{P}(\mu_Y).
\]
We assume all coefficients $c_Y$ satisfy $c_Y< 1$. 
Recall that we denote $\sum_{\alpha\in\Phi_{Q^U}}\alpha\circ \mathcal{H}$
by $\chi^{ac}$ and this is the isotropy character associated to the 
anticanonical line bundle on $G/H$.

Fix a smooth positive reference metric $h_{\mathrm{ref}}$ on $\mathcal{L}$, 
 and denote its curvature form by $\omega_{\mathrm{ref}}$. 
Let $\mathrm{rPSH}^K(X,\omega_{\mathrm{ref}})$ denote the space of  
smooth $K$-invariant strictly $\omega_{\mathrm{ref}}$-plurisubharmonic potentials on $X$.  
The functions in $\mathrm{rPSH}^K(X,\omega_{\mathrm{ref}})$ are in one-to-one 
correspondence with smooth positive hermitian metrics on $\mathcal{L}$. 
We denote by $h_{\phi}$ the metric 
corresponding to $\phi\in \mathrm{rPSH}^K(X,\omega_{\mathrm{ref}})$ and 
we write $\omega_{\phi}=\omega_{\mathrm{ref}} + i\partial \bar{\partial} \phi$
for the curvature of $h_{\phi}$, which depends on $\phi$ only up to 
an additive constant.

To any $\phi\in \mathrm{rPSH}^K(X,\omega_{\mathrm{ref}})$ is associated 
a toric potential $u$: the toric potential of $h_{\phi}$. Note that under our 
assumptions ($X$ is smooth and toroidal) and by Proposition~\ref{prop_smoothness}, 
$X$ admits a smooth toric 
submanifold $Z$, and $u$ is the toric potential of the restriction of 
$h_{\phi}$ to the restricted ample $\mathbb{Q}$-line bundle $\mathcal{L}|_Z$, 
hence the convex potential $u^*$ of $u$ satisfy the Guillemin-Abreu 
regularity conditions in terms of the polytope $-\Delta^t$.

\subsection{Scalar curvature}

The scalar curvature $S$ of 
a smooth Kähler form $\omega$ is defined by the 
formula 
\[
S=\frac{n\mathrm{Ric}(\omega)\wedge \omega^{n-1}}{\omega^n}
\]
Note that $\mathrm{Ric}(\omega)$ is the curvature form of the metric 
on $K_X^{-1}$ corresponding to the volume form 
$\omega^n$,  
whose toric potential we denote by $\tilde{u}$. 
Assuming that $\omega$ is the curvature form of a metric on 
$\mathcal{L}$, we can determine this toric 
potential using Theorem~\ref{thm_curv}.
Using the liberty to choose the multiplicative constant for the section 
defining the toric potential (a multiple of the dual of 
$\bigwedge_{\diamondsuit}\gamma_{\diamondsuit}$), 
we may assume that, 
\begin{align*}
\tilde{u}(a)& = -\ln\det d^2_au
-\sum_{\alpha\in\Phi_{Q^u}\cup\Phi_s^+}\ln((2\chi-d_au)(\alpha^{\vee}))
\\ & \quad 
+\sum_{\beta\in \Phi_s^+}\ln \sinh(-2\beta(a))
-\sum_{\alpha\in\Phi_{Q^u}}2\alpha(a),
\end{align*}
for $a\in -\mathfrak{a}_s^+$.

We will, for the rest of the section, fix a choice of 
orthonormal basis $(l_j)_j$ of $\mathfrak{a}_s$ (with respect to some 
fixed scalar product whose corresponding Lebesgue measure is normalized 
by $\mathcal{M}$) and 
corresponding dual basis $(l_j^*)$ of $\mathfrak{a}_s^*$.
We write $\alpha^{\vee,j}$ for the 
coordinates of $\alpha^{\vee}$. 
We use the notations $d_au=\sum_ju_j(a)l_j^*$, $(u^{j,k})$ for the 
inverse matrix of $(u_{j,k})$, etc.
To simplify notations, we sometimes omit summation symbols in which case 
we sum over repeated indices in a given term.
We then have 
\[
\tilde{u}=-\ln\det (u_{l,m})
-\sum_{\alpha\in\Phi_{Q^u}\cup\Phi_s^+}\ln((2\chi_l-u_l)\alpha^{\vee,l})
+I_H
\]
where we set $I_H(a)=\sum_{\beta\in \Phi_s^+}\ln \sinh(-2\beta(a))
-\sum_{\alpha\in\Phi_{Q^u}}2\alpha(a)$. 
We set $p=d_au$ and consider both $a$ and $p$ as variables in the dual 
spaces $\mathfrak{a}_s$ and $\mathfrak{a}_s^*$.

\begin{prop}
\label{prop_scalar_curvature}
The scalar curvature at $\exp(a)H$ is equal to 
\begin{align*}
& -u^{*,i,j}_{i,j}(p)
+\Big{(}-2u^{*,i,j}_j(p)+(I_H)_i(a)\Big{)}\frac{P_{DH,i}'}{P_{DH}'}(p)
+u^*_{i,j}(p)I_{H,i,j}(a)
\\ & 
 -u^{*,i,j}(p)\frac{P_{DH,i,j}'}{P_{DH}'}(p)
+\sum_{\alpha\in \Phi_{Q^u}}\frac{2\chi^{ac}(\alpha^{\vee})}{(2\chi-p)(\alpha^{\vee})}
\end{align*}
\end{prop}

\begin{proof}
We compute, using Jacobi's formula, 
\[
\tilde{u}_j=-u^{l,m}u_{m,l,j}
-\sum_{\alpha\in\Phi_{Q^u}\cup\Phi_s^+}\frac{-u_{l,j}\alpha^{\vee,l}}{(2\chi-u_l)\alpha^{\vee,l}}
+I_{H,j}
\]
Using the variable $p=d_au$ and convex conjugate, we have 
$d_pu^*=a$, $d^2_au=(d^2_pu^*)^{-1}$ and thus 
$u^{*,i,j}_j(p)=u^{k,j}(a)u_{j,k,i}(a)$.
We may then give another expression of $\tilde{u}_j$:
\[
\tilde{u}_j(a)=-u^{*,j,i}_i(p)
-\sum_{\alpha\in\Phi_{Q^u}\cup\Phi_s^+}\frac{-u^{*,l,j}(p)\alpha^{\vee,l}}{(2\chi-p)(\alpha^{\vee})}
+I_{H,j}(d_pu^*).
\]
then 
\begin{align*}
\tilde{u}_{j,k}(a)& = -u^{*,j,i}_{i,s}(p)u^{*,s,k}(p)
-\sum_{\alpha\in\Phi_{Q^u}\cup\Phi_s^+}
\left(\frac{-u^{*,l,j}\alpha^{\vee,l}}{(2\chi-p)(\alpha^{\vee})}\right)_s(p)u^{*,s,k}(p) \\
& \quad +I_{H,j,k}(a).
\end{align*}
 
We now compute for $a\in -\mathfrak{a}_s^+$,
\begin{align*}
\frac{n\mathrm{Ric}(\omega_{\phi})\wedge \omega_{\phi}^{n-1}}{\omega_{\phi}^n} & =
\mathrm{Tr}((\tilde{u}_{l,m}(a))(u^{l,m}(a))) 
+\sum_{\alpha\in \Phi_{Q^u}\cup\Phi_s^+} \frac{-\tilde{u}_l(a)\alpha^{\vee,l}}{(2\chi-p)(\alpha^{\vee})} \\ 
& \quad +\sum_{\alpha\in \Phi_{Q^u}}\frac{2\chi^{ac}(\alpha^{\vee})}{(2\chi-p)(\alpha^{\vee})} 
\end{align*}
which, by using the previous expressions, 
is equal to 
\begin{align*} 
& -u^{*,i,j}_{i,j} 
+\sum_{\alpha\in \Phi_{Q^u}\cup\Phi_s^+}
2\frac{u^{*,i,j}_j\alpha^{\vee,i}}{(2\chi-p)(\alpha^{\vee})} 
-\sum_{\alpha,\beta\in \Phi_{Q^u}\cup\Phi_s^+} 
\frac{u^{*,i,j}\alpha^{\vee,i}\beta^{\vee,j}}{(2\chi-p)(\alpha^{\vee})(2\chi-p)(\beta^{\vee})} \\
& -\sum_{\alpha\in \Phi_{Q^u}\cup\Phi_s^+}
\frac{u^{*,i,j}\alpha^{\vee,i}\alpha^{\vee,j}}{((2\chi-p)(\alpha^{\vee}))^2}
+u^*_{i,j}(I_H)_{i,j}(a)
 \\
& -\sum_{\alpha\in \Phi_{Q^u}\cup\Phi_s^+} \frac{(I_H)_i(a)\alpha^{\vee,i}}{(2\chi-p)(\alpha^{\vee})}
+\sum_{\alpha\in \Phi_{Q^u}}\frac{2\chi^{ac}(\alpha^{\vee})}{(2\chi-p)(\alpha^{\vee})}
\end{align*}

Let $P_{DH}'(p)=P_{DH}(2\chi-p)$, then note that 
\[
P_{DH,i}'(p)=P_{DH}'(p)\sum_{\alpha}\frac{-\alpha^{\vee,i}}{(2\chi-p)(\alpha^{\vee})}
\]
and 
\[
P_{DH,i,j}'(p)=P_{DH}'(p)\left(\sum_{\alpha,\beta}\frac{\alpha^{\vee,i}\beta^{\vee,j}}{(2\chi-p)(\alpha^{\vee})(2\chi-p)(\beta^{\vee})}+\sum_{\alpha}\frac{\alpha^{\vee,i}\alpha^{\vee,j}}{(2\chi-p)(\alpha^{\vee})^2}\right).
\]
Plugging this into the last expression of the scalar curvature yields the result.
\end{proof}

\begin{rem}
The computation of the scalar curvature here is only on the homogeneous space, 
hence holds under weaker hypothesis than in the setting: we only need to assume 
that $\mathcal{L}$ is a line bundle on $G/H$ whose restriction to the symmetric 
fiber is trivial, and that $h$ is a smooth and positive metric on $\mathcal{L}$.
\end{rem}

Denote by $\bar{S}$ the average scalar curvature, defined as 
\[
\bar{S}=\frac{\int_{X} n\mathrm{Ric}(\omega)\wedge \omega^{n-1}}{\int_{X} \omega^{n}}.
\]

\subsection{The functionals}

\subsubsection{The $J$-functional}

The $J$-functional is defined (up to a constant) 
on the space $\mathrm{rPSH}^K(X,\omega_{\mathrm{ref}})$ by its variations as follows:
if $\phi_t$ is a smooth path in 
$\mathrm{rPSH}^K(X,\omega_{\mathrm{ref}})$ 
between the origin and $\phi$, then 
\[
J(\phi)=\int_0^1 \int_{X} 
\dot{\phi}_t \frac{\omega_{\mathrm{ref}}^n-\omega_{\phi_t}^n}{(2\pi)^n\mathcal{L}^n} dt.
\]

\begin{prop}
\label{prop_J_functional}
Let $u$ denote the toric potential of $h_{\phi}$. Then
\[
\left|J(\phi)-u(0)-\frac{n!}{\mathcal{L}^n}\int_{\Delta^+}
u^*(2\chi-2q)P_{DH}(q)dq\right| 
\]
is bounded independently of $\phi$.
\end{prop}

\begin{proof}
We have 
\[
J(\phi) = \int_{X} \phi \frac{\omega_{\mathrm{ref}}^n}{(2\pi)^n\mathcal{L}^n} 
- \int_0^1 \int_{X} \dot{\phi}_t \frac{\omega_{\phi_t}^n}{(2\pi)^n\mathcal{L}^n} dt
\]
The definition of convex conjugate yields $\dot{u}_t^*(d_au)=-\dot{u}_t(a)$ hence 
by Proposition~\ref{prop_inthoro} and Corollary~\ref{cor_identify_constant}, 
we have
\begin{align*}
\int_0^1 \int_X \dot{\phi}_t \omega_{\phi_t}^n dt & 
= - \int_0^1 C_H'\int_{\chi+\Delta^t\cap\bar{C}^+}\dot{u}_t^*(2\chi-2q)P_{DH}(q)dq 
\\
& = -C_H' \int_{\chi+\Delta^t\cap\bar{C}^+} (u^*-u_{\mathrm{ref}}^*)(2\chi-2q)P_{DH}(q)dq
\\
& = - (2\pi)^n n! \int_{\Delta^+} (u^*-u_{\mathrm{ref}}^*)(2\chi-2q)P_{DH}(q)dq
\end{align*}
On the other hand, it follows from classical results \cite[Proposition 2.7]{GZ05} that  
\[
\left|\frac{1}{(2\pi)^n\mathcal{L}^n}\int_X \phi \omega_{\mathrm{ref}}^n - 
\sup_X(\phi) \right|
\]
is bounded independently of $\phi\in\mathrm{rPSH}^K(X,\omega_{ref})$. 
We have 
\begin{itemize}
\item $\sup_X(\phi)=\sup_{\mathfrak{a}_s}(u-u_{\mathrm{ref}})$, 
\item $w_{-2\Delta^t}-C_1 \leq u_{\mathrm{ref}} \leq w_{-2\Delta^t}+u_{\mathrm{ref}}(0)$ for some constant $C_1$ by Proposition~\ref{prop_positive_metrics},  
and 
\item $\sup_{\mathfrak{a}_s} u-w_{-2\Delta^t} = u(0)$ by convexity, 
\end{itemize}
hence 
\[
\left|\int_X \phi \omega_{\mathrm{ref}}^n-(2\pi)^n\mathcal{L}^n u(0)\right|
\]
is bounded independently of $\phi$.
\end{proof}

\subsubsection{Mabuchi functional}

The (log)-Mabuchi functional $\mathrm{Mab}_{\Theta}$ relative to $\Theta$ is defined also 
by integration along smooth path: $\mathrm{Mab}_{\Theta}(\phi)$ is equal to 
\[
\int_0^1 \Big{\{} \int_{X} 
\dot{\phi}_t(\bar{S}_{\Theta}\omega_{\phi_t}-n\mathrm{Ric}(\omega_{\phi_t}))\wedge \frac{\omega_{\phi_t}^{n-1}}{(2\pi\mathcal{L})^n}  
+2\pi n\sum_Y c_Y \int_{Y}\dot{\phi}_t\frac{\omega_{\phi_t}^{n-1}}{(2\pi\mathcal{L})^n}  
\Big{\}}dt
\] 
where 
$\bar{S}_{\Theta}=\bar{S}-n\sum_Yc_Y\mathcal{L}|_Y^{n-1}/\mathcal{L}^n$.

Let $\tilde{\Delta}^+_Y$ denote the bounded cone with vertex $\lambda_0$ and base $\Delta^+_Y$. 
Note that in general $\Delta^+\neq \bigcup_Y \tilde{\Delta}^+_Y$.
In is however the case under assumption (T), that is if we assume that $X$ is toroidal. 

We set $\Lambda_Y=(n_Y-c_Y)/v_{\mathcal{L}}(\mu_Y)$. 

\begin{thm} 
\label{thm_mabuchi}
Under assumption (T), up to the choice of a normalizing additive constant, 
$\frac{\mathcal{L}^n}{n!} \mathrm{Mab}_{\Theta}(\phi)$ is equal to 
\begin{align*}
&
\sum_Y \Lambda_Y 
\int_{\tilde{\Delta}^+_Y} \Big{(}nu^*(p)-u^*(p)\sum \frac{\chi(\alpha^{\vee})}{q(\alpha^{\vee})} +d_pu^*(p)\Big{)}P_{DH}(q)dq \\
& 
+\int_{\Delta^+} u^*(p)\Big{(}\sum \frac{\chi^{ac}(\alpha^{\vee})}{q(\alpha^{\vee})}-\bar{S}_{\Theta}\Big{)}P_{DH}(q)dq 
-\int_{\Delta^+}I_H(a)P_{DH}(q)dq \\
&
-\int_{\Delta^+}\ln\det (u^*_{i,j})(p)P_{DH}(q)dq 
\end{align*}
\end{thm}

\begin{proof}
The summands 
\[
\int_0^1 \int_X \dot{\phi}_t\bar{S}_{\Theta}  \omega_{\phi_t}^ndt/(2\pi\mathcal{L})^n
\]
and 
\[
2\pi n \sum_Y c_Y \int_{Y}\dot{\phi}_t\omega_{\phi_t}^{n-1}dt/(2\pi\mathcal{L})^n
\] 
may be dealt with as in the computation for $J$, yielding respectively 
\[
\bar{S}_{\Theta} \frac{n!}{\mathcal{L}^n} \int_{\Delta^+}(u_{\mathrm{ref}}^*-u^*)(2\chi-2q)P_{DH}(q)dq
\] 
and 
\[
\sum_Y c_Y\frac{n!}{\mathcal{L}^n} \int_{\Delta_Y^+}(u_{\mathrm{ref}}^*-u^*)(2\chi-2q)P_{DH}(q)dq_Y.
\]

The harder part is 
\[
\int_0^1 \int_{X} -n \dot{\phi}_t \mathrm{Ric}(\omega_{\phi_t}) \wedge \omega_{\phi_t}^{n-1}/(2\pi\mathcal{L})^n dt.
\]

Using the integration formula as well as the formula for the 
scalar curvature, we have 
\begin{align*}
I:= &\int_{X} -n \dot{\phi}_t \mathrm{Ric}(\omega_{\phi_t}) \wedge \omega_{\phi_t}^{n-1} = 
\int_{X} -n \dot{\phi}_t \frac{\mathrm{Ric}(\omega_{\phi_t}) \wedge \omega_{\phi_t}^{n-1}}{\omega_{\phi_t}^n}\omega_{\phi_t}^n \\
= & \frac{C_H'}{2^n} \int_{\Delta'} \dot{u}^*_t  \Big{(}
-u^{*,i,j}_{t,i,j}P_{DH}'
-2u^{*,i,j}_{t,j}P_{DH,i}'  
-u^{*,i,j}_{t}P_{DH,i,j}' \\
&
+u^*_{t,i,j}I_{H,i,j}(a)P_{DH}'
+I_{H,i}(a)P_{DH,i}'
+\sum_{\alpha\in \Phi_{Q^u}}\frac{2\chi^{ac}(\alpha^{\vee})}{(2\chi-p)(\alpha^{\vee})}P_{DH}'
\Big{)}
dp
\end{align*}
where the variable, if omitted, is $p=d_au$.

Apart from the last term, the situation is analogous to Li-Zhou-Zhu \cite{LZZ}
hence we may follow the same steps. However note that translation by $\chi$ 
in the Duistermaat-Heckman polynomial will sometimes introduce extra terms. 
Denote by $\nu$ the unit outer normal to $\partial \Delta'$. 
On a codimension one face corresponding to 
a facet $Y$ of $X$, $\nu$ is a positive multiple of $\mu_Y$.
 
We will apply several times the divergence theorem
as follows, without writing the details every time. 
For $0<s<1$, let $\Delta'_{s}$ denote the bounded 
cone with vertex $2(\chi-\lambda_0)$ and base the dilation by $s$ of 
the boundary $s\partial \Delta'$. We may apply the divergence theorem 
on $\Delta'_{s}$ to smooth functions on the interior of $-2\Delta^t$, 
then take the limit as $s\rightarrow 1$, 
applying dominated convergence and using the appropriate 
convergence result. We let $d\sigma$ denote the area measure on 
all boundaries. 
 
It follows from Section~\ref{subsec_metric_facet} 
that for $p\in \Delta'_Y$ and any $i$,
\[
\lim_{s\rightarrow 1} u_{i,j}(d_{sp}u^*)\nu_j = 0 
\] 
and 
\[
\lim_{s\rightarrow 1} \tilde{u}_{j}(d_{sp}u^*)(\mu_Y)_j = n_Y. 
\]
Recall that 
\[
\tilde{u}_j(a)=-u^{*,j,i}_i(p)
-\sum_{\alpha\in\Phi_{Q^u}\cup\Phi_s^+}\frac{-u^{*,l,j}(p)\alpha^{\vee,l}}{(2\chi-p)(\alpha^{\vee})}
+I_{H,j}(d_pu^*).
\]
We deduce from these facts, and the fact that $P_{DH}'$ vanishes 
on restricted Weyl walls, that 
\[
\lim_{s\rightarrow 1}
\int_{\partial \Delta'_{s}} \dot{u}_t^*
(-u^{*,i,j}_{t,j}+I_{H,i}(a))\nu_iP_{DH}'d\sigma
= \sum_Y \int_{\Delta'_Y} 2n_Y \frac{\nu}{\mu_Y} \dot{u}_t^*P_{DH}'d\sigma
\]
where $\Delta'_Y$ denotes the facet of $\Delta'$ whose outer normal is 
$\mu_Y$.
 
A first application of the divergence theorem then yields, by passing to the limit, 
that the above quantity is equal to 
\[
\int_{\Delta'}\Big{(}\dot{u}_t^*(-u_{t,j}^{*,i,j}+I_{H,i}(a))P_{DH}'\Big{)}_idp.
\]
We may compute 
\begin{align*}
\Big{(}\dot{u}_t^*(-u_{t,j}^{*,i,j}+I_{H,i}(a))P_{DH}'\Big{)}_i = &
-\dot{u}^*_{t,i}u^{*,i,j}_{t,j}P_{DH}'
+\dot{u}^*_{t,i}I_{H,i}(a)P_{DH}'
 \\ &
-\dot{u}^*_{t}u^{*,i,j}_{t,i,j}P_{DH}'
+\dot{u}^*_{t}u^*_{j,i}I_{H,i,j}(a)P_{DH}'
 \\ &
-\dot{u}^*_{t}u^{*,i,j}_{t,j}P_{DH,i}'
+\dot{u}^*_{t}I_{H,i}(a)P_{DH,i}'
\end{align*}
and 
\[
\dot{u}^*_{t,i}I_{H,i}(a)P_{DH}'=\frac{d}{dt}(I_H(a)P_{DH}').
\]

Consider now the vector field 
$(\dot{u}_{t,i}^*P_{DH}'-\dot{u}_t^*P_{DH,i}')u_t^{*,i,j}$. 
Applying the divergence theorem to this vector field yields 
\begin{align*}
0=& 
\int_{\Delta'} (
\dot{u}^*_{t,i,j}u_t^{*,i,j}P_{DH}'
+\dot{u}^*_{t,i}u_{t,j}^{*,i,j}P_{DH}'
-\dot{u}^*_{t}u_{t,j}^{*,i,j}P_{DH,i}'
-\dot{u}^*_{t}u_t^{*,i,j}P_{DH,i,j}'
)dp
\end{align*}
Note here that 
\[
\dot{u}^*_{t,i,j}u_t^{*,i,j}=\frac{d}{dt}(\ln \det (u^*_{t,i,j})).
\]

From these two applications of the divergence theorem and the expression 
of the scalar curvature, we deduce by taking the sum that $I$ is $C_H'/2^n$ times 
\emph{the derivative with respect to} $t$ of 
\begin{align*}
& \sum_Y \int_{\Delta'_Y} 2n_Y \frac{\nu}{\mu_Y} u_t^*P_{DH}'d\sigma
-\int_{\Delta'} I_H(a)P_{DH}'dp
-\int_{\Delta'} \ln \det (u^*_{t,i,j})P_{DH}'dp \\ &
+\int_{\Delta'} u_t^* \sum_{\alpha\in \Phi_{Q^u}}\frac{2\chi^{ac}(\alpha^{\vee})}{(2\chi-p)(\alpha^{\vee})}P_{DH}'dp
\end{align*}
hence the value of the above expression at $t=1$ is the corresponding term of the Mabuchi 
functional, up to a constant independent of $\phi$. 

We finally use the divergence theorem one last time, applied to 
the vector field $u_t^*p_iP_{DH}'$
to obtain, for every $Y$,  
\begin{align*}
\int_{\Delta'_Y} p_i\nu_i u_t^*P_{DH}'d\sigma
& = 
\int_{\tilde{\Delta}'_Y} (u_t^*P_{DH,i}'p_i+ru_t^*P_{DH}+u^*_{t,i}p_iP_{DH}')dp \\
& = 
\int_{\tilde{\Delta}'_Y} (nu_t^*-u_t^*\sum_{\alpha}\frac{2\chi_i\alpha^{\vee,i}}{(2\chi-p)(\alpha^{\vee})}+u^*_{t,i}p_i)P_{DH}'dp 
\end{align*}
Since $2n_Y\frac{\nu}{\mu_Y}=\frac{n_Y}{v_{\mathcal{L}}(\mu_Y)}p(\nu)$
and $dp_Y=\frac{p(\nu)}{v_{\mathcal{L}}(\mu_Y)}d\sigma$, where $dp_Y$ is the 
image of the measure $dq_Y$ under translation by $-\chi$ then dilation by $2$. 
This allows to transform 
the remaining integrals on $\Delta'_Y$ or $\Delta_Y^+$ to 
integrals on $\Delta^+$, after the change of variable from $p$ to $q$.
Putting everything together gives the result.
\end{proof}

\begin{rem}
\label{rem_barS}
As a corollary of the proof, by applying the same steps,  
we can compute $\bar{S}_{\Theta}$. We obtain 
\[
\bar{S}_{\Theta}= \sum_Y \int_{\tilde{\Delta}_Y^+}
(n \Lambda_Y P_{DH}(q) + d_qP_{DH}(\chi^{ac}-\Lambda_Y\chi))dq\Big{/}\int_{\Delta^+}P_{DH}(q)dq
\]
\end{rem}

\subsection{Action of $Z(L)^0$}

Following Donaldson \cite{Don02}, let us write the Mabuchi functional as the sum 
of a linear and non-linear part 
$\mathrm{Mab}_{\Theta}=\mathrm{Mab}^l_{\Theta}+\mathrm{Mab}^{nl}_{\Theta}$ 
by setting
\begin{align*}
\frac{\mathcal{L}^n}{n!} \mathrm{Mab}^l_{\Theta}(\phi) = &
\sum_Y \Lambda_Y 
\int_{\tilde{\Delta}^+_Y} (nu^*(p)-u^*(p)\sum \frac{\chi(\alpha^{\vee})}{q(\alpha^{\vee})} +d_pu^*(p))P_{DH}dq \\
& 
+\int_{\Delta^+} u^*(p)(\sum \frac{\chi^{ac}(\alpha^{\vee})}{q(\alpha^{\vee})}-\bar{S}_{\Theta})P_{DH}dq 
+\int_{\Delta^+}4\rho_H(a)P_{DH}(q)dq
\end{align*}
where $2\rho_H=\sum_{\alpha\in\Phi_{Q^u}\cup\Phi_s^+}\alpha\circ\mathcal{P}$.
We will use the notation $M^l(u^*)$ to denote the above as a function of 
$u^*$, where $u^*$ is the convex conjugate of the toric potential of $h_{\phi}$.
Similarly, we use the notation 
$M^{nl}(u^*)=\frac{\mathcal{L}^n}{n!} \mathrm{Mab}_{\Theta}(\phi)-M^l(u^*)$.

\begin{rem}
\label{rem_alt_Ml}
In the last step of the proof of Theorem~\ref{thm_mabuchi},  
if we apply the divergence theorem to the 
vector field $4u^*P_{DH}'\rho_H$ instead of $u_t^*p_iP_{DH}'$ we obtain 
another expression of the linear part of the Mabuchi functional 
$\frac{\mathcal{L}^n}{n!} \mathrm{Mab}_{\Theta}^l(\phi)$:  
\begin{align*}
M^l(u^*)=&
\sum_Y \frac{1-c_Y}{v_{\mathcal{L}}(\mu_Y)}
\int_{\Delta'_Y} p(\nu)u^*(p)P_{DH}'(p)d\sigma 
-\int_{\Delta^+} d_pP_{DH}'(4\rho_H)u^*(p)dp \\
& 
+\int_{\Delta^+} u^*(p)(\sum \frac{\chi^{ac}(\alpha^{\vee})}{q(\alpha^{\vee})}-\bar{S}_{\Theta})P_{DH}dq.  
\end{align*}
\end{rem}

\subsubsection{Invariance of Mabuchi functional and log-Futaki invariant}

Consider the connected center $\mathcal{Z}(L)^0$ of $L$. 
It acts on the right on $G/H$ and the action extends to $X$. 
The induced action on $K$-invariant singular hermitian metrics on $\mathcal{L}$ 
stabilizes the set $\mathrm{rPSH}^K(X,\omega_{\mathrm{ref}})$. 
More precisely, for $b\in\mathfrak{a}_s\cap \mathfrak{z}(\mathfrak{l})$, 
and $h$ a $K$-invariant, non-negatively curved singular hermitian metric on $\mathcal{L}$ with 
toric potential $u$, the toric potential of the image by $\exp(b)$ 
of $h$ is $a\mapsto u(a+b)$.
This translates on convex conjugates as replacing $u^*$ by 
$u_b^*=u^*-b$.
Note that it is enough to consider only elements of $\mathfrak{a}_s\cap \mathfrak{z}(\mathfrak{l})$
since $\mathcal{Z}(L)^0=T\cap \mathcal{Z}(L)^0$, and $T\cap H$ as well 
as $T\cap K$ act trivially.
Since 
$du^*_b=du^*-b$, $\alpha(b)=0$ for $\alpha\in\Phi_L$, 
$\chi(b)=0$ and $d^2u_b^*=d^2u^*$, 
we have the following proposition.

\begin{prop}
\label{prop_log_Futaki}
The Mabuchi functional is invariant under the action of $\mathcal{Z}(L)^0$ on 
the right if and only if 
\begin{align*}
0 = &
\sum_Y \int_{\tilde{\Delta}_Y^+}
-2q(b)\big{(}((n+1)\Lambda_Y-\bar{S}_{\Theta})P_{DH}(q)
+d_qP_{DH}(\chi^{ac}-\chi) \big{)}dq  \\ &
+\int_{\Delta^+}2\sum_{\alpha\in\Phi_{Q^u}}\alpha(b)P_{DH}(q)dq
\end{align*}
for all $b\in\mathfrak{a}_s\cap \mathfrak{z}(\mathfrak{l})$.
\end{prop}

The expression in the theorem could naturally be interpreted as a 
log Calabi-Futaki invariant.

\subsubsection{Normalization of potentials}
\label{subsec_normalize}

The action of $Z(L)^0$ allows on the other hand to normalize the 
psh potentials, as follows. 
Given $\phi \in \mathrm{rPSH}(X,\omega_{ref})$, we may add a constant 
and use the action of an element of $Z(L)^0$ to obtain another 
potential $\hat{\phi} \in \mathrm{rPSH}(X,\omega_{ref})$, such that 
if $\hat{u}$ is the corresponding toric potential, and 
$\hat{u}^*$ its convex conjugate, we have  
$\mathrm{min}_{-2\Delta^t}\hat{u}^*=\hat{u}^*(2(\chi-\lambda_0))=0$.

\subsection{Coercivity criterion}

\subsubsection{Statement of the criterion and examples}

\begin{defn}
The Mabuchi functional is \emph{proper modulo the 
action of $Z(L)^0$} if there exists positive constants $\epsilon$ 
and $C$ such that 
for any $\phi\in \mathrm{rPSH}(X,\omega_{ref})$, 
there exists $g\in Z(L)^0$ such that 
\[
\mathrm{Mab}_{\Theta}(\phi)\geq \epsilon J(g\cdot \phi)-C.
\]
\end{defn}

Consider the function $F_{\mathcal{L}}$ defined piecewise by 
\[
F_{\mathcal{L}}(q)=(n+1)\Lambda_Y-\bar{S}_{\Theta}
+\sum\frac{(\chi^{ac}-\Lambda_Y\chi)(\alpha^{\vee})}{q(\alpha^{\vee})}
\]
for $q$ in the unbounded cone with vertex $\lambda_0-\chi$ and generated by 
$\tilde{\Delta}_Y^+$. 
Note that 
$F_{\lambda\mathcal{L}}(q)=\frac{1}{\lambda}F_{\mathcal{L}}(q)$.
The Mabuchi functional for the line bundle $\mathcal{L}$ 
is proper if and only if it is proper for any positive rational multiple of 
$\mathcal{L}$. As an application of this remark, if $F_{\mathcal{L}}>0$, 
we may choose the multiple 
$\lambda \mathcal{L}$ in such a way that 
\[
\int_{\Delta^+}P_{DH}dq = 
\sum_Y \int_{\tilde{\Delta}_Y^+} F_{\lambda\mathcal{L}}(q)
P_{DH}(q)dq.
\]
We now replace $\mathcal{L}$ by its multiple to assume the equality above is 
satisfied with $\lambda=1$.
We then define a barycenter of $\Delta^+$ by: 
\[
\mathrm{bar} = \sum_Y \int_{\tilde{\Delta}_Y^+} q F_{\mathcal{L}}(q)
P_{DH}(q)dq /\int_{\Delta^+}P_{DH}dq.
\]

\begin{thm}
\label{thm_coercivity}
Assume the following: 
\begin{itemize}
\item $F_{\mathcal{L}}>0$, 
\item $(\min_Y \Lambda_Y)(\mathrm{bar}-\chi)-2\rho_H$ 
is in the relative interior of the dual cone of $\mathfrak{a}_s^+$,
\item assumptions (T) and (R) are satisfied.
\end{itemize}
Then the Mabuchi functional is proper modulo the action of 
$Z(L)^0$. 
\end{thm}

\begin{exa}
We have determined in previous examples everything that is necessary 
to check when the criterion apply for the example of the wonderful 
compactification $X$ of, say, the symmetric space of type AIII($2$, $5$).
We consider the ample Cartier divisors $(1+b)Y_1+Y_2$ for 
$0<b<1$ rational, and corresponding uniquely $G$-linearized 
$\mathbb{Q}$-line bundles. They run over all ample divisors on $X$ 
that are trivial on the open orbit, up to rational multiple. 
We illustrate in Figure~\ref{fig_pol_AIII} the corresponding 
polytopes and subdivision by $\tilde{\Delta}^+_{Y_1}$ and 
$\tilde{\Delta}^+_{Y_2}$.
Then it is easy, with computer assistance, to check when the criterion 
is satisfied in terms of $b$, and we obtain bounds $b^-\simeq 0.31$ 
and $b^+\simeq 0.54$ such that when $b^-<b<b^+$, the Mabuchi 
functional (for $\mathcal{L}=\mathcal{O}((1+b)Y_1+Y_2)$ and 
$\Theta=0$) is proper modulo the action of $Z(L)^0$.
\end{exa}

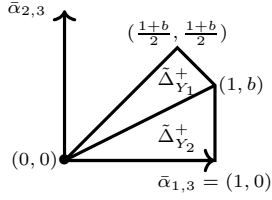
\begin{figure}
\centering
\caption{
Moment polytopes for type AIII($2$, $5$)}
\label{fig_pol_AIII}
\begin{tikzpicture}  
\draw (0,0) node{$\bullet$};
\draw [very thick, ->] (0,0) -- (2,0);
\draw (2,-.3) node{\scriptsize $\bar{\alpha}_{1,3}=(1,0)$};
\draw [very thick, ->] (0,0) -- (0,2);
\draw (-.5,2) node{\scriptsize $\bar{\alpha}_{2,3}$};

\draw [very thick] (0,0) -- (2,0) -- (2,1) -- (3/2,3/2) -- (0,0);
\draw (2.35,1) node{\scriptsize $(1,b)$};
\draw (3/2,3/2+.2) node{\scriptsize $(\frac{1+b}{2},\frac{1+b}{2})$};
\draw (-0.4,0) node{\scriptsize $(0,0)$};
\draw [very thick] (0,0) -- (2,1);
\draw (1.5,0.3) node{\scriptsize $\tilde{\Delta}^+_{Y_2}$};
\draw (1.5,1.05) node{\scriptsize $\tilde{\Delta}^+_{Y_1}$};
\end{tikzpicture}
\end{figure}

\begin{rem}
The two first assumptions imply readily, from the last section, that 
the Mabuchi functional is invariant under the action of 
$Z(L)^0$.
\end{rem}

\begin{rem}
In the case of group compactifications, assumption (R) 
is automatically satisfied, and we may use \cite{LZZ} in the 
later stages of the proof to remove assumption (T). 
\end{rem}

If $\mathcal{L}=K_{X,\Theta}^{-1}$ then 
$\Lambda_Y=1$ for all $Y$, $\bar{S}_{\Theta}=n$, and 
$\chi=\chi^{ac}$ (up to changing the linearization of 
$\mathcal{L}$ by a character of $G$).
Furthermore, 
\[
\chi^{ac}+2\rho_H= \sum_{\alpha\in \Phi_{Q^u}\cup\Phi_s^+}\alpha\circ\mathcal{H}
+\sum_{\alpha\in \Phi_{Q^u}\cup\Phi_s^+}\alpha\circ\mathcal{P}
=\sum_{\alpha\in \Phi_{Q^u}\cup\Phi_s^+}\alpha.
\]
We then have the corollary:

\begin{cor}
\label{cor_log}
If $\mathcal{L}=K_{X,\Theta}^{-1}$ 
then the Mabuchi functional is proper modulo the action 
of $Z(L)^0$ if and only if 
assumptions (T) and (R) are satisfied and the translated barycenter 
$
\mathrm{bar}-\sum_{\alpha\in \Phi_{Q^u}\cup\Phi_s^+}\alpha
$
is in the relative interior of the dual cone of $\mathfrak{a}_s^+$. 
\end{cor}

\begin{rem}
This generalizes the criterion for existence of Kähler-Einstein metrics 
obtained in \cite{DelKSSV} in the sufficient direction. In fact the 
condition is also necessary in this case, as would follow from a computation 
of log-Donaldson-Futaki invariants along special equivariant test configurations 
for example. 
\end{rem}

\begin{rem}
Again, for group compactifications, even assumption (T) 
may be removed by using \cite{LZZ} in the later stages of the 
proof.
\end{rem}

We prove Theorem~\ref{thm_coercivity} in the following subsections 
so unless otherwise stated we make the assumptions as in the statement. 

\subsubsection{Inequality for $\mathrm{Mab}^l_{\Theta}$}

Assume $u^*=\hat{u}^*$ is normalized as in Section~\ref{subsec_normalize}. 
Note the following elementary lemma, following directly from the convexity and 
normalization of $\hat{u}^*$. 

\begin{lem}
\label{lem_poltoboundary}
Assume $u^*$ is normalized, then 
\[
\int_{\Delta'} u^*(p)P_{DH}'(p)dp \leq 
C_{\partial} \int_{\partial \Delta'} u^*P_{DH}'(p)d\sigma
\]
for some constant $C_{\partial}$ independent of $u^*$.
\end{lem}

\begin{prop}
\label{prop_coercivity_linear}
Under the assumptions of Theorem~\ref{thm_coercivity}, 
there exists a positive 
constant $C_l$ such that for any normalized $u^*$, 
\[
M^l(u^*)\geq 
C_l \int_{\partial \Delta'} u^*P_{DH}'(p)d\sigma. 
\]
\end{prop}

\begin{proof}
Suppose there exists a sequence of normalized $u_j^*$ such that 
\[
\int_{\partial \Delta'} u_j^*P_{DH}'(p)d\sigma =1
\] 
and $M^l(u)$ decreases to $0$.
By compactness 
we may assume that $u^*_j$ converges locally uniformly on 
$-2\Delta^t$ to a convex 
function $v_{\infty}$, still satisfying 
$v_{\infty}(2(\chi-\lambda_0))=\min v_{\infty} =0$.

Let $p_0=2\chi-2\mathrm{bar}$. 
By convexity of $u^*$ we have 
$d_pu^*(p-p_0)\geq u^*(p)-u^*(p_0)$, 
hence 
\begin{align*}
\mathrm{M}^l(u^*) \geq  & 
\sum_Y \int_{\tilde{\Delta}_Y'} 
F_{\mathcal{L}}(q)
(u^*(p)-u^*(p_0)-d_{p_0}u^*(p-p_0))P_{DH}'(p)dp \\
& +
\sum_Y \int_{\tilde{\Delta}_Y'} d_pu^*(\Lambda_Yp_0+4\rho_H)P_{DH}'(p)dp
\\ & 
+
\sum_Y \int_{\tilde{\Delta}_Y'} 
F_{\mathcal{L}}(q)
d_{p_0}u^*(p-p_0)P_{DH}'(p)dp
\\
& +
\sum_Y \int_{\tilde{\Delta}_Y'} 
\left( n\Lambda_Y-\bar{S}_{\Theta}
+\sum\frac{(\chi^{ac}-\Lambda_Y\chi)(\alpha^{\vee})}{q(\alpha^{\vee})} \right)
u^*(p_0)P_{DH}'(p)dp
\end{align*}

The last summand vanishes by the expression of $\bar{S}_{\Theta}$ 
obtained in Remark~\ref{rem_barS}. 
The third summand, on the other hand, vanishes by definition of $p_0$.
The second term is non-negative by the assumptions of Theorem~\ref{thm_coercivity}  
and the first term is non-negative by convexity of $u^*$.

Then the fact that $M^l(u_j)$ converges to zero implies that 
$v_{\infty}$ is an affine function by the first term, and 
that its linear part is given by an element of 
$\mathfrak{Y}(T_s/[L,L])$ by the second term. 
Since $v_{\infty}$ is normalized and $2(\chi-\lambda_0)$ is in the interior 
of $\mathfrak{Y}(T_s/Z(L))\cap \Delta'$, 
this means that $v_{\infty}=0$.
As a consequence, we have 
\[
\int_{\Delta'}u_j^*(p)P_{DH}'(p)dp\rightarrow 0.
\]

Let $\delta=\min_Y (1-c_Y)/v_{\mathcal{L}}(\mu_Y)$, it is 
positive by assumption. 
By the expression of $M^l$ given in Remark~\ref{rem_alt_Ml}, 
and using again that $u^*_j$ converges to $0$, we obtain 
that for $j$ large enough, 
\[
M^l(u_j^*)\geq \delta \int_{\partial \Delta'}u^*_jP_{DH}'(p)d\sigma=\delta>0,
\]
which provides a contradiction hence proves the proposition. 
\end{proof}

\begin{rem}
Assumption (R) was not used at all here.
\end{rem}

\subsubsection{Inequality for $\mathrm{Mab}^{nl}_{\Theta}$}

The strategy to transfer the coercivity result on the linear part 
to the full fonctional now follows a general strategy already used by 
Donaldson in \cite{Don02}. 
The first step is to get a rather weak estimate on the non-linear part.

\begin{prop}
\label{prop_weak_Mnl}
There exists uniform positive constants $C_1$, $C_2$, $C_3$, such that 
\[
M^{nl}(u^*)\geq 
-C_1\int_{\Delta'} u^*(p)P_{DH}'(p)dp
-C_2\int_{\partial \Delta'}u^*(p)dp
-C_3
\]
\end{prop}

\begin{proof}
First note that $-I_H(a)-4\rho_H(a)\geq 0$ for $a\in-\mathfrak{a}_s^+$. 
Using this inequality and the convexity of $-\ln\det$, we have 
\[
M^{nl}(u^*)\geq 
-\int_{\Delta'}u^{*,i,j}_{\mathrm{ref}}u^*_{i,j}P_{DH}'dp
-C_2
\]
We now apply the divergence theorem to the vector field 
\[
u^{*,i,j}_{\mathrm{ref},i}u^*P_{DH}'-u^{*,i,j}_{\mathrm{ref}}u^*_iP_{DH}'
+u^{*,i,j}_{\mathrm{ref}}u^*P_{DH,j}'
\]
to obtain 
\begin{align*}
\sum_Y \int_{\Delta'_Y}u^*P_{DH}'d\sigma 
= & 
\int_{\Delta'} 
\Big{(}
u^*u^{*,i,j}_{\mathrm{ref},i,j}P_{DH}'
+u^*2u^{*,i,j}_{\mathrm{ref},j}P_{DH,i}' \\ & \qquad 
+u^*u^{*,i,j}_{\mathrm{ref}}P_{DH,i,j}'
-u^{*,i,j}_{\mathrm{ref}}u^*_{i,j}P_{DH}'
\Big{)}.
dp
\end{align*}
Here we used assumption (R) to check that $P_{DH,j}'$ 
vanishes on restricted Weyl walls. 
Note that $u^{*,i,j}_{\mathrm{ref},i,j}$ is bounded.
Finally, since $u^{*,i,j}_{\mathrm{ref},j}P_{DH,i}'$ 
and $u^{*,i,j}_{\mathrm{ref}}P_{DH,i,j}'$ are bounded, and 
thanks to assumption (T), we may apply \cite[Lemma~4.6]{LZZ} 
to obtain the statement (Even though \cite[Lemma~4.6]{LZZ} is proved for the particular case 
of $P_{DH}'$ coming from a group compactification, it applies 
much more generally for a positive measure with a product of 
linear functions as density and polytope contained in a chamber 
defined by these linear functions). 
\end{proof}

\subsubsection{Conclusion of the proof}

\begin{proof}[Proof of Theorem~\ref{thm_coercivity}]
Let $0<\epsilon<1$. 
We have 
$M^{nl}(u^*)\geq M^{nl}(\epsilon u) - C_4$ 
for some constant $C_4$,
with the obvious definition for $M^{nl}(\epsilon u)$ (it is not 
important here that $\epsilon u^*$ does not appear as the 
convex conjugate of the toric potential of a smooth metric).
Furthermore, the proof of Proposition~\ref{prop_weak_Mnl} applies just 
as well to $\epsilon u^*$ and we obtain 
\begin{align*}
M(u^*)& \geq 
M^l(u^*)-\epsilon(C_1\int_{\Delta'} u^*(p)P_{DH}'(p)dp
+C_2\int_{\partial \Delta'}u^*(p)dp
+C_3)-C_4 \\
& \geq \epsilon \int_{\Delta'}u^*P_{DH}' 
+M^l(u^*) 
-\epsilon\Big{(}(C_1+1)\int_{\Delta'}u^*P_{DH}'\\ &  \qquad 
+
C_2\int_{\partial \Delta'}u^*(p)dp+C_3\Big{)}-C_4 \\
& \geq \epsilon \int_{\Delta'}u^*P_{DH}'
+M^l(u^*)-\epsilon((C_1+1)C_{\partial}C_l+C_2C_l)M^l(u^*)
\\ &  \qquad  -\epsilon C_3-C_4 \\
& \geq \epsilon \int_{\Delta'}u^*P_{DH}'
-\epsilon C_3-C_4, 
\end{align*}
by choosing $\epsilon=((C_1+1)C_{\partial}C_l+C_2C_l)^{-1}$.

This is enough to conclude: let $\phi\in \mathrm{rPSH}^K(X,\omega_\mathrm{ref})$, 
let $\hat{\phi}=g\cdot \phi + C$ be the normalization of $\phi$, obtained 
\emph{via} the action of some $g\in Z(L)^0$ and addition of a constant, 
then since the assumptions imply that the Mabuchi functional is 
invariant under the action of $Z(L)^0$, we have 
$\mathrm{Mab}(\phi)=\mathrm{Mab}(\hat{\phi})$. 
By Proposition~\ref{prop_J_functional}, 
since the toric potential of a normalized $\hat{\phi}$ satisfies 
$u(0)=0$, we have 
$\int_{\Delta'}\hat{u}^*P_{DH}'\geq \frac{\mathcal{L}^n}{n!}J(\hat{\phi})-C_5$ 
for some constant $C_5$ independent of $\phi$.
Hence we have 
\begin{align*}
\mathrm{Mab}_{\Theta}(\phi) & = \frac{n!}{\mathcal{L}^n}M^l(\hat{u}^*) \\
& \geq \epsilon \frac{n!}{\mathcal{L}^n}\int_{\Delta'}\hat{u}^*P_{DH}' -\frac{n!}{\mathcal{L}^n}(\epsilon C_3+C_4) \\
& \geq \epsilon J(\hat{\phi})-\frac{n!}{\mathcal{L}^n}(C_5+\epsilon C_3+C_4)
\end{align*}
\end{proof}

\bibliographystyle{alpha}
\bibliography{biblio}
\end{document}